\documentclass[11pt,a4paper]{article}
\usepackage{amsmath,amsfonts,amssymb,amsthm,amsbsy,mathrsfs}
\usepackage{appendix}
\usepackage{enumitem}
\usepackage[colorlinks,linkcolor=blue,anchorcolor=blue,citecolor=blue]{hyperref}
\usepackage{color,xcolor}
\usepackage{authblk}
\usepackage{bbm}
\usepackage{algorithm,algorithmic}
\usepackage{breakcites}
\usepackage[left=1in,right=1in,top=1.2in,bottom=1.2in]{geometry}
\numberwithin{equation}{section}
\newtheorem{definition}{Definition}[section]
\newtheorem{theorem}{Theorem}[section]
\newtheorem{lemma}{Lemma}[section]
\newtheorem{remark}{Remark}[section]

\makeatletter
\newenvironment{breakablealgorithm}
  {
    \begin{center}
      \refstepcounter{algorithm}
      \hrule height.8pt depth0pt \kern2pt
      \parskip 0pt
      \renewcommand{\caption}[2][\relax]{
        {\raggedright\textbf{\fname@algorithm~\thealgorithm} ##2\par}%
        \ifx\relax##1\relax 
          \addcontentsline{loa}{algorithm}{\protect\numberline{\thealgorithm}##2}%
        \else 
          \addcontentsline{loa}{algorithm}{\protect\numberline{\thealgorithm}##1}%
        \fi
        \kern2pt\hrule\kern2pt
     }
  }
  {
     \kern2pt\hrule\relax
   \end{center}
  }
\makeatother

\title{Robust mean estimation under star-shaped constraints with heavy-tailed noise}
\author[1]{Tuorui Peng}
\author[2]{Akshay Prasadan}
\author[1]{Matey Neykov}
\affil[1]{Department of Statistics and Data Science, Northwestern University}
\affil[2]{Department of Statistics and Actuarial Science, Simon Fraser University}

\begin{document}

\maketitle

\begin{abstract}
    We study the problem of robust mean estimation with adversarially contaminated data under star-shaped constraints in a heavy-tailed noise setting, where only a finite second moment $ \sigma ^2 $ is assumed. 
    For a contamination level $ \varepsilon$ below some constant, we show that the minimax rate of the squared $ \ell_2 $ loss is $ \max( \delta ^{*2}, \varepsilon \sigma ^2) \wedge d^2 $ for a star-shaped set with diameter $ d $ (set $d = \infty$ if the set is unbounded), with $ \delta ^* $ determined via the local entropy $ \log M^\mathrm{ loc }(\delta ,c) $ as
    \begin{align*}
        \delta ^*:= \sup\bigg\{\delta \geq 0: N\frac{\delta ^2}{\sigma ^2}\leq \log M^\mathrm{ loc }(\delta ,c) \bigg\},  
    \end{align*}
    where $ c $ is a sufficiently large constant. Crucially, we require that the sample size satisfies $N \gtrsim \mathop{ \sup }\limits_{\delta \geq 0}  \log M^\mathrm{ loc }(\delta ,c)$. We also show that the minimax rate is $ \max(\delta^{*2},\varepsilon ^2\sigma ^2) \wedge d^2 $ for known or sign-symmetric distributions, matching the rate achieved in the Gaussian case.
\end{abstract}

\tableofcontents

\bibliographystyle{apalike}
\newcommand{\x}{\frac{\delta ^2}{\sigma ^2}}
\newcommand{\Y}{\mathcal{Y}}

\section{Introduction}

\subsection{Problem Setting}

We first give a formal setup of the model, which is known as the adversarial contamination model, or the strong contamination model \cite{diakonikolas_being_2018,diakonikolas_robust_2019}. 
\begin{definition}\label{definition:model_setup}
    The data $ X $ is drawn from the following process: First, the original uncorrupted data $ \tilde{ X } = \{\tilde{ X_i }\}_{i=1}^N   $ of sample size $ N $ and dimension $ n $ is drawn from a location model $ \tilde{X}_i= \mu +\xi _i $, with the noise term $ \xi _i \mathop{ \sim }\limits^{i.i.d.} \xi $. Then, some fixed proportion $ \varepsilon $ (strictly less than some constant $ <1/2 $) of the data is corrupted by an adversarial algorithm $ \mathcal{C} $, which has oracle knowledge of the data generation model and the estimation method. The algorithm is allowed to examine, and then contaminate, the original data $ \tilde{X} $ to arrive at the observed data $ X $.
\end{definition}

We investigate this problem across multiple settings that vary in their constraints on the mean vector $ \mu  $ and assumptions on the noise distribution $ \xi $. Specifically, we constrain $ \mu  $ to belong to a known star-shaped set $ K\subseteq \mathbbm{R}^n $, which generalizes convexity. We consider both bounded and unbounded cases for $ K $. For the noise distribution, we assume $ \xi $ is heavy-tailed with only finite second moments. More precisely, we write $ \xi \in \Xi_{\sigma ^2} $ where the distribution family $ \Xi_{\sigma ^2} $ has zero mean and finite leading eigenvalue $\lambda_1$ of its covariance matrix: $ \lambda _1 = \sigma ^2 <\infty $. In the case of a known upper bound, we slightly abuse the notation by writing $\lambda _1\leq \sigma ^2 $. We also analyze the special case of known or (sign-)symmetric distributions, denoted by $ \Xi_{\sigma ^2}^\mathrm{ sym }  $.

Our objective is to characterize the minimax rate for estimating the true mean $ \mu  $. The minimax rate with respect to squared $\ell_2$ loss is defined as:
\begin{align*}
    \mathfrak{M}=\mathfrak{M}(\mu ;K, \Xi_{\sigma^2}, \ell_2):= \mathop{ \inf }\limits_{\hat{\mu }} \mathop{ \sup }\limits_{\mu \in K} \mathop{ \sup }\limits_{\xi \in \Xi_{\sigma ^2}}  \mathop{ \sup }\limits_{\mathcal{C}} \mathbbm{E}_{ \mu  }\left[ \left\Vert \hat{\mu }(\mathcal{C}(\tilde{X})) - \mu  \right\Vert ^2 \right] .
\end{align*} 
For the symmetric distribution case, we denote the corresponding minimax rate by $ \mathfrak{M}_\mathrm{sym} $ over the family $ \Xi_{\sigma ^2}^\mathrm{ sym }  $. 

\subsection{Related Work}\label{subsection:related_work}

The foundations of robust statistics were established in the 1960s through the pioneering work of \cite{tukey_survey_1959} and \cite{huber_robust_1964}. Since then, a rich literature has developed around the theory of robust statistics (see \cite{huber_robust_2011} for a comprehensive historical overview). Recently, there has been renewed interest in robust statistics from both the statistics and theoretical computer science communities \cite{lugosi_sub-gaussian_2019,diakonikolas_being_2018}, with a particular focus on robust estimation under adversarial contamination, in contrast to Huber's classical random contamination framework. We now review the literature most relevant to our work. 
    
\cite{diakonikolas_being_2018,diakonikolas_robust_2019} addressed the existence of an efficient algorithm for robust mean estimation in the setting of sub-Gaussian distribution with known covariance. The error rate was dimension-free and optimal up to logarithmic factors.  \cite{lugosi_robust_2019} showed that the trimmed mean estimator can obtain the optimal rate for adversarial contamination models under only finite second moment assumption, though lacking computational efficiency. \cite{diakonikolas_outlier_2021} studied efficient algorithms that achieved the optimal rate in an adversarial contamination model, but with a requirement of known covariance. \cite{diakonikolas_outlier-robust_2022} further developed the algorithm for sparse mean estimation under heavy-tailed setting. However, their result requires mild knowledge of the fourth moment. \cite{bhatt_minimax_2022} based their method on Catoni's estimator and allowed even heavier-tailed distributions. A lower bound of the minimax rate was presented. \cite{depersin_robust_2022} developed an algorithm with better efficiency and allowed heavy-tailed settings with just bounded second moment. We remark that all the above results are obtained as high-probability bounds. 
    
There are other closely related works with slightly different focuses or perspectives. \cite{minsker_uniform_2025} considered the contamination model as an application of their uniform bound on estimation problems and recovered the optimal rate. \cite{novikov_robust_2023} studied the robust mean estimation problem for some special classes of symmetric distributions under heavy-tailed scenarios and showed that Gaussian rates can be achieved in these cases. \cite{oliveira_finite-sample_2025} proved the minimax property of the trimmed mean estimator under even weaker moment assumptions. However, their rate involving $\varepsilon$ is dimension-dependent and thus sub-optimal. The recent work of \cite{neykov_polynomial-time_2025} proposed a polynomial-time algorithm which produced near-optimal rates for multiple estimation tasks including robust mean estimation, constrained on a special class of Type-2 convex body $K$. For a survey of more topics in robust statistics and contamination model, we refer readers to \cite{diakonikolas_algorithmic_2023,lugosi_mean_2019}. 

We also mention some works concerning the trimmed mean estimator, since the trimmed mean estimator plays a crucial role in our methodology. Asymptotic properties of trimmed mean estimators have been extensively studied throughout the last century \cite{stigler_asymptotic_1973,jaeckel_robust_1971,hall_large_1981}. Most recently, non-asymptotic results have emerged, beginning with \cite{oliveira_sub-gaussian_2019} who established the sub-Gaussian performance of trimmed mean estimators under bounded second moment assumptions. Subsequently, \cite{lugosi_robust_2019} proved the statistical optimality of trimmed mean estimators in adversarial contamination models and proposed its multivariate extensions. \cite{rico_optimal_2022} studied the performance of the trimmed mean estimator for the mean of a function. A later work by \cite{oliveira_trimmed_2025} improved the result to the uniform error bound over function classes. \cite{oliveira_finite-sample_2025} studied the tail property of the trimmed mean estimator under various moment assumptions. Two highlights of their work are the sub-Gaussian rate with sharp constants under a finite variance assumption, and the minimax property of the trimmed mean estimator under weaker moment conditions and adversarial contamination.
    
Our work extends \cite{neykov_minimax_2022,prasadan_information_2024} to the more general heavy-tailed setting where we assume only finite second moments. This natural relaxation, combined with star-shaped constraints, encompasses a significantly broader class of distributions and constraint sets. Importantly, we require only the knowledge of the leading eigenvalue (or an upper bound thereof) of the covariance matrix, rather than full knowledge of the covariance structure as in \cite{diakonikolas_being_2018,diakonikolas_outlier_2021}. Another interesting aspect of our work is that the minimax risk can be improved to match the rate in Gaussian settings. Previous work \cite{novikov_robust_2023} showed that this is possible for some special classes of symmetric distributions, while we  extend this result to any (sign-)symmetric distribution. We then present the results when applying to the well-studied cases of $ \ell_0 $- and $ \ell_1 $-ball constraints. Notably, our $ \ell_0 $ result differs from \cite{comminges_adaptive_2021}, who considered the sample size $ N=1 $ case (while we assume a sufficiently large sample size); our $ \ell_1 $ result aligns with existing work \cite{rigollet_high-dimensional_2023,prasadan_facts_2025}. To generalize the results of \cite{prasadan_information_2024}, we employ new techniques to address the challenges posed by heavy-tailed noise distributions. A different estimator is also utilized to characterize the symmetry of the distribution in the corresponding case. Our results shed light on the interplay between adversarial contamination, heavy-tailed noise, and structural constraints. They also provide a theoretical foundation for future research on robust learning under mild distributional assumptions. 

\subsection{Organization of the Paper}

The remainder of this paper is organized as follows. \autoref{section:lower_bounds} establishes lower bounds, which decompose into two components imposed by information-theoretic limits and adversarial contamination, respectively. \autoref{section:framework_for_upper_bound} presents our framework for the upper bounds, introducing the directed tree construction and tournament series algorithm that operates on this tree structure. Using the framework, we derive the bounds in different scenarios and obtain the minimax rates for bounded constraints, unbounded constraints, and symmetric noise distribution cases in \autoref{section:upper_bound_bounded_case}, \ref{section:unbounded_case}, and \ref{section:symmetric_case}, respectively. Then in \autoref{section:examples} we demonstrate the applicability of our results to several canonical models. Finally, \autoref{section:discussion} concludes with a discussion of our findings and directions for future research. 

\subsection{Notations and Definitions}\label{subsection:notation_and_definitions}

For the reader's convenience, \autoref{table:notation} summarizes the notations and definitions used throughout this paper. Subsequently, we define two fundamental concepts for our analysis: the \textit{local metric entropy} of a set $ K $ (with respect to the $ \ell_2 $ norm) and \textit{star-shaped sets}. We also list some key properties of these concepts that will be instrumental in our main results.

\begin{table}[htbp]
    \centering
    \begin{tabular}{|c|c|}
        \hline
        Notation & Description \\ \hline
        $ a\vee b $   & $ \max \{a,b\} $ \\
        $ a\wedge b $ & $ \min \{a,b\} $ \\
        $ [N] $ & $ \{1,2,\ldots,N\} $ for any $ N\in \mathbbm{N} $ \\
        $ \left\Vert \, \cdot \, \right\Vert    $ & the standard $ \ell_2 $ norm in $ \mathbbm{R}^n $   \\
        $ B(\nu ,r) $ & the Euclidean ball of radius $ r $ centered at $ \nu  $: $ \{\mu \in \mathbbm{R}^n : \left\Vert \mu -\nu  \right\Vert \leq r \} $ \\
        $ \mathbbm{1}(\, \cdot \, ) $& indicator function \\
        $\mathcal{N}(\mu, \Sigma)$ & normal distribution with mean $\mu$ and covariance matrix $\Sigma$\\
        $\mathrm{Bin}(n,p)$ & binomial distribution with parameters $ n $ and $ p $ \\
        $\mathbbm{S}^{n-1}$ & the unit $n-1$ dimensional sphere embedded in $\mathbbm{R}^n$\\
        $ w^\top $ & the transpose of a vector $ w $ \\
        $ \Xi_{\sigma ^2} $ & the distribution family with mean $ 0 $ and covariance matrix $ \Sigma \preceq \sigma ^2 I $ \\
        $ \bar{K} $ & the closure of the set $ K $ \\
        $ \mathrm{Diam}(K) $ & the diameter of the set $ K $: $ \sup_{x,y\in K}\left\Vert x-y \right\Vert  $ \\
        \hline
    \end{tabular}
    \caption{Notations and Definitions}
    \label{table:notation}
\end{table}

\begin{definition}[Local Metric Entropy]\label{definition:local_metric_entropy}
    Given a set $ K\subseteq \mathbbm{R}^n $, the $ \eta $-packing number of $ K $ is defined as the maximum cardinality $ M=M(\eta,K) $ of a set $ \{\nu _1,\ldots,\nu _M\}\subset K $ such that $ \left\Vert \nu _i-\nu _j \right\Vert >\eta $ for all $ i\neq j $.

    Then given a fixed constant $ c>0 $, define the maximal local packing of $ K $ as:
    \begin{align*}
        M^{ \mathrm{loc} }_K(\eta ,c) := \mathop{ \sup }\limits_{\nu \in K}M\big(\dfrac{ \eta }{ c } ,B(\nu ,\eta)\cap K\big),
    \end{align*}
    where we may omit the subscript $ K $ when the context is clear, writing simply $ M^{ \mathrm{loc} }(\eta ,c) $ instead. The corresponding local metric entropy is then defined as $ \log M^{ \mathrm{loc} }(\eta ,c) $.
\end{definition}

\begin{definition}[Star-shaped set]\label{definition:star_shaped_set}
    A set $ K\subseteq \mathbbm{R}^n $ is called star-shaped if there exists a point $ k^*\in K $ such that for every $ k\in K $, the entire line segment connecting $ k^* $ and $ k $ lies within $ K $. Formally, this means $ (1-t)k^*+tk\in K $ for all $ t\in [0,1] $ and all $ k \in K $. Any such point $ k^* $ is called a center of the star-shaped set $ K $.
\end{definition}

\begin{remark}
    The star-shaped property is a natural generalization of convexity. Every convex set is star-shaped (with any point in the set serving as a center), but the converse does not hold, making star-shaped sets a substantially richer class of constraint sets. As an example, for studying the problem of sparse estimation, the constraint is an $\ell_0$-ball, which is star-shaped but non-convex.
\end{remark} 

Two useful properties of a star-shaped set are presented here. 
We refer readers to the preceding work \cite{neykov_minimax_2022, prasadan_information_2024} for detailed proofs of these claims.

\begin{lemma}\label{lemma:long_enough_segment_in_K}
    Every bounded star-shaped set $K$ contains line segments of length at least $d/3$, where $d = \mathrm{Diam}(K)$ denotes the diameter of $K$.
\end{lemma}

\begin{lemma}\label{lemma:monotonicity_of_local_entropy}
    For any star-shaped set $ K\subseteq \mathbbm{R}^n $ and constant $ c>0 $, the local metric entropy $ \log M^{ \mathrm{loc} }(\eta ,c) $ is non-increasing in $ \eta $.
\end{lemma}

\section{Lower Bounds}\label{section:lower_bounds}

In this section, we establish fundamental lower bounds on the minimax rate for robust mean estimation under constraints with heavy-tailed noise. Our analysis encompasses two aspects: the information-theoretic lower bound and the contamination-induced lower bound. We study both the rate for general finite-variance distributions and the special case of symmetric distributions.

\paragraph{Information Theoretic Lower Bound.} The following lemma is based on Fano's inequality, providing a lower bound on the minimax risk in terms of the local metric entropy.
\begin{lemma}[Information-Theoretic Lower Bound]\label{lemma:packing_lower_bound}
    Let $ c $ be the constant from \autoref{definition:local_metric_entropy}, which is fixed sufficiently large. Then for any $ \delta  $ satisfying $ \log M^\mathrm{ loc }(\delta,c) > 4\big( \frac{N\delta^2}{2\sigma^2}\vee \log 2\big)$, we have
    \begin{align*}
        \mathfrak{M} \geq  \mathop{ \inf }\limits_{\hat{\mu }} \mathop{ \sup }\limits_{\mu \in K} \mathbbm{E}_{ \tilde{X}\sim \mathcal{N}(\mu ,\sigma ^2I)  }\left\Vert \hat{\mu }(\tilde{X})-\mu  \right\Vert ^2 \geq \dfrac{ \delta ^2 }{ 8c^2 }  .
    \end{align*}
\end{lemma} 
The proof follows directly from \cite[Lemma 2.1]{prasadan_information_2024}, which applies to Gaussian distributions, combined with the observation that $\mathcal{N}(0, \sigma^2I) \in \Xi_{\sigma^2}$.

\paragraph{Contamination Lower Bound.} The following two lemmas establish lower bounds arising from adversarial contamination. 

\begin{lemma}[Contamination Lower Bound]\label{lemma:diakonikolas_lower_bound}

    Let $ K $ be a bounded star-shaped set with diameter $ d $. For any $ \varepsilon \in (0,1/2) $, it holds that
    \begin{align*}
        \mathfrak{M} \gtrsim  \varepsilon \sigma ^2 \wedge d^2.
    \end{align*}
\end{lemma}

\begin{lemma}[Symmetric Case Contamination Lower Bound]\label{lemma:gaussian_contamination_lower_bound}
    Let $ K $ be a bounded star-shaped set with diameter $ d $. For any $\varepsilon\in (0,1/2)$, it holds that
    \begin{align*}
         \mathfrak{M}_\mathrm{sym} \gtrsim \varepsilon^2 \sigma ^2 \wedge d^2.
    \end{align*}
\end{lemma}

We prove \autoref{lemma:diakonikolas_lower_bound} by constructing a specific distribution and a corresponding contamination scheme. The detailed proof is provided in \autoref{section:proof_lower_bound}. The proof of \autoref{lemma:gaussian_contamination_lower_bound} follows from \cite[Lemma 2.2]{prasadan_information_2024}, which applies to Gaussian distributions, and noticing the fact that $\mathcal{N}(0, \sigma^2I) \in \Xi_{\sigma^2}^\mathrm{ sym }$.

\paragraph{Unbounded Case.} We now present corresponding results for the unbounded case, which follow as direct corollaries of the bounded case analysis.

\begin{lemma}\label{lemma:diakonikolas_lower_bound_unbounded}
    Let $ K\subseteq \mathbbm{R}^n $ be a star-shaped set. Then for any $ \varepsilon \in (0,1/2) $, it holds that
    \begin{align*}
        \mathfrak{M} \gtrsim \varepsilon \sigma ^2 .
    \end{align*}
\end{lemma}
\begin{lemma}\label{lemma:gaussian_contamination_lower_bound_unbounded}
    Let $ K\subseteq \mathbbm{R}^n $ be a star-shaped set. Then for any $ \varepsilon \in (0,1/2) $, it holds that
    \begin{align*}
         \mathfrak{M}_\mathrm{sym} \gtrsim \varepsilon^2 \sigma ^2 .
    \end{align*}
\end{lemma} 

The proof of \autoref{lemma:diakonikolas_lower_bound_unbounded} follows the same approach as \autoref{lemma:diakonikolas_lower_bound}, with the following modification: In the proof of \autoref{lemma:diakonikolas_lower_bound}, we adopt the convention $ \left\Vert \Delta  \right\Vert \asymp \sigma  $ in place of $\left\Vert \Delta  \right\Vert \asymp \sigma  \wedge d/\sqrt{\varepsilon}$, since an unbounded $ K $ contains segments of arbitrary length. Then the result follows accordingly. Readers can refer to \autoref{section:proof_lower_bound} for the detailed proof. The proof of \autoref{lemma:gaussian_contamination_lower_bound_unbounded} proceeds similarly with the same adjustment.

\section{Framework of Robust Algorithm for Upper Bound}\label{section:framework_for_upper_bound}

Our approach for establishing information-theoretic upper bounds employs a tournament-on-the-tree algorithm originally introduced in \cite{neykov_minimax_2022} and subsequently refined in \cite{prasadan_information_2024}. Two key constructions involved are the directed tree and the tournament, presented in \autoref{subsection:directed_tree_construction_and_properties} and \autoref{subsection:tournament_based_on_robust_comparison}, where some of their important properties are given. Based on the two components, in \autoref{subsection:robust_algorithm_for_bounded_case} and \autoref{subsection:robust_algorithm_for_unbounded_case} we discuss the algorithmic framework under bounded or unbounded constraints, respectively. The output of the algorithms would be a chain along the tree. As we shall see in \autoref{section:upper_bound_bounded_case}-\ref{section:symmetric_case}, with a properly chosen robust comparison estimator to be integrated into the algorithm, we are able to obtain a risk upper bound by taking sufficiently many steps on the chain. Finally, by combining the upper bound with its lower counterpart, we have the desired minimax rate.

\subsection{Directed Tree Construction and Properties}\label{subsection:directed_tree_construction_and_properties}

We begin by reproducing the construction from \cite{prasadan_information_2024} in \autoref{algorithm:tree_construction}, which gives a directed pruned tree $ \mathcal{G} $ rooted at a specified point $ s $. We establish the following notation: For any node $ u \in \mathcal{G} $, we denote its parent as $ \mathcal{P}(u) $, where there exists a directed edge $ \mathcal{P}(u) \to u $. The offspring of $ u $ is denoted $ \mathcal{O}(u) $, representing the set of all nodes with directed edges originating from $ u $, formally $ \mathcal{O}(u) = \{v\in \mathcal{G}:\, \mathcal{P}(v)=u\} $. The layer (or depth) of the directed tree is indexed by $ J $.

\begin{breakablealgorithm}
    \renewcommand{\algorithmicrequire}{\textbf{Input:}}
    \renewcommand{\algorithmicensure}{\textbf{Output:}}
    \caption{Directed Tree Construction}
    \label{algorithm:tree_construction}
   \begin{algorithmic}[1]
        \REQUIRE Set $ K $, Root node $ s\in K $, diameter $ d>0 $, constant $ c $ chosen large enough;

        \ENSURE A directed tree structure $ \mathcal{G} $ where each layer $ J $ is denoted $ \mathcal{L}(J) $;

        \STATE $ \mathcal{L}(1)\leftarrow \{s\} $;

        \STATE Form a maximal $ (d/c) $-packing set of $ B(s,d)\cap K $, then add directed edges from $ s $ to each node in the packing set, i.e., the packing set is the offspring of $ s $, denoted $ \mathcal{O}(s) $;

        \STATE $ \mathcal{L}(2)\leftarrow \mathcal{O}(s)   $;

        \FOR{$ J = 3,4,\ldots $}

            \FORALL{$ u \in \mathcal{L}(J-1)  $}

                \STATE Form a maximal $ (d/2^{J-1}c) $-packing set of $ B(u, d/2^{J-1}) \cap K $, then add directed edges from $ u $ to these nodes, denoted as $ \mathcal{O}(u) $;

            \ENDFOR

            \STATE $ \mathcal{U}_J \leftarrow $ lexicographically ordered set of all offspring nodes $ \mathcal{O}(u) $ for all $ u\in \mathcal{L}(J-1) $, as constructed in the previous step;

            \WHILE{$ \mathcal{U}_J \neq \varnothing $}

                \STATE Pick first element, say $ u_l^J $, of $ \mathcal{U}_J $;

                \STATE $ \mathcal{T}_J(u_l^J) \leftarrow \{u_k^J\in \mathcal{U}_J:\, \left\Vert u_k^J-u_l^J \right\Vert \leq \frac{d}{2^{J-1}c},\, k\neq l \} $;

                \STATE For each $ u_k^J \in \mathcal{T}_J(u_l^J) $, remove directed edge $ \mathcal{P}(u_k^J) \to u_k^J $ and the node $ u_k^J $ from $ \mathcal{G} $, add edge $ \mathcal{P}(u_k^J) \to u_l^J $;

                \STATE Remove $ \{u_l^J\} \cup \mathcal{T}_J(u_l^J) $ from $ \mathcal{U}_J $;

            \ENDWHILE

            \STATE $ \mathcal{L}(k)\leftarrow \bigcup_{u\in \mathcal{L}(J-1)} \mathcal{O}(u) $;            

        \ENDFOR

        \RETURN $ \mathcal{G} $.
    \end{algorithmic} 
\end{breakablealgorithm}

We now present essential properties of the tree structure $ \mathcal{G} $, as established in \cite[Lemma 3.1, 3.2, 3.3]{prasadan_information_2024}:

\begin{lemma}\label{lemma:tree_properties}
    Let $ \mathcal{G} $ be the pruned tree constructed above in \autoref{algorithm:tree_construction}. Then the following properties hold:
    \begin{itemize}[topsep=2pt,itemsep=0pt]
        \item For any $ J\geq 3 $, $ \mathcal{L}(J) $ is a $ \frac{d}{2^{J-1}c} $-packing, and also a $ \frac{d}{2^{J-2}c} $-covering of $ K $;
        \item For any $ J\leq 2 $ and any parent node $ \Y_{J-1}\in \mathcal{L}(J-1) $, its offspring $ \mathcal{O}(\Y_{J-1}) $ form a $ \frac{d}{2^{J-2}c} $-covering set of $ B(\Y_{J-1},\frac{d}{2^{J-2}}) \cap K $. Moreover, its cardinality satisfies $ \mathrm{ card }(\mathcal{O}(\Y_{J-1})) \leq M^\mathrm{ loc }(\frac{d}{2^{J-2}},2c) $;
        \item For any $\mu\in K$ and $ J\geq 2 $, we have
        \begin{align*}
            \mathrm{ card }\big( \mathcal{L}(J-1)\cap B(\mu ,\frac{d}{2^{J-2}}) \big) \leq   M^\mathrm{ loc }(\frac{d}{2^{J-2}},c) \leq M^\mathrm{ loc }(\frac{d}{2^{J-2}},2c);
        \end{align*}
        \item Let $ [\Y_1,\Y_2,\ldots] $ be a winner chain in $ \mathcal{G} $, i.e., $ \Y_{J+1}\in \mathcal{O}(\Y_J) $ for all $ J\in \mathcal{N}^+ $. Then we have
        \begin{align*}
            \left\Vert \Y_{J'} - \Y_J \right\Vert \leq \frac{d(2+4c)}{c2^{J'}},\quad J\geq J'\geq 1. 
        \end{align*}
    \end{itemize}
\end{lemma}

\subsection{Tournament Based on Robust Comparison}\label{subsection:tournament_based_on_robust_comparison}

Another central technique in establishing the upper bounds involves conducting a tournament among candidate points that comprise the local packing set at each layer. The objective is to identify which candidate best represents the observed data. This boils down to a hypothesis testing problem within a pair of separated points, say $ \nu _1 $ and $ \nu _2 $, to determine which one is more representative. We now introduce the notation for such robust comparison and the tournament mechanism built upon it. 

\paragraph{Robust Comparison.} To determine which of two points better represents the possibly corrupted data $X$, we formalize the hypothesis testing problem: Given an ordered pair $(\nu_1,\nu_2)$ of points $ \nu _1,\nu _2\in \mathbbm{R}^n $ satisfying $ \left\Vert \nu _1-\nu _2 \right\Vert \geq C\delta  $ for some constant $ C>2 $, we seek to test whether the true mean $ \mu $ of $ \tilde{X} $ belongs to either ball $ B(\nu _1,\delta) $ or $ B(\nu _2,\delta) $. The testing function is defined as
\begin{align*}
    \psi=\psi_{(\nu _1,\nu _2)}(X) = \begin{cases}
        0,& \text{if deciding } \mu \in B(\nu _1,\delta),\\
        1,& \text{if deciding } \mu \in B(\nu _2,\delta).
    \end{cases}
\end{align*}

We further define the notation $\nu_1 \succ \nu_2$ when $\psi_{(\nu_1,\nu_2)} = 0$ and conversely $\nu_1 \prec \nu_2 $ when $\psi_{(\nu_1,\nu_2)} = 1$. This notation provides a convenient framework for expressing which candidate is favored by the data.

The specific estimator leading to the testing function $ \psi $ depends on the model settings, for example, the assumptions on $ K $ and the distribution family. We discuss these estimators in subsequent sections: \autoref{subsection:robust_estimator} addresses general finite-variance distributions, while \autoref{section:symmetric_case} treats symmetric or known distributions. 

\paragraph{Tournament.} Based on a robust comparison method, we can conduct a tournament over a set of candidates $\mathcal{M}$. This tournament concept, originating from \cite{lecam_convergence_1973} and \cite{birge_approximation_1980}, is formalized in the following definition.
\begin{definition}[Tournament]\label{definition:tournament_T}
    Let $ \mathcal{M} = \{\nu _1,\ldots,\nu _M\} \subset K $. For any $\delta$, $ C>2 $ and $ \nu \in \mathcal{M} $, define $ \tilde{\mathcal{M}}(\nu):=\{ \nu ' \in \mathcal{M}:\, \nu ' \succ \nu  \text{ and } \left\Vert \nu -\nu ' \right\Vert \geq C\delta \}  $ and let
    \begin{align*}
        T(\delta ,\nu ,\mathcal{M}) & := \mathbbm{1}(\tilde{\mathcal{M}}(\nu)\neq \varnothing)\cdot \mathop{ \max }\limits_{\nu '\in \tilde{\mathcal{M}}(\nu)} \left\Vert \nu -\nu ' \right\Vert   .
    \end{align*}
    Then we define the tournament winner as any point $ \nu _i\in \mathcal{M} $ that minimizes $ T(\delta ,\nu _i, \mathcal{M}) $.
\end{definition}

\textbf{Comment: }Intuitively, the minimizer $ \nu_i\in \mathcal{M} $ of $ T(\delta ,\nu _i, \mathcal{M}) $ is the one with minimal worst-case distance to its contenders in the tournament. 

\subsection{Algorithm for Bounded Case}\label{subsection:robust_algorithm_for_bounded_case}
We now present the robust algorithm for bounded sets $ K $, which outputs a winner chain $ \Y = [\Y_1,\Y_2,\ldots] $ within the directed tree $ \mathcal{G} $ constructed in \autoref{subsection:directed_tree_construction_and_properties}. The algorithm applies the tournament from \autoref{definition:tournament_T} at each layer of the tree, traversing it from root to leaves. 

Recall that for each node $ \Y_{J-1}\in \mathcal{L}(J-1) $, its offspring $ \mathcal{O}(\Y_{J-1}) $ constitutes a local packing set of $ B(\Y_{J-1},d/2^{J-2})\cap K $ with packing radius $ d/(2^{J-2}c) $, which decreases geometrically as $ J $ increases. Beginning with $ \Y_1:=s $ as the tree root, at each layer $ J $ we let $ \Y_{J+1} $ be the tournament winner from \autoref{definition:tournament_T} among the offspring $ \mathcal{O}(\Y_J) $. This process yields a tournament winner chain $ \Y = [\Y_1,\Y_2,\ldots] $ where each node $ \Y_J $ optimally represents the data within its local neighborhood. The complete procedure is detailed in \autoref{algorithm:robust_algorithm}.
\begin{breakablealgorithm}
    \renewcommand{\algorithmicrequire}{\textbf{Input:}}
    \renewcommand{\algorithmicensure}{\textbf{Output:}}
    \caption{Robust Tournament Algorithm for an Upper Bound.}
    \label{algorithm:robust_algorithm}

    \begin{algorithmic}[1]
        \REQUIRE Directed tree $ \mathcal{G} $;

        \ENSURE A winner chain $ \Y=[\Y_1,\Y_2,\ldots] $ in $ \mathcal{G} $;

        \STATE Initialize $ J=1 $ and $ \Y = [\Y_1] $;

        \FOR{$ J=1,2,\ldots $}

            \STATE $ \Y_{J+1}\leftarrow \mathop{ \arg\min }\limits_{\nu \in \mathcal{O}(\Y_J)} T\big( \frac{d}{2^k(C+1)}, \,\nu ,\, \mathcal{O}(\Y_J) \big) $, breaking ties lexicographically;

            \STATE $ \Y.\text{append}(\Y_{J+1}) $;

        \ENDFOR

        \RETURN $ \Y=[\Y_1,\Y_2,\ldots] $.
    \end{algorithmic}
\end{breakablealgorithm}
As we will see in later sections, traversing sufficiently many steps along the chain $ \Y $ ensures a controlled upper bound on the estimation error rate.

\subsection{Algorithm for Unbounded Case}\label{subsection:robust_algorithm_for_unbounded_case}

We now turn to the more challenging case of unbounded constraint sets $ K $. Two key components enable us to handle the unboundedness effectively. First, we introduce a large enough radius $ R $ within which we can guarantee, with high probability, that the data concentrates around the true mean $ \mu $. This leads to the construction of a random set $ S(R) $ that contains $ \mu $ with high probability. Second, we develop an appropriate countable covering and packing structure for $ K $ that facilitates a directed tree construction at each covering point. These two elements work in tandem to localize the estimation problem within a manageable region, ultimately enabling us to select an optimal tree on which to execute our tournament-based algorithm from \autoref{subsection:robust_algorithm_for_bounded_case}.

For the unbounded case, we assume $ K $ to be a closed set to ensure the existence of our countable covering and packing construction. When $ K $ is not closed, we use its closure $ \bar{K} $ instead. As discussed in \autoref{lemma:closed_K}, this substitution preserves the minimax rates.

$R$ should be chosen to be larger than some constant $R_0$ depending on $n,\sigma,\varepsilon$. More details concerning the constant selection and the corresponding tail bounds for $ S(R) $ are presented in \autoref{section:unbounded_case} and \autoref{section:proof_unbounded_case}. We focus here on the construction of the algorithm.

\paragraph{Introducing the Random Set $ S(R) $.} We define the random set $ S=S(R) $ as a data-dependent subset of $ K $, determined by the observed data $ X_i $ and a specified radius $ R $.  
\begin{definition}\label{definition:random_set_S}
    For each fixed $ \nu  $ and $ i\in [N] $, let event $ E_{\nu ,i}:= \{\left\Vert X_i - \nu  \right\Vert >R \} $ and define
    \begin{align*}
        S=S(R):= \big\{ \nu \in K:\, \sum_{i=1}^N \mathbbm{1}(E_{\nu ,i}) \leq N/2-1 \big\}, 
    \end{align*}
    that is, the random set of points for which strictly more than half of the observed data points lie within distance $ R $. The set $ S(R) $ possesses the following important properties:
    \begin{enumerate}[topsep=2pt,itemsep=0pt]
        \item $ S(R) $ is nested when $ R $ increases, i.e. $ S(R_1) \subseteq S(R_2) $ for $ R_1<R_2 $.
        \item $ \mathrm{ Diam }(S(R)) \leq 2R $ for any $ R>0 $. 
    \end{enumerate}
    These properties were established in \cite[Lemmas 5.2 and 5.4]{prasadan_information_2024}.
\end{definition}

\paragraph{Countable Packing of $ K $.} We set $ m = \frac{R}{c-1} $ and construct an $ m $-packing set $ S_m $ of $ K $, which simultaneously serves as a $ 2m $-covering of $ K $. The existence of such a construction is guaranteed by \cite[Lemma 5.5]{prasadan_information_2024} under our assumption that $ K $ is closed. 

For each point $ s $ in the $ m $-packing set $ S_m $, we apply \autoref{algorithm:tree_construction} with root $ s $, setting $ d_m = 2m+2R $ in place of $ d $ and restricting to $ K\cap B(s,d_m) $ in place of $ K $. Each resulting tree is denoted $ \mathcal{G}^s $, where $ s $ specifies the root and $ \mathcal{L}^s(J) $ represents its $ J $-th layer. We remark that the `forest' of directed trees $ \{\mathcal{G}^s:\, s\in S_m \}  $ is data-independent, as it relies solely on the set $ K $ and the distance $ R $.

Since the construction is the same except for a few initial conditions, we directly have the following properties for each tree $\mathcal{G}^s$ which are similar to \autoref{lemma:tree_properties}.

\begin{lemma}\label{lemma:tree_properties_unbounded}
    Let $ \mathcal{G}^s $ be the pruned tree constructed as above with root $ s $ and $ d_m = 2m + 2R $. We have the following properties:
    \begin{itemize}[topsep=2pt,itemsep=0pt]
        \item For any $ J\geq 3 $, $ \mathcal{L}^s(J) $ is a $ \frac{d_m}{2^{J-1}c} $-packing, and also a $ \frac{d_m}{2^{J-2}c} $-covering of $ K\cap B(s,d_m) $;
        \item For any $ J\geq 2 $ and parent node $ \Y_{J-1}\in \mathcal{L}^s(J-1) $, its offspring $ \mathcal{O}(\Y_{J-1}) $ form a $ \frac{d_m}{2^{J-2}c} $-covering of $  B(\Y_{J-1},\frac{d_m}{2^{J-2}})\cap  K\cap B(s,d_m) $. Moreover, the cardinality satisfies $ \mathrm{ card }\big( \mathcal{O}(\Y_{J-1}) \big) \leq M^\mathrm{ loc }_K(\frac{d_m}{2^{J-2}}, 2c)   $;
        \item For any $\mu\in K$ and $ J\geq 2 $, we have
        \begin{align*}
            \mathrm{ card }\big(  \mathcal{L}^s(J-1) \cap B(\mu ,\dfrac{ d_m }{ 2^{J-2} } ) \big) \leq M^\mathrm{ loc }_K(\frac{d_m}{2^{J-2}}, c)  \leq M^\mathrm{ loc }_K(\frac{d_m}{2^{J-2}}, 2c);
        \end{align*}
        \item For any winner chain $ [\Y_1,\Y_2,\ldots] $ in the tree $ \mathcal{G}^s $, where $ \Y_{J+1}\in\mathcal{O}(\Y_J) $ for all $ J\in \mathbbm{N} $, we have
        \begin{align*}
            \left\Vert \Y_{J'}-\Y_{J} \right\Vert \leq \dfrac{ d_m(2+4c) }{ c2^{J'
            } }   ,\quad J\geq J'\geq 1.
        \end{align*}
    \end{itemize}
\end{lemma}

\paragraph{Robust Algorithm.} Then we can present the robust algorithm, which considers the following two subcases. For details of the estimator and error upper bound see \autoref{theorem:expected_risk_upper_bound_unbounded}.
\begin{itemize}[topsep=2pt,itemsep=0pt]
    \item If $ S = \varnothing $, we directly take the estimator as the lexicographically smallest point in $ S(\hat{R}) $ where $ \hat{R} = \min \{ t>0: S(t)\neq \varnothing \}  $;
    \item If $ S \neq \varnothing $, we pick $ s\in S_m $ to be the point in $ S_m $ that is closest to $ S(R) $ (breaking ties lexicographically).  Then \autoref{algorithm:robust_algorithm} is applied to the tree $ \mathcal{L}^{s}(J) $ to obtain a chain $ \Y = [\Y_1,\Y_2,\ldots] $. Then a node after enough steps along the chain is chosen as the estimator.
\end{itemize}
The estimator determined by the above process ensures a controlled upper bound on the error rate.

\section{Minimax Rate for Bounded Case}\label{section:upper_bound_bounded_case}

In this section, we establish the minimax rate for bounded constraint sets $K$. Our approach proceeds in three stages: First, we specify the robust comparison estimator for bounded constraint case in \autoref{subsection:robust_estimator}. This estimator acts as a key component in the tournaments-on-the-tree framework mentioned in \autoref{section:framework_for_upper_bound}. Using the framework, we then prove the probabilistic error control in \autoref{subsection:prob_error_control}. Finally in \autoref{subsection:bounded_upper_and_minimax_rate}, a matching upper bound in risk is established, then combined with the lower bound to characterize the minimax rate in the case of bounded $ K $. The analysis employs techniques based on Cantelli's inequality to handle the heavy-tailed nature of the data. Also, such nature affects the conditions for the algorithm to work, which yields some interesting insights on how the heavy-tailedness complicates the robust mean estimation problem.

\subsection{Robust Estimator}\label{subsection:robust_estimator}

As discussed in \autoref{subsection:tournament_based_on_robust_comparison}, the testing function $ \psi $ characterizes which of two candidate points better represents the observed data. We now present the construction of an estimator for this testing function.

For technical convenience, we draw auxilliary noise terms $W_i\mathop{ \sim }\limits^{i.i.d.} \mathcal{N}(0, \sigma ^2I)$ and instead consider the datapoints  $\{X_i+W_i\}_{i=1}^N$, following \cite{lugosi_mean_2019}. This ensures the existence of a density, and importantly, only affects the variance by a constant multiplicative factor and thus does not alter the minimax rate.

Now, to characterize the relation between candidate points $(\nu_1,\nu_2)$ and the data, we introduce the following one-dimensional discriminant quantities
\begin{align}\label{equation:definition_V_i}
    \tilde{V}_i:= \dfrac{ \Vert \tilde{X}_i+{W}_i-\nu _1 \Vert ^2 - \Vert \tilde{X}_i+{W}_i-\nu _2 \Vert ^2  }{ \left\Vert \nu _2-\nu _1 \right\Vert  }, \qquad V_i:= \dfrac{ \left\Vert X_i +W_i -\nu _1 \right\Vert ^2 - \left\Vert X_i +W_i -\nu _2 \right\Vert ^2  }{ \left\Vert \nu _2-\nu _1 \right\Vert  },
\end{align}
recalling that $ \tilde{X}_i $ denotes the original uncorrupted data and $ X_i $ denotes the potentially corrupted data. Note that $ (1-\varepsilon )N $ of the $ V_i $ values equal the corresponding $ \tilde{V}_i $ values (from uncorrupted data), while the remaining $ V_i $ values reflect the effect of adversarial contamination.

The discriminant quantity $ V_i $ provides a natural voting mechanism: if $ V_i \leq 0 $, the $ i^{\mathrm{th}} $ data point favors candidate $ \nu _1 $, and vice versa for $ \nu _2 $. Our task therefore reduces to constructing a robust estimator that effectively aggregates the $ V_i $ values. We proceed by first introducing the classical median estimator in \autoref{definition:median_estimator}, then the trimmed mean estimator from \cite{lugosi_robust_2019} in \autoref{definition:trimmed_mean_estimator}. Finally, we combine these approaches to create a unified robust estimator in \autoref{definition:robust_estimator} that powers our tournament mechanism. For convenience and without loss of generality, we assume throughout that the sample size is even and denote it as $ 2N $\footnote{For odd sample sizes, we can simply add one data point, say $\mathbf{0}$, which does not affect the minimax rate. The corruption rate would remain bounded below $1/32$ to satisfy the condition of \autoref{theorem:two_point_prob_bound}, as long as the sample size $N$ is large enough, which is automatically satisfied as in \autoref{theorem:expected_risk_upper_bound}.}.

\paragraph{Median Estimator}

The median serves as a classical robust estimator for univariate data. In our testing framework, the median estimator determines which of two candidates is closer to the majority of the data points. For the even sample size $ 2N $ considered here, we formally define the estimator using the $ N^{\mathrm{th}} $ order statistic of the $ V_i $ values. For notational convenience, we continue to refer to this as the `median' of $ V_i $s.

\begin{definition}[Median Estimator Testing]\label{definition:median_estimator}
    \begin{align*}
        \psi_\mathrm{ Median } &:=  \mathbbm{1}\big( \mathrm{ card } \big(\big\{ i\in [2N] :\, \left\Vert X_i+W_i-\nu _1 \right\Vert \geq \left\Vert X_i+W_i-\nu _2 \right\Vert \big\}\big) \geq N \big).
    \end{align*}
\end{definition}

\paragraph{Trimmed Mean Estimator}

The trimmed mean estimator is also a long-standing robust estimator for univariate data. \cite{lugosi_robust_2019} established its non-asymptotic properties in adversarial contamination settings. The method operates by first estimating appropriate quantiles using half of the data, then trimming the remaining half based on these quantiles to compute the trimmed mean. We employ this approach to construct our trimmed mean estimator, denoted $ \mathrm{ TM }_{\delta _0}(\{V_i\}_{i=1}^{2N})  $. Let $ \delta _0 $ represents the desired Type I error bound, which influences the quantile to be used for trimming. The estimator
guarantees that $\vert \mathrm{ TM }_{\delta _0}(\{V_i\}_{i=1}^{2N}) $ is close to $ \mathbbm{E}_{  }\left[ V_i \right] $ and thus we can effectively determine which candidate is closer to the true mean by looking at the sign of $ \mathrm{ TM }_{\delta _0}(\{V_i\}_{i=1}^{2N})  $.

\begin{definition}[Trimmed Mean Estimator Testing]\label{definition:trimmed_mean_estimator}
    \begin{align*}
     \psi_\mathrm{ TM }&:=  \mathbbm{1}\big( \mathrm{ TM }_{\delta _0}(\{V_i\}_{i=1}^{2N})  \geq 0 \big).
    \end{align*}
\end{definition}

\paragraph{Combined Robust Testing}

As will become apparent, the probabilistic tail bounds for the above two robust estimators are effective only when $ \frac{\delta ^2}{\sigma ^2} $ is either large or small, respectively. For tail bounds that are effective for a broader range of $\frac{\delta ^2}{\sigma ^2}$ so that we can proceed the estimation algorithms, we construct the following hybrid robust estimator that adaptively combines the trimmed mean and median estimators, with the switching boundary related to $ C $ and the rate of the tail probability.

\newcommand{\thenewboundary}{(96C_2)^{-1}}
\begin{definition}[Robust Testing]\label{definition:robust_estimator}
    Given potentially corrupted data $ X =\left\{ X_1,\ldots,X_{2N} \right\} $ and an ordered pair $ (\nu _1, \nu _2) $ of candidate points $ \nu _1,\nu _2\in \mathbbm{R}^n  $ satisfying $ \left\Vert \nu _1-\nu _2 \right\Vert \geq C\delta  $ for some sufficiently large constant $ C>2 $, we define the robust estimator as follows:
    \begin{align*}
        \psi_{(\nu _1,\nu _2),\delta ;\delta _0,C_2}(\{X_i\}_{i=1}^{2N}) = \begin{cases}
            \mathbbm{1}(\mathrm{ TM }_{\delta _0} (\{V_i\}_{i=1}^{2N}) \geq 0) ,& \text{if } \frac{\delta ^2}{\sigma ^2}\leq \thenewboundary ,\\
            \mathbbm{1}(\mathrm{ Median }(\{V_i\}_{i=1}^{2N}) \geq 0)  ,& \text{if } \frac{\delta ^2}{\sigma ^2}> \thenewboundary ,
        \end{cases} 
    \end{align*}
    where $ V_i $ is defined in \autoref{equation:definition_V_i}.
    $C_2$ is a constant determined later as in \autoref{lemma:useful_constants_1}, which concerns the boundary between the two estimators above.
    We simplify notation by omitting auxiliary subscripts and writing $ \psi_{(\nu _1,\nu _2)} $ or $\psi$ when the context is clear.
\end{definition}

\subsection{Probabilistic Error Control}\label{subsection:prob_error_control}

The robust estimator developed above enables us to determine the preference relation $ \nu _1 \mathop{ \succ  }\limits_{\prec } \nu_2 $ between any candidate pair $ (\nu _1, \nu _2) $. Crucially, we can establish control over the Type I error of the robust estimator from \autoref{definition:robust_estimator}, as formalized in \autoref{theorem:two_point_prob_bound}. 

\begin{theorem}[Robust Testing Tail Probability]\label{theorem:two_point_prob_bound}
    There exist universal constants $C>2$, $C_5>0$ with the following properties: For any $\sigma>0$ and $ \varepsilon \in (0,1/32)$, consider the following hypothesis testing problem:
    \begin{align*}
        H_0:\, \mu \in B(\nu _1,\delta ) \text{ v.s. } H_1:\, \mu \in B(\nu _2,\delta ) \text{ for } \left\Vert \nu _1-\nu _2 \right\Vert \geq C\delta ,
    \end{align*}
    where $ \nu _1,\nu _2\in K $. Let $\psi_{(\nu_1,\nu_2)}$ be the robust test from \autoref{definition:robust_estimator}. If $\delta$ satisfies $\x \gtrsim N^{-1}\vee \varepsilon$, then 
    \begin{align*}
        \mathop{ \sup }\limits_{\mu :\left\Vert \mu-\nu _1 \right\Vert \leq \delta } \mathbbm{P}_\mu \left( \psi=1 \right)  \vee\mathop{ \sup }\limits_{\mu :\left\Vert \mu-\nu _2 \right\Vert \leq \delta } \mathbbm{P}_\mu \left( \psi=0 \right) &\leq   \exp\left[ -C_5N\log (1+\x)  \right].
    \end{align*}
\end{theorem}

Building upon the tail bound for the two-point comparison, we can extend the analysis following \cite[Theorem 4.5]{prasadan_information_2024} to obtain a union bound over the covering set. This yields an error control for the tournament introduced in \autoref{definition:tournament_T}.

\begin{theorem}[Tournament Tail Probability]\label{theorem:covering_set_prob_bound}
    Let $ C,C_5 $ be the constants from \autoref{theorem:two_point_prob_bound}, and let $ \mathcal{M}=\{\nu _1,\ldots,\nu _M\} $ be a $ \delta $-covering of $ K'\subset K $ with cardinality $M$. Consider the tournament function $ T(\delta ,\nu ,\mathcal{M}) $ from \autoref{definition:tournament_T} with the robust test from \autoref{definition:robust_estimator} applied. Assume $ \mu \in K' $ and denote $ i^*\in \mathop{ \arg\min }\limits_{i\in [M]} T(\delta ,\nu _i,\mathcal{M})  $. Under the condition $ \x \gtrsim N^{-1}\vee \varepsilon  $, we have
    \begin{align*}
        \mathbbm{P}\left( \left\Vert \nu _{i^*} - \mu  \right\Vert \geq (C+1)\delta  \right) \leq M\exp\left[ -C_5N\log (1+\dfrac{ \delta ^2 }{ \sigma ^2 } ) \right] .
    \end{align*}
\end{theorem}

\subsection{Upper Bound and Minimax Rate}\label{subsection:bounded_upper_and_minimax_rate}

Having established the error control for individual tournament, we now derive an upper bound for the overall risk of our robust algorithm. The key insight is that by choosing the traverse length appropriately in the winner chain from \autoref{algorithm:robust_algorithm}, we can achieve the desired risk upper bound.

\begin{theorem}[Bounded Heavy-Tailed Upper Bound]\label{theorem:expected_risk_upper_bound}
    Let $ C,C_5 $ be the constants from \autoref{theorem:two_point_prob_bound} where $C$ is chosen large enough, and set $ c=2(C+1) $ as the constant in the local metric entropy. Let $\Y$ be the tournament winner chain from the algorithm in \autoref{algorithm:robust_algorithm}. We construct the estimator as follows:    
    For any $J\in \mathbbm{N}$, define $\delta_J = \frac{d}{2^{J-1}(C+1)}$. Let $ J^*\geq 1 $ be the largest $ J $ such that
    \begin{align}\label{eq:J_star_condition}
        C_5N\log(1+ \frac{\delta_J^2}{\sigma^2}) &\geq 4\log M^{\mathrm{ loc } }(c\delta_J,2c) \vee \log 2,
    \end{align}
    and set $ J^*=1 $ if this condition is never satisfied. Then at some step $J^{**}\geq J^*$ we use its output as the estimator $ \nu ^{**}=\Y_{J^{**}+1} $.

    Under the condition $ N\gtrsim \mathop{ \sup }\limits_{\delta >0} \log M^{\mathrm{ loc } }(\delta ,2c) $ we have the following risk upper bound:
    \begin{align*}
        \mathbbm{E}_W\mathbbm{E}_X\left\Vert \nu^{**}-\mu  \right\Vert ^2 \lesssim  \max\{\delta_{J^*}^2, \varepsilon \sigma ^2\} \wedge d^2.
    \end{align*}
\end{theorem}

\begin{remark}\label{remark:J_star_condition_modified}
    We observe that \autoref{eq:J_star_condition} uses $ \tilde{c}:= 2c $ as the second parameter in the local metric entropy, which differs from the constant $ c $ in the metric entropy lower bound from \autoref{lemma:packing_lower_bound}. However, replacing $ c $ with $ \tilde{c}=2c $ in the lower bound only changes it by a constant factor. Consequently, we can harmonize the constants by using the same second parameter $ \tilde{c} $ in both the lower bound condition of \autoref{lemma:packing_lower_bound} and \autoref{eq:J_star_condition}. Therefore, without loss of generality, we can replace \autoref{eq:J_star_condition} with the following equivalent condition: let $ J^*\geq 1 $ be the maximal integer $ J $ such that
    \begin{align}\label{eq:J_star_condition_modified}
        C_5N\log(1+ \frac{\delta_J^2}{\sigma^2}) &\geq 4\log M^{\mathrm{ loc } }(\frac{c}{2}\delta_J,c) \vee \log 2,
    \end{align}
    and set $ J^*=1 $ if the above condition is never satisfied.
\end{remark}

We now combine our upper and lower bounds to establish the minimax rate for bounded constraint sets:

\begin{theorem}[Bounded Heavy-Tailed Minimax Rate]\label{theorem:expected_risk_rate_bounded_case}
    Let $\varepsilon \in (0,\frac{1}{32})$.  Define
    \begin{align}\label{equation:definition_of_delta_star}
        \delta^*=\sup\big\{ \delta\geq 0:\, N \frac{\delta^2}{\sigma^2} \leq \log M^\mathrm{loc}(\delta,c) \big\}.
    \end{align}
    For sufficiently large $c$ and under the assumption $N\gtrsim \sup_\delta \log M^\mathrm{loc}(\delta,c)$, the minimax rate is $ \mathfrak{M}\asymp\max(\delta^{*2}, \varepsilon\sigma^2)\wedge d^2$.
\end{theorem}

\begin{remark}\label{remark:volumn_argument1}
    By the volume argument as presented in \cite[Lemma 5.7]{wainwright_high-dimensional_2019}, for any $ \delta >0 $ we have $ \log M^\mathrm{loc}(\delta ,c) \leq n\log(1+2c) \asymp n $. Therefore, $ N\gtrsim n $ acts as a sufficient condition for $ N\gtrsim \sup_\delta \log M^\mathrm{loc}(\delta,c) $. On the other hand, in the special case that $K$ contains an interior point, we have $\sup_\delta \log M^\mathrm{loc}(\delta,c) \asymp n$. Thus the condition $ N\gtrsim \sup_\delta \log M^\mathrm{loc}(\delta,c) $ simply boils down to $ N\gtrsim n $ in this scenario.
\end{remark}

\section{Minimax Rate for Unbounded Case}\label{section:unbounded_case}

Having established the minimax rate for bounded constraint sets, we now turn to the more challenging case of unbounded sets $ K $. As mentioned in \autoref{subsection:robust_algorithm_for_unbounded_case}, the analysis requires additional technical machinery to localize the problem within an appropriate finite region, characterized by the $S(R)$ from \autoref{definition:random_set_S}, to handle the unbounded nature of $ K $. With a proper $R$, we are able to obtain the estimator directly or determine a tree in the forest on which we can run the tournament-on-the-tree algorithm, depending on the behavior of $S(R)$.

The choice of the localization radius $ R $ and the resulting probability bounds are formalized as follows. 

\begin{theorem}[Set $S$ Tail Probability]\label{theorem:random_set_S_prob_bound}
    Let $ S=S(R) $ be as in \autoref{definition:random_set_S}. There exists an absolute constant $\gamma$, and constant $ R_0 $ depending on $ n,\sigma,\varepsilon  $ such that for any $ R\geq R_0 $ it holds that
    \begin{align*}
        \mathbbm{P}_{  }\left( \mu \not\in S \right) \leq  \exp\left[ -\dfrac{ N(1/2-\varepsilon )\gamma  }{ 16 } \log(1+\frac{R^2}{\sigma ^2}) \right] \wedge \frac{1}{2}.
    \end{align*}
\end{theorem}

With the probability bound for the set $ S $, we can now present the minimax upper bound for the robust estimator in the unbounded case from \autoref{subsection:robust_algorithm_for_unbounded_case}.

\begin{theorem}[Unbounded Heavy-Tailed Upper Bound]\label{theorem:expected_risk_upper_bound_unbounded}
    Let $ C,C_5 $ be the constants from \autoref{theorem:two_point_prob_bound} where $C$ is chosen large enough, and take $ c=2(C+1) $ as the constant in the local metric entropy. Let $\Y$ be the tournament winner chain in \autoref{subsection:robust_algorithm_for_unbounded_case}.
    For any $J\in \mathbbm{N}$, define $\delta_J = \frac{d_m}{2^{J-1}(C+1)}$ where $m$ and $d_m$ are as defined in \autoref{subsection:robust_algorithm_for_unbounded_case}. Let $ J^*\geq 1 $ be the largest $ J $ such that
    \begin{align}\label{eq:J_star_condition_unbounded}
        C_5N\log(1+ \frac{\delta_J^2}{\sigma^2})& \geq 6\log M^{\mathrm{ loc } }(c\delta_J,2c) \vee \log 2 ,
    \end{align}
    and set $ J^*=1 $ if this condition is never satisfied. 
    
    The algorithm output is defined as follows: If $ S(R)\neq \varnothing $, at some step $J^{**}\geq J^*$ we use its output as the estimator $\nu^{**} = \Y_{J^{**}+1}$, let $ \nu ^{**} $ be the lexicographically smallest point in $ S(\hat{R}) $, where $ \hat{R} = \min \{t>0 : S(t)\neq \varnothing \} $. Under the condition $ N\gtrsim  \sup_{\delta >0}\log M^\mathrm{ loc }(\delta ,2c) $, we have the following risk upper bound:
    \begin{align*}
        \mathbbm{E}_{ W }\mathbbm{E}_{ X } \left\Vert \nu ^{**}-\mu  \right\Vert ^2 \lesssim  \max\{\delta_{J^*}^2, \varepsilon \sigma ^2\} .
    \end{align*}
\end{theorem}

\begin{remark}\label{remark:J_star_condition_unbounded_modified}
    As noted in \autoref{remark:J_star_condition_modified}, we can use the following modified condition while obtaining the same upper bound rate:
    \begin{align}\label{eq:J_star_condition_unbounded_modified}
        C_5N\log(1+\frac{\delta _J^2}{\sigma ^2}) \geq 6\log M^\mathrm{ loc }(\frac{c}{2}\delta _J, c)\vee \log 2     ,
    \end{align}    
    and set $ J^* = 1 $ if the above condition is never satisfied.
\end{remark}

Now we can establish the minimax rate for unbounded constraint sets, combining the upper and lower bounds. The treatment of non-closed sets by working with their closures $ \bar{K} $ is addressed in \autoref{lemma:closed_K}.

\begin{theorem}[Unbounded Heavy-Tailed Minimax Rate]\label{theorem:expected_risk_rate_unbounded_case}
    Let $ \varepsilon \in (0,\frac{1}{32}) $, and $\delta^*$ as in \autoref{equation:definition_of_delta_star}.
    For sufficiently large $ c $ and under the assumption $ N\gtrsim \sup_\delta \log M^\mathrm{ loc }(\delta ,c)  $, the minimax rate is $ \mathfrak{M}\asymp \max(\delta^{*2}, \varepsilon\sigma^2)$.
\end{theorem}

\begin{remark}
    The setting of the theorem is the same as \autoref{theorem:expected_risk_upper_bound}. We also recall that the argument from \autoref{remark:volumn_argument1} of a sufficient condition $N\gtrsim n$ still applies here.
\end{remark}

\section{Minimax Rate for Symmetric or Known Distribution Case}\label{section:symmetric_case}

In this section, we present the interesting result that when the noise distribution is sign-symmetric, or the distribution of the noise is known to the statistician, the dependence on the contamination fraction $ \varepsilon $ is improved to match the optimal rate in the Gaussian case. This phenomenon was first observed in \cite{novikov_robust_2023} where they investigated some special cases of symmetric distributions and gave computationally efficient estimators. Later, \cite{prasadan_information_2024} showed that such improvement holds for \textit{any} symmetric or known sub-Gaussian noise in the minimax risk framework. Similar results for sub-Gaussian distribution are also presented in \cite{minasyan_statistically_2023}.
We extend this result to the heavy-tailed setting and obtain a minimax risk rate by employing the Huber's loss based estimator from \cite{novikov_robust_2023} (for the one dimensional case), demonstrating that this improvement persists even under the weaker assumption of finite second moments. The improvement for symmetric distributions also applies to the case of known distributions, since such problem can be reduced to the symmetric case by considering its difference with an independent copy of the noise. 

The key technical component is the following robust comparison estimator for symmetric distributions, which adapts the approach from \cite{novikov_robust_2023} to construct our testing function.
\begin{definition}[Symmetric Case Estimator]\label{definition:novikov_estimator}
    Given potentially corrupted data $ X=\{X_1,\ldots,X_N\} $ and an ordered pair $ (\nu _1,\nu _2) $ of candidate points $\nu_1,\nu_2 \in \mathbbm{R}^n $ satisfying $ \left\Vert \nu _1 - \nu _2 \right\Vert \geq C\delta $ for some sufficiently large constant $ C>2 $. Let $ \hat{\mu }_\mathrm{ Novikov } $ be the estimator defined in \cite{novikov_robust_2023}[Theorem 1.5].    
    We define the robust estimator for symmetric distributions as follows:
    \begin{align*}
    \psi^\mathrm{ sym }_{(\nu _1,\nu _2), \delta }(\{V_i\}_{i=1}^N) := \mathbbm{1}\left( \hat{\mu }_\mathrm{ Novikov }\big( \{V_i\}_{i=1}^N \big) >0  \right) ,
    \end{align*}
    where the one-dimensional discriminant quantities $ V_i $ are from \autoref{equation:definition_V_i}. We remind readers that the symmetry of $ \tilde{X}_i $ implies the symmetry of $ \tilde{V}_i $ around its mean by the definition. We write $\psi^\mathrm{sym}$ for brevity in the following.
\end{definition} 

The performance of this estimator is characterized by the following result, which serves as a counterpart to \autoref{theorem:two_point_prob_bound}. 

\begin{theorem}[Symmetric Case Tail Probability]\label{theorem:novikov_prob_bound}
    There exist universal constants $ C $, $ Q $, $ C_{12} $ with the following properties: 
    Assume $ \varepsilon \leq 1/Q $ and $ N\geq Q $. For any $ C\delta $-separated points $ \nu _1,\nu _2\in K $ with $ \mu \in B(\nu_1,\delta ) $ and $ \delta $ satisfying $ \frac{\delta ^2}{\sigma ^2}\gtrsim N^{-1} \vee \varepsilon ^2 $, 
    then $ \hat{\mu }_\mathrm{ Novikov } $ satisfies
    \begin{align*}
        \mathbbm{P}_{ \mu \in B(\nu_1,\delta ) }\left( \hat{\mu }_\mathrm{ Novikov }\big( \{V_i\}_{i=1}^N \big) >0  \right) &\leq \exp\left[ - C_{12}N\frac{\delta ^2}{\sigma ^2} \right]  .
    \end{align*}
\end{theorem}

As a direct result, we have the following Type I error control: 
\begin{align*}
    \mathop{ \sup }\limits_{\mu :\left\Vert \mu-\nu _1 \right\Vert \leq \delta } \mathbbm{P}_\mu \left( \psi^\mathrm{sym}=1 \right)  \vee\mathop{ \sup }\limits_{\mu :\left\Vert \mu-\nu _2 \right\Vert \leq \delta } \mathbbm{P}_\mu \left(  \psi^\mathrm{sym}=0 \right) &\leq   \exp\left[ -C_{12}N\frac{\delta ^2}{\sigma ^2} \right].
\end{align*}

The remainder of the analysis follows the same framework as in \autoref{theorem:covering_set_prob_bound}, \ref{theorem:expected_risk_upper_bound}, and \ref{theorem:expected_risk_rate_bounded_case} for bounded $ K $, and \autoref{theorem:expected_risk_rate_unbounded_case} for unbounded $ K $. The key improvements are: (i) the condition on $\delta$ is now $ \frac{\delta ^2}{\sigma ^2}\gtrsim N^{-1} \vee \varepsilon ^2 $ rather than $ \frac{\delta ^2}{\sigma ^2}\gtrsim N^{-1} \vee \varepsilon $ due to the symmetry assumption, and (ii) we obtain exponential rather than polynomial tail bounds. The exponential tail behavior can be handled directly as in \cite[Section 4.2]{prasadan_information_2024}. We present the final minimax rate below.

\begin{theorem}[Symmetric Heavy-Tailed Minimax Rate]\label{theorem:exptected_risk_rate_symmetric}
    There exist sufficiently large universal constants $ c $ and $ Q $ such that the following holds: Let $ \varepsilon \in (0,1/Q) $, and $\delta^*$ as defined in \autoref{equation:definition_of_delta_star}. Assume $ N\gtrsim \sup_\delta  \log M^\mathrm{loc}(\delta,c) $. Then the minimax rate is $\mathfrak{M}_\mathrm{ sym }\asymp \max(\delta^{*2}, \varepsilon^2\sigma^2)\wedge d^2$ for bounded $ K $, and $\max(\delta^{*2}, \varepsilon^2\sigma^2)$ for unbounded $ K $.
\end{theorem}

\begin{remark}
    We make a side note about the condition on the sample size. Using the same argument as the proof of \autoref{theorem:expected_risk_upper_bound}, $ \log M^\mathrm{ loc }  $ can be made arbitrarily large by choosing $ c $ large enough. Thus, the condition $ N \geq Q $ of \autoref{theorem:novikov_prob_bound} can be absorbed into the condition $ N\gtrsim \sup_\delta \log M^\mathrm{ loc }(\delta ,c)  $.
\end{remark}

\section{Examples}\label{section:examples}

In this section, we illustrate our minimax rate results through two canonical constraint sets: the $ \ell_0 $ (sparsity) and $ \ell_1 $ ball constraints. These examples demonstrate the applicability of \autoref{theorem:expected_risk_rate_bounded_case} and \autoref{theorem:expected_risk_rate_unbounded_case} to well-studied problems in high-dimensional statistics. Our framework readily extends to other important constraint families, including general $ \ell_p $ balls ($ p\geq 0 $), polytopes, etc. For tidiness of the results, we set $ \sigma = 1 $ in the remaining part of this section. 

\subsection{Sparse Vector Constraint}\label{subsection:sparse_vector_constraint}

We first consider the fundamental problem of $ s $-sparse vector estimation, where the constraint set is the unbounded $ K = B_0(s) := \{x \in \mathbbm{R}^n: \left\Vert x \right\Vert _0 \leq s \} $ and the integer $ 1 \leq s \leq n/8 $ denotes the sparsity level. The local metric entropy for this constraint set was established in \cite[Lemma 5.14]{prasadan_information_2024}: 
\begin{align*}
    \log M^\mathrm{ loc }_{B_0(s)}(\delta ,c)\asymp s\log(1+\frac{n}{2s}).
\end{align*}
Applying \autoref{theorem:expected_risk_rate_unbounded_case}, we obtain the minimax rate 
\begin{align*}
\mathfrak{M}\asymp \max\left( \frac{ s\log\left(1+n/2s\right) }{ N }, \varepsilon  \right).
\end{align*}

This rate differs from that obtained in \cite{comminges_adaptive_2021}, who established rates of $  O\big(s\log(en/s) \big) $ for sub-Gaussian distributions and $ O(n) $ for finite-variance distributions. This discrepancy likely stems from our assumption of moderately large sample sizes $ N\gtrsim \sup_\delta \log M^\mathrm{ loc }(\delta ,c) \asymp   s\log(1+\frac{n}{2s}) $ in the finite-variance setting, in contrast to their $N=1$ setting. 

\subsection{\texorpdfstring{$\ell_1$}{ell1} Ball Constraint}

Our second example considers the $ \ell_1 $ ball constraint $ K = B_1(1) := \{x \in \mathbbm{R}^n: \left\Vert x \right\Vert _1 \leq 1 \} $. The local metric entropy for this constraint set was characterized in \cite[Lemma 14]{prasadan_facts_2025} as follows:
\begin{align*}
    \log M^\mathrm{ loc }_{B_1(1)}(\delta ,c)\asymp \begin{cases}
        \frac{\log(\delta ^2n)}{\delta ^2}, & \delta \gtrsim 1/\sqrt{n},\\
        n, & \delta \lesssim 1/\sqrt{n}\text{ or } \delta \asymp 1/\sqrt{n}.
    \end{cases}
\end{align*}
The sample size condition becomes $ N\gtrsim n $ in this setting. The following result characterizes the minimax rate across different parameter regimes.
\begin{lemma}\label{lemma:ell_1_ball_minimax_rate}
    For $ K = B_1(1) $, the minimax rate is $ \mathfrak{M}\asymp\max\big( \sqrt{\frac{\log (n/\sqrt{N})}{N}}, \varepsilon  \big) \wedge 1 $ if $ n\lesssim N \lnsim n^2 $, and $ \mathfrak{M}\asymp\max\big( \frac{n}{N}, \varepsilon  \big)\wedge 1 $ if $ N\gtrsim n^2 $.
\end{lemma}

\section{Discussion}\label{section:discussion}

In this work, we studied the problem of constrained robust mean estimation under adversarial corruption with only a finite second moment assumption. To extend the sub-Gaussian case in \cite{prasadan_information_2024}, we adopted the authors' directed tree construction and generalized the testing function to a heavy-tailed setting, thereby accommodating the more realistic finite second moment assumption. Novel techniques were developed to handle the challenges posed by heavy-tailed distributions. Our main results, presented in \autoref{theorem:expected_risk_rate_bounded_case} and \autoref{theorem:expected_risk_rate_unbounded_case}, established the minimax rates for both bounded and unbounded constraint sets $ K $. These rates are characterized by the local metric entropy of $ K $ and exhibit a dependence on the contamination proportion $ \varepsilon $ of order $ O(\varepsilon) $, which is shown to be optimal through matching lower bounds. Furthermore, we investigated the case with extra distributional structure in \autoref{theorem:exptected_risk_rate_symmetric}, specifically assuming symmetric or known noise distributions. We refined the testing functions using an estimator from \cite{novikov_robust_2023} to achieve improved rates. The result indicates that the minimax rate reduces to the Gaussian distribution case of $O(\varepsilon^2)$ with the extra knowledge of distribution. This finding highlights the significant impact of distributional assumptions on the achievable rates in robust estimation problems. Finally, we illustrated the applicability of our results through two canonical examples: sparse vector estimation and $ \ell_1 $ ball constrained estimation.

There are several possible extensions of this work that merit further investigation. First, the tournament-series-on-tree framework summarized in \autoref{section:framework_for_upper_bound} can be applied to other related estimation problems beyond mean estimation, such as regression and density estimation, potentially under contamination and various constraints. Moreover, while we specifically focused on squared $ \ell_2 $ loss in this work, other loss functions may also be considered. Second, in the sparse vector estimation example of \autoref{subsection:sparse_vector_constraint}, we noted that there is a gap between our result and existing results from \cite{comminges_adaptive_2021}, potentially due to different sample size assumptions. Further research is needed to clarify this discrepancy. There are also possible extensions on the practical side. Our results established the theoretical limits but lack computational efficiency. It would be interesting to investigate whether a computationally efficient algorithm could achieve the same rate. A recent work of \cite{neykov_polynomial-time_2025} made progress by proposing a polynomial-time algorithm which gives poly-logarithmic dependence rate, where the constraint set $ K $ belongs to a special class of Type-2 convex bodies. It is worth exploring whether their method can be extended to star-shaped or general $K\subseteq \mathbbm{R}$ and achieve better rates. Additionally, Lepskii's scheme \cite{lepskii_problem_1991} may be applied to enable adaptation to unknown $ \sigma $ and $ \varepsilon $ scenarios. Covariance estimation methods may also be combined with our approach for the unknown $ \sigma $ case. 

\bibliography{Citation}

\appendix

\section{Proof for Lower Bound}\label{section:proof_lower_bound}

\begin{proof}[Proof of \autoref{lemma:diakonikolas_lower_bound}]\label{proof:lemma:diakonikolas_lower_bound}

First, we construct a noise distribution setting that has finite variance while being sufficiently dispersed. In this proof, for alignment with traditional notation, we use $ \delta _x $ to denote the Dirac delta distribution at a point $ x\in \mathbbm{R}^n $. Let $ k^* $ be a center of $ K $, and consider the following mixture distribution:
    \begin{align*}
        M_\varepsilon := (1-\dfrac{ \varepsilon  }{ 2 }  )\delta _{k^*} + \dfrac{ \varepsilon  }{ 2 }  \delta _{k^* + \frac{\Delta }{\sqrt{\varepsilon /2\cdot  (1-\varepsilon /2)}}} ,
    \end{align*}
    such that the mean of $M_\varepsilon$ is
    \begin{align*}
        \mu _{\mathcal{M}_\varepsilon } := k^* + \Delta\frac{ \sqrt{\varepsilon/2 } }{\sqrt{1-\varepsilon/2 }}.
    \end{align*}
    In the definition of the distribution, $ \Delta \in \mathbbm{R}^n $ is a fixed vector such that $ \left\Vert \Delta  \right\Vert \asymp \sigma \wedge \frac{d}{\sqrt{\varepsilon} } $, which is always achievable according to \autoref{lemma:long_enough_segment_in_K} by noting that $\Vert \mu_{\mathcal{M}_\varepsilon} - k^*\Vert \asymp \left\Vert \Delta \right\Vert \sqrt{\varepsilon}$. This condition also ensures that $ M_\varepsilon $ has variance bounded by $ \sigma ^2 $, and $\mu _{\mathcal{M}_\varepsilon } \in K$.

    Next, we construct a corruption procedure $ \tilde{\mathcal{C}} $ that enables us to lower bound the minimax rate. The procedure $ \tilde{\mathcal{C}} $ operates as follows: let $ W $ be the number of times $ k^* + \frac{\Delta }{\sqrt{\varepsilon /2 (1-\varepsilon /2)}} $ is drawn in place of $ k^* $. The corruption procedure then performs:
    \begin{itemize}[topsep=2pt,itemsep=0pt]
        \item If $ W\leq \varepsilon N $, convert all occurrences of $ k^* + \frac{\Delta }{\sqrt{\varepsilon /2 (1-\varepsilon /2)}} $ to $ k^* $;
        \item If $ W>\varepsilon N $, do nothing.
    \end{itemize}
    It is clear that $ W $ follows a binomial distribution $ W\sim \mathrm{Binom}(N, \frac{ \varepsilon  }{ 2 } ) $. By the median property from \cite[Theorem 1]{kaas_mean_1980}, for any $ \varepsilon <1/2 $, we have that $ \mathbbm{P}\left( W\geq \frac{\varepsilon }{2}N \right)>\frac{1}{2} $ and $ \mathbbm{P}\left( W\leq \frac{\varepsilon }{2}N \right)>\frac{1}{2} $. For any fixed estimator $ \tilde{\mu } $, we have:
    \begin{align*}
        \mathfrak{M}&=\mathop{ \sup }\limits_{\hat{\mu}}\mathop{ \sup }\limits_{\mu \in K} \mathop{ \sup }\limits_{\tilde{X}\sim\xi \in \Xi_{\sigma ^2}  }  \mathop{ \sup }\limits_{\mathcal{C}} \mathbbm{E}_{  }\left[ \left\Vert \hat{\mu }(\mathcal{C}(\tilde{X})) - \mu  \right\Vert ^2 \right] \\
        &\geq \mathop{ \sup }\limits_{\mu \in K} \mathop{ \sup }\limits_{\tilde{X}\sim\xi \in \Xi_{\sigma ^2}  }  \mathop{ \sup }\limits_{\mathcal{C}} \mathbbm{E}_{  }\left[ \left\Vert \tilde{\mu }(\mathcal{C}(\tilde{X})) - \mu  \right\Vert ^2 \right] \\
        &\geq\frac{1}{2}\underbrace{\mathop{ \sup }\limits_{\mathcal{C}
        } \mathbbm{E}_{ \tilde{X}\sim\mathcal{M}_\varepsilon^{\otimes N}  }\left[ \left\Vert \tilde{\mu }(\mathcal{C}(\tilde{X})) - \mu_{\mathcal{M}_\varepsilon }  \right\Vert ^2 \right]}_{\text{I}} + \frac{1}{2}\underbrace{\mathop{ \sup }\limits_{\mathcal{C}
        } \mathbbm{E}_{ \tilde{X}\sim\delta _{k^*} ^{\otimes N}  }\left[ \left\Vert \tilde{\mu }(\mathcal{C}(\tilde{X})) - k^*   \right\Vert ^2 \right]}_{\text{II}}.
    \end{align*}
    We deal with the two parts separately: 
    \begin{itemize}[topsep=2pt,itemsep=0pt]
        \item[I] We have:
        \begin{align*}
             \text{I}&=  \mathop{ \sup }\limits_{\mathcal{C}
             } \mathbbm{E}_{ \tilde{X}\sim\mathcal{M}_\varepsilon^{\otimes N}  }\left[ \left\Vert \tilde{\mu }(\mathcal{C}(\tilde{X})) - \mu_{\mathcal{M}_\varepsilon }  \right\Vert ^2 \right]\\
              &\mathop{ \geq }\limits^{(i)}  \mathbbm{E}_{ \tilde{X}\sim\mathcal{M}_\varepsilon^{\otimes N}  }\left[ \left\Vert \tilde{\mu }(\tilde{\mathcal{C}}(\tilde{X})) - \mu_{\mathcal{M}_\varepsilon }  \right\Vert ^2 \right]\\
             \geq& \mathbbm{P}_{  }\left( W\leq N\frac{\varepsilon }{2} \right) \mathbbm{E}_{ \tilde{X}\sim\delta _{k^*}^{\otimes N}  }\left[ \left\Vert \tilde{\mu }(\tilde{X}) - \mu_{\mathcal{M}_\varepsilon }  \right\Vert ^2 \right]\\
             \mathop{ \geq }\limits^{(ii)}&  \dfrac{ 1 }{ 2 }  \mathbbm{E}_{ \tilde{X}\sim\delta _{k^*} ^{\otimes N}  }\left[ \left\Vert \tilde{\mu }(\tilde{X}) - \mu_{\mathcal{M}_\varepsilon }  \right\Vert ^2 \right],
        \end{align*}
            where in $ (i) $ we choose the corruption procedure to be $ \tilde{\mathcal{C}} $, and in $ (ii) $ we utilize the binomial property of $ W $.
        \item[II] We have:
        \begin{align*}
            \text{II}&=\mathop{ \sup }\limits_{\mathcal{C}
            } \mathbbm{E}_{ \tilde{X}\sim\delta _{k^*} ^{\otimes N}  }\left[ \left\Vert \tilde{\mu }(\mathcal{C}(\tilde{X})) - k^*   \right\Vert ^2 \right]\\
            &\geq  \dfrac{ 1 }{ 2 } \mathop{ \sup }\limits_{\mathcal{C}
            } \mathbbm{E}_{ \tilde{X}\sim\delta _{k^*} ^{\otimes N}  }\left[ \left\Vert \tilde{\mu }(\mathcal{C}(\tilde{X})) -k^*   \right\Vert ^2 \right] \\
            &\geq \dfrac{ 1 }{ 2 } \mathbbm{E}_{ \tilde{X}\sim\delta _{k^*} ^{\otimes N}  }\left[ \left\Vert \tilde{\mu }(\tilde{X}) -k^*   \right\Vert ^2 \right].
        \end{align*}
    \end{itemize}

    Combining the two parts, we obtain:
    \begin{align*}
        \mathop{ \sup }\limits_{\mu \in K} \mathop{ \sup }\limits_{\tilde{X}\sim\xi \in \Xi_{\sigma ^2}  }  \mathop{ \sup }\limits_{\mathcal{C}} \mathbbm{E}_{  }\left[ \left\Vert \tilde{\mu }(\mathcal{C}(\tilde{X})) - \mu  \right\Vert ^2 \right]  &\geq \dfrac{ 1 }{ 4 }  \mathbbm{E}_{ \tilde{X}\sim\delta _{k^*} ^{\otimes N}  }\left[ \left\Vert \tilde{\mu }(\tilde{X}) - \mu_{\mathcal{M}_\varepsilon }  \right\Vert ^2 \right] + \dfrac{ 1 }{ 4 }\mathbbm{E}_{ \tilde{X}\sim\delta _{k^*} ^{\otimes N}  }\left[ \left\Vert \tilde{\mu }(\tilde{X}) -k^*   \right\Vert ^2 \right]\\
        &\geq \dfrac{ 1 }{ 8 }  \mathbbm{E}_{ \tilde{X}\sim\delta _{k^*} ^{\otimes N}  }\left[ \left\Vert {k^*} - \mu _{\mathcal{M}_\varepsilon } \right\Vert ^2 \right]\\
        &\asymp \left\Vert {k^*} - \mu _{\mathcal{M}_\varepsilon } \right\Vert ^2 \\
        &\asymp  \dfrac{ \varepsilon /2 }{ 1-\varepsilon /2 }\cdot (\sigma ^2\wedge \dfrac{ d^2 }{ \varepsilon  } ) \\
        &\gtrsim  \varepsilon \sigma ^2 \wedge d^2.
    \end{align*}

    Since the above argument holds for any $ \tilde{\mu }$, taking the infimum over all possible $ \tilde{\mu }$ yields:
    \begin{align*}
        \mathfrak{M}&= \mathop{ \inf }\limits_{\hat{\mu }} \mathop{ \sup }\limits_{\mu \in K} \mathop{ \sup }\limits_{\xi \in \Xi_{\sigma ^2}}  \mathop{ \sup }\limits_{\mathcal{C}} \mathbbm{E}_{ \mu  }\left[ \left\Vert \hat{\mu }(\mathcal{C}(\tilde{X})) - \mu  \right\Vert ^2 \right]  \gtrsim \varepsilon \sigma ^2 \wedge d^2.
    \end{align*}
    
\end{proof}

\section{Proof for Bounded Case}\label{section:proof_upper_bound_bounded_case}

\subsection*{Proof of \autoref{theorem:two_point_prob_bound}}\label{subsection:proof_two_point_prob_bound}

The proof is divided into two parts: the trimmed mean part and the median part, corresponding to the two cases $ \frac{\delta ^2}{\sigma ^2} \leq\thenewboundary $ and $ \frac{\delta ^2}{\sigma ^2} \geq \thenewboundary $, respectively.

We first verify that the discriminant quantity $ \tilde{V}_i $, as defined in \autoref{equation:definition_V_i}, has finite variance and establish a useful property in \autoref{appendix_lemma:V_i_finite_variance}. We then demonstrate the existence of appropriate constants in \autoref{lemma:useful_constants_1}. Subsequently, we derive the probability tail bounds for the two cases in \autoref{appendix_lemma:trimmed_mean_part_proof} and \autoref{appendix_lemma:median_part_proof}, respectively. Finally, we combine the two cases to complete the proof of the theorem.

\begin{lemma}\label{appendix_lemma:V_i_finite_variance}
    Let discriminant quantities $\tilde{V}_i$ be as defined in \autoref{equation:definition_V_i}, in which $\Vert \nu_1-\nu_2 \Vert \geq C\delta$. Then the uncorrupted $ \tilde{V}_i $ has variance $ \sigma _V^2\leq 8\sigma ^2 $. Further if (without loss of generality) the underlying truth is $ \mu \in B(\nu _1,\delta) $, we have
    \begin{align*}
        \tilde{V}_i\leq  2(\tilde{X}_i+W_i-\mu )^\top \hat{\nu} - (C-2)\delta ,
    \end{align*}
    where $ \hat{\nu }:= \frac{\nu_2-\nu_1}{\left\Vert \nu_2-\nu_1 \right\Vert } $. Using the expression we have $m:= \mathbbm{E}(\tilde{V}_i) \leq -(C-2)\delta  $ as a direct consequence.
\end{lemma}

\begin{proof}

    We have
    \begin{align*}
        \tilde{V}_i  &=\dfrac{ \left\Vert \tilde{X}_i+W_i-\nu _1 \right\Vert ^2 - \left\Vert \tilde{X}_i+W_i-\nu _2 \right\Vert ^2  }{ \left\Vert \nu _2-\nu _1 \right\Vert  }\\
        &\mathop{ = }\limits^{(i)} 2(\tilde{X}_i+W_i-\mu )^\top \hat{\nu} + 2(\mu -\nu_1)^\top \hat{\nu} + \dfrac{ \left\Vert \nu _1 \right\Vert ^2 + \left\Vert \nu_2 \right\Vert ^2 + 2\nu _1^\top(\nu _2-\nu _1) }{ \left\Vert \nu _2-\nu _1 \right\Vert  }
    \end{align*}
    where in $ (i) $ we define $ \hat{\nu} = \frac{\nu_2-\nu_1}{\left\Vert \nu_2-\nu_1 \right\Vert } $. Since $ \tilde{X}_i $ and $ W_i $ each have variance $ \sigma ^2 $ and are independent, and $ \left\Vert \hat{\nu } \right\Vert = 1 $, we conclude that $ \tilde{V}_i $ has variance $ \sigma _V^2\leq  8\sigma ^2 $. 

    Under the condition that $ \left\Vert \mu -\nu _1 \right\Vert \leq \delta $, we have
    \begin{align*}
        \tilde{V}_i& \leq  2(\tilde{X}_i+W_i-\mu )^\top \hat{\nu} + \left\Vert \mu -\nu _1 \right\Vert  + \left\Vert \nu _2-\nu _1 \right\Vert\\
        \mathop{ \leq }\limits^{(ii)} &  2(\tilde{X}_i+W_i-\mu )^\top \hat{\nu} +(-1+\frac{2}{C}) \left\Vert \nu _2-\nu _1 \right\Vert \\
        \leq&2(\tilde{X}_i+W_i-\mu )^\top \hat{\nu} - (C-2)\delta 
    \end{align*}
    where in $ (ii) $ we use $ \left\Vert \mu -\nu _1 \right\Vert \leq \delta $.

\end{proof}

\renewcommand{\x}{\frac{ \delta ^2 }{ \sigma_V ^2 } }
\begin{lemma}[Useful Constants]\label{lemma:useful_constants_1}
    There exists positive absolute constants $ C, C_0, C_1, C_2, C_3 $ and $ \alpha \in (0,\frac{1}{2}) $ with the following properties: 
    
    Let $ \sigma >0 $, $ \varepsilon \in (0,1/32) $ be given. Define the function:
    \begin{align*}
        g(t)&:=(\dfrac{ 1 }{ 2 } +t)\log (\dfrac{ 1 }{ 2 } +t) + (\dfrac{ 1 }{ 2 } -t)\log (1-2t).
    \end{align*}
    Denote $ D_1 := 4\sqrt{2}\cdot\sqrt{C_2+C_0^{-1}\log 4}  $; for any $ \delta >0 $, denote $ \varrho := \exp\left[ -C_3\log (1+\frac{\delta ^2}{\sigma ^2}) \right] $. 
    
    The choice of the constants are such that the following properties hold:
    \begin{itemize}[topsep=2pt,itemsep=0pt]
        \item[(i)] \hypertarget{constant_property_1}{}$ D_1+6\sqrt{ \frac{ 128 }{ C_1^2 } + 6D_1^2 } \leq C-2  $;
        \item[(ii)] \hypertarget{constant_property_2}{}$ \thenewboundary\cdot \frac{12\log 4}{C_0} < \frac{1}{8} $; 
        \item[(iii)] \hypertarget{constant_property_3}{}$ \big( 1+ \frac{(C-2)^2}{8} \xi \big)^{-1} \leq \frac{1}{2}\exp[-C_3\log(1+\xi )] $ for $ \xi \geq \thenewboundary $; 
    \end{itemize}
    Further, if $ \frac{\delta ^2}{\sigma ^2} \geq \thenewboundary$, the following holds:
    \begin{itemize}[topsep=2pt,itemsep=0pt]
        \item[(iv)] \hypertarget{constant_property_4}{}$ (1/2-\alpha )\log (1/\varrho ) \geq -2g(\alpha ) $;
        \item[(v)] \hypertarget{constant_property_5}{}$ \varepsilon <\alpha (1-\varrho) $.
    \end{itemize}

    Also, we remark that $C$ can be chosen arbitrarily large without influencing other constants.
\end{lemma}

\begin{proof}

    We can determine the constants and derived quantities as follows: 
    \begin{enumerate}[topsep=2pt,itemsep=0pt]
        \item pick $ \alpha \in (0,\frac{1}{2}) $ such that $
            \alpha \big( 1- \exp\left[ \frac{ 2g(\alpha ) }{ 1/2-\alpha  }  \right] \big) > \frac{1}{16}, $
        which is possible by the properties of $ g $ mentioned in \cite[Lemma B.1 (ix)]{prasadan_information_2024};
        \item pick $ C_3\in (0,1) $;
        \item pick $ C_2 $ such that $ C_3\log (1+\thenewboundary)> \frac{-2g(\alpha )}{1/2-\alpha } $. Here by the property that $ \lim_{x\to 0}\frac{-2g(x)}{1/2-x}=\log 4 $ and $ \lim_{x\to 1/2} \frac{-2g(x)}{1/2-x}= \infty$ with the continuity of $ g $, we have the right-hand side of the inequality $  \frac{-2g(\alpha )}{1/2-\alpha }\in (\log 4,\infty) $.
        Such choice is always possible by some $ C_2 $ small enough;
        \item given the above $ C_2,C_3 $, we can find a line with intercept $ \frac{1}{2} $, say denoted $ \xi \mapsto a_{C_2,C_3}\xi +\frac{1}{2} $ such that $ a_{C_2,C_3}\xi +\frac{1}{2} \geq (1+\xi )^{C_3} $ for all $ \xi \geq \thenewboundary $. We use the subscript $ C_2,C_3 $ to emphasize that the slope $ a_{C_2,C_3} $ depends on the choice of $ C_2,C_3 $. Such slope is possible since $ C_3\in (0,1) $.
        \item pick $ C_0 $ such that $ \thenewboundary\cdot \frac{12\log 4}{C_0} < \frac{1}{8} $;
        \item pick $ C $ and $ C_1 $ such that
        \begin{align*}
            \begin{cases}
                C-2 \geq D_1 + 6\sqrt{ \frac{ 128 }{ C_1^2 } + 6D_1^2 }, \\
                \dfrac{ (C-2)^2 }{ 16 }\geq a_{C_2,C_3},
            \end{cases}
        \end{align*}
        where note that the second condition gives $ \frac{ (C-2)^2 }{ 16 }\xi + \frac{1}{2}\geq a_{C_2,C_3}\xi +\frac{1}{2} \geq (1+\xi )^{C_3} $ for $ \xi \geq \thenewboundary $.
    \end{enumerate}
    
    We remark that the above construction of constants already satisfies the conditions \hyperlink{constant_property_1}{(i)}, \hyperlink{constant_property_2}{(ii)}, \hyperlink{constant_property_3}{(iii)}. In the following we verify that conditions \hyperlink{constant_property_4}{(iv)} and \hyperlink{constant_property_5}{(v)} are satisfied when $ \frac{\delta ^2}{\sigma ^2} \geq \thenewboundary $.

    With the construction of $ C_2 $, we have for $ \frac{\delta ^2}{\sigma ^2} \geq \thenewboundary $:
    \begin{align*}
        \log(1/\varrho  ) &= C_3\log(1+\frac{\delta ^2}{\sigma ^2}) \geq C_3\log(1+\thenewboundary)> \frac{-2g(\alpha )}{1/2-\alpha },
    \end{align*}
    which recovers \hyperlink{constant_property_4}{(iv)}. Further we have
    \begin{align*}
        \alpha (1-\varrho ) \geq \alpha \big(1-\exp\left[ \dfrac{ 2g(\alpha ) }{ 1/2-\alpha  }  \right]\big) > \dfrac{ 1 }{ 16 }> \varepsilon  ,
    \end{align*}
    by the construction of $ \alpha  $. Thus using the above process, all the five conditions are satisfied.

\end{proof}

\begin{lemma}\label{appendix_lemma:trimmed_mean_part_proof}
    \textbf{(Trimmed Mean Part) }There exists positive absolute constants $ C>2, C_0, C_1,C_2 $ such that: if $ \varepsilon \in (0,{1/32}) $, for any $ C\delta  $ separated $ \nu _1,\nu _2\in K $ with $ \mu \in B(\nu _1,\delta ) $ and with $ C_0N^{-1}\vee C_1^2\varepsilon \leq \frac{\delta ^2}{\sigma ^2} \leq \thenewboundary $. We have
    \begin{align*}
        \mathbbm{P}_{\mu\in B(\nu_1,\delta)} \left( \mathrm{ TM }_{\delta _0} (\{V_i\}_{i=1}^{2N}) >0  \right) \leq \exp\left[ -C_2N\log(1+\dfrac{ \delta ^2 }{ \sigma ^2 })  \right] .
    \end{align*}
\end{lemma}

\begin{proof}

    Let $ C, C_0, C_1 ,C_2 $ and $ D_1 $ be the constants from \autoref{lemma:useful_constants_1}. Denote the confidence level $ \delta _0:=\exp\left[ -C_2N\frac{\delta ^2}{\sigma ^2} \right] $; the quantile level $ \tilde{\varepsilon } := 8\varepsilon + 12\log(4/\delta _0)/N= 8\varepsilon +12\frac{ \log (4/\exp\left[ -C_2N\delta ^2/\sigma^2 \right]) }{ N } $. We first  verify the following two conditions so that we can apply \cite[Theorem 1]{lugosi_robust_2019}.

    \begin{itemize}[topsep=2pt,itemsep=0pt]
        \item \textit{The confidence level satisfies $\delta_0 \geq e^{-N/4} $}: This is guaranteed by the condition that $ \frac{\delta ^2}{\sigma ^2} \leq \thenewboundary $ and thus
        \begin{align*}
             \delta_0 = \exp\left[ -C_2N\frac{\delta ^2}{\sigma ^2} \right] &\geq  \exp\left[ -C_2N\cdot \thenewboundary \right] > \exp[-N/4].
        \end{align*}
        \item \textit{The quantile level $ \tilde{\varepsilon } $ is valid}: Indeed we have
        \begin{align*}
            \tilde{\varepsilon }& = 8\varepsilon + 12\dfrac{ \log (4/\exp\left[ -C_2N\frac{\delta ^2}{\sigma ^2} \right]) }{ N }\\
            &= 8\varepsilon + \dfrac{ 12\log 4 }{ N } + 12C_2\frac{\delta ^2}{\sigma ^2}\\
            &\mathop{ \leq }\limits^{(i)}  8\varepsilon +  \frac{\delta ^2}{\sigma ^2}\cdot (\frac{ 12\log 4 }{ C_0 } + 12C_2)\\
            &\mathop{ \leq }\limits^{(ii)} 8\varepsilon + \thenewboundary\cdot (\frac{ 12\log 4 }{ C_0 } + 12C_2)\\
            &\mathop{ \leq }\limits^{(iii)}  8\varepsilon + \dfrac{ 1 }{ 4 }  \\
            &\mathop{ \leq }\limits^{(iv)}  \dfrac{ 1 }{ 2 } ,
        \end{align*}
        where in $ (i) $ we use the condition that $ \frac{\delta ^2}{\sigma ^2} \geq C_0N^{-1} $, in $ (ii) $ we use the condition that $ \frac{\delta ^2}{\sigma ^2} \leq \thenewboundary $, in $ (iii) $ we use the condition \hyperlink{constant_property_2}{(ii)} in \autoref{lemma:useful_constants_1} and in $ (iv) $ we use the condition that $ \varepsilon <1/32 $. Also, from the second line we observe that $\tilde{\varepsilon} >0$ since $C_2>0$.
    \end{itemize}

    With the above regularity conditions on $\delta_0$ and $\tilde{\varepsilon }$ verified, we apply \cite[Theorem 1]{lugosi_robust_2019} to the discriminant quantity $ V_i $. With probability at least $ 1-\delta_0 $, we have
    \begin{align*}
        \left\vert \mathrm{ TM }(\{V_i\}_{i=1}^{2N}) - m  \right\vert \leq &  \underbrace{2\sigma_V \sqrt{\dfrac{ \log (4/\delta _0) }{ N } }}_{\mathrm{ I } } + \underbrace{3\mathcal{E}\big( 4\tilde{\varepsilon } \big)}_{\mathrm{ II } },
    \end{align*}
    where we recall that $ m = \mathbbm{E}_{  }( \tilde{V}_i )$, and $\mathcal{E}(\, \cdot \, ) $ is from \cite{lugosi_robust_2019}. For our bounded second moment setting and with the constants chosen in \autoref{lemma:useful_constants_1}, we have the following bounds for the two terms $ \mathrm{ I } $ and $ \mathrm{ II } $ respectively.
    \begin{itemize}[topsep=2pt,itemsep=0pt]
        \item[I] First we have
        \begin{align}\label{equation:trimmed_mean_bound_log_delta}
                \dfrac{ \log (4/\delta _0) }{ N } &= N^{-1}\log 4 + C_2\frac{\delta ^2}{\sigma ^2}\notag \\
                &\mathop{ \leq }\limits^{(i)}      \frac{\delta ^2}{\sigma ^2}\big( C_0^{-1} \log 4 + C_2 \big)\notag \\
                 &\mathop{ = }\limits^{(ii)} \dfrac{ D_1^2 }{ 32 }\frac{\delta ^2}{\sigma ^2},
        \end{align}
        where $ (i) $ is from the condition that $ \frac{\delta ^2}{\sigma ^2} \geq C_0N^{-1} $; $ (ii) $ is by the definition of $ D_1 $ in \autoref{lemma:useful_constants_1}. Then we have
        \begin{align*}
            2\sigma_V \sqrt{\dfrac{ \log (4/\delta _0) }{ N } }& \mathop{ \leq }\limits^{(i)}  
            4\sqrt{2}\sigma \sqrt{\dfrac{ \log (4/\delta _0) }{ N } }
            \mathop{ \leq }\limits^{(ii)}  D_1\delta ,
        \end{align*}
        where $ (i) $ is by the fact that $ \sigma_V^2\leq 8\sigma ^2 $ as shown in \autoref{appendix_lemma:V_i_finite_variance} and $ (ii) $ is by \autoref{equation:trimmed_mean_bound_log_delta}. 
        \item[II] As remarked in \cite{lugosi_robust_2019}[Theorem 1] , we have for finite variance case that
        \begin{align*}
            3\mathcal{E}\big( 4\tilde{\varepsilon } \big) &\mathop{ \leq }\limits^{(i)}  3 \sigma _V \sqrt{8\tilde{\varepsilon } }
            \mathop{ \leq }\limits^{(ii)}  24\sigma \sqrt{\tilde{\varepsilon } }\\
            &= 24\sigma \sqrt{ 8\varepsilon + 12\frac{ \log (4/\delta _0) }{ N } }\\
            &\mathop{ = }\limits^{(iii)}  24\sigma \sqrt{ 8\varepsilon + \dfrac{ 3D_1^2 }{ 8 }\frac{\delta ^2}{\sigma ^2}}\\
            &\mathop{ \leq }\limits^{(iv)}   24  \sqrt{ \frac{8}{C_1^2}\delta^2 + \dfrac{ 3D_1^2 }{ 8 }\delta^2 }
        \end{align*}
        where $ (i) $ is from their Theorem 1; $ (ii) $ is by the fact that $ \sigma_V^2\leq 8\sigma ^2 $; $ (iii) $ is by \autoref{equation:trimmed_mean_bound_log_delta} and $ (iv) $ is by the condition that $ \frac{\delta ^2}{\sigma ^2} \geq C_1^2\varepsilon  $.
    \end{itemize}

    Putting the above together we have
    \begin{align*}
         {2\sigma_V \sqrt{\dfrac{ \log (4/\delta _0) }{ N } }} + {3\mathcal{E}\big( 4\tilde{\varepsilon } \big)}&\mathop{ \leq }\limits^{(i)}  D_1\delta + 24 \sqrt{ \frac{8}{C_1^2}\delta^2 + \dfrac{ 3D_1^2 }{ 8 }\delta^2 } \\
         &=  \delta \big( D_1 + 6\sqrt{ \dfrac{ 128 }{ C_1^2 } + 6D_1^2 } \big)\\
         &\mathop{ \leq }\limits^{(ii)}  (C-2)\delta ,
    \end{align*}
    where $ (i) $ is by the above discussion of the two terms; $ (ii) $ is by the condition \hyperlink{constant_property_1}{(i)} in \autoref{lemma:useful_constants_1}. 

    Now we obtain the probabilistic tail bound
    \begin{align}\label{equation:reduce_from_mean_to_zero}
        \mathbbm{P}_{\mu\in B(\nu_1,\delta)} \left( \mathrm{ TM }_{\delta _0} (\{V_i\}_{i=1}^{2N}) >0  \right) &\mathop{\leq}^{(i)}  \mathbbm{P}_{\mu\in B(\nu_1,\delta)} \left(\left\vert \mathrm{ TM }(\{V_i\}_{i=1}^{2N}) - m  \right\vert > (C-2)\delta   \right) \\
        &\leq\mathbbm{P}_{\mu\in B(\nu_1,\delta)} \left(\left\vert \mathrm{ TM }(\{V_i\}_{i=1}^{2N}) - m  \right\vert >  3\mathcal{E}\big( 4\tilde{\varepsilon } \big) + 2\sigma_V \sqrt{\dfrac{ \log (4/\delta _0) }{ N } }  \right)\\
        &\mathop{\leq}^{(ii)}\delta _0 =  \exp\left[ -C_2N\frac{\delta ^2}{\sigma ^2} \right] ,
    \end{align}
    where in $(i) $ we use the fact that $m<-(C-2)\delta  $ which is a direct consequence of \autoref{appendix_lemma:V_i_finite_variance}; in $(ii)$ we apply \cite{lugosi_robust_2019}[Theorem 1]. Finally by the relation $\frac{\delta ^2}{\sigma ^2} \geq \log (1+\frac{\delta ^2}{\sigma ^2})$ we have the desired result.

\end{proof}

\begin{lemma}[Median Part]\label{appendix_lemma:median_part_proof}
    There exists positive absolute constants $ C>2, C_4 $ such that: if $ \varepsilon \in (0,{1/32}) $, for any $ C\delta  $ separated $ \nu _1,\nu _2\in K $ with $ \mu \in B(\nu _1,\delta ) $ and $  \thenewboundary \leq \frac{\delta ^2}{\sigma ^2}   $, we have
    \begin{align*}
        \mathbbm{P}_{\mu\in B(\nu_1,\delta)} \left(\mathrm{ Median }\big( \{V_i\}_{i=1}^{2N} \big) \geq 0 \right) \leq \exp\left[ -C_4N\log (1+\frac{\delta ^2}{\sigma ^2})  \right] .
    \end{align*}
\end{lemma}

\begin{proof}

    It suffices to prove for arbitrary $ N' $ that
    \begin{align*}
         \mathbbm{P}_{\mu\in B(\nu_1,\delta)} \left( \mathrm{ card }(  \{i\in[N']:\, \left\Vert X_i+W_i-\nu _1 \right\Vert \geq \left\Vert X_i+W_i-\nu _2 \right\Vert \}) \geq \frac{N'}{2} \right) \leq \exp\left[ -\frac{C_4N'}{2}\log(1+\frac{\delta ^2}{\sigma ^2})  \right] ,
    \end{align*}
    so that the desired statement follows by taking $ N' = 2N $.

    Let the constants $ \alpha,C,\{C_i\}_{i=0}^3,D_1 $ be as in \autoref{lemma:useful_constants_1}, which are the same as in the previous lemma. We begin by establishing the probability tail property of the uncorrupted $ \tilde{V_i} $ and then demonstrate that the error rate of the desired median estimator based on the corrupted $ V_i $ is bounded by its counterpart based on the uncorrupted $ \tilde{V_i} $. We have for $ \tilde{V}_i $ that:
    \begin{align*}
        \mathbbm{P}_{\mu\in B(\nu_1,\delta)}\left( \tilde{V}_i>0   \right)&\mathop{ \leq }\limits^{(i)}  \mathbbm{P}_\mu\left( 2(\tilde{X}_i+W_i - \mu )^\top \hat{\nu} > (C-2)\delta  \right)\\
        &\mathop{ \leq } \limits^{(ii)} \dfrac{ 1 }{ 1+(C-2)^2\frac{\delta ^2}{8\sigma ^2}  }\\
        &\mathop{ \leq }\limits^{(iii)}  \dfrac{ 1 }{ 2 }\exp\left[ -C_3\log(1+\frac{\delta ^2}{\sigma ^2} ) \right] \\
        &\mathop{ := }\limits^{(iv)}  \dfrac{ 1 }{ 2 }\varrho,
    \end{align*}
    where $ (i) $ follows from \autoref{appendix_lemma:V_i_finite_variance} where recall that $\hat{\nu}=\frac{\nu_2-\nu_1}{\Vert \nu_2-\nu_1\Vert}$; $ (ii) $ using the Cantelli's inequality; $ (iii) $ is by the condition \hyperlink{constant_property_3}{(iii)} in \autoref{lemma:useful_constants_1} by taking $ \xi = \frac{\delta ^2}{\sigma ^2} $ and noticing that $ \xi \geq \thenewboundary $; $ (iv) $ is by the definition of $\varrho$ in \autoref{lemma:useful_constants_1}.
    
    Next, we apply the method from the second case of the Gaussian setting's testing result (binomial concentration) in \cite[proof of Theorem 3.6]{prasadan_information_2024}. The details are as follows: Define the events
    \begin{align*}
        A_i&:= \{\Vert \tilde{X}_i+W_i - \nu _1 \Vert \geq \Vert \tilde{X}_i+W_i - \nu _2 \Vert \}, \\
        B_i&:= \{\left\Vert X_i+W_i - \nu _1 \right\Vert \geq \left\Vert X_i+W_i - \nu _2 \right\Vert \} ,
    \end{align*}
    where $ \tilde{X}_i $ denotes the uncorrupted samples while $ X_i $ denotes the possibly corrupted samples. Using the notation we directly have $\mathbbm{P}_{\mu\in B(\nu_i,\delta)}(A_i)\leq \varrho/2$ by the above argument. Then we define
    \begin{align*}
        \psi &:= \mathbbm{1}(\text{at least }\frac{N'}{2}\text{ of }B_i\text{ are true})\\
        \tilde{\psi} &:= \mathbbm{1}(\text{at least }\frac{N'}{2}-N'\alpha (1-\varrho)\text{ of }A_i\text{ are true}),
    \end{align*}
    where notice that $ \psi=\psi_\mathrm{ Median } $ is the median estimator in \autoref{definition:median_estimator}.

    Fix any $ \mu $ such that $ \left\Vert \mu -\nu _1 \right\Vert \leq \delta $. For a given $ \delta $ satisfying $ \frac{\delta ^2}{\sigma ^2}\geq \thenewboundary $, we have 
    \begin{align}\label{equation:median_estimator_binomial_concentration}
        \mathbbm{P}\left( \tilde{\psi} =1\right) &\mathop{=}^{(i)} \mathbbm{P}_{  }\left( \mathrm{ Bin }(N', \mathbbm{P}\left( A_i \right) ) \geq  \frac{N'}{2}-N'\alpha (1-\varrho) \right) \notag  \\
        &\mathop{=}^{(ii)} \mathbbm{P}_{  }\left( \mathrm{ Bin }(N', \mathbbm{P}\left( A_i^\complement \right) ) \leq N'(p-\zeta) \right) \notag \\
        &\mathop{\leq}^{(iii)} \mathbbm{P}_{  }\left( \mathrm{ Bin }(N', p ) \leq N'(p-\zeta) \right) \notag \\
        &\mathop{\leq}^{(iv)} \exp\left[ -N'\cdot \mathrm{KL}((p-\zeta)\Vert p)  \right],
    \end{align}
    where $(i)$ since $\mathbbm{1}(A_i)$ are IID Bernoulli random variables so we can express $\tilde{\psi}$ in the form of a binomial random variable; in $(ii)$ we define $ p:= 1-\varrho /2$, $ \zeta = (1/2-\alpha )(1-\varrho) $ for convenience; $(iii)$ we use the stochastic dominance that $\mathrm{Bin}(N',p)\leq_{\mathrm{st}}\mathrm{Bin}(N',\mathbbm{P}(A_i^\complement))$, meaning that we have $\mathbbm{P}(\mathrm{Bin}(N',p) \geq x ) \leq \mathbbm{P}(\mathrm{Bin}(N',\mathbbm{P}(A_i^\complement))\geq x)$ for any $x\in \mathbbm{R}$;  
    $(iv)$ by the Chernoff bound of Binomial random variable and here we denote by $ \mathrm{ KL }(w\Vert z)  $ the Kullback-Leibler divergence between two Bernoulli distribution with parameters $ w $ and $ z $, respectively.

    Now we lower bound $\mathrm{KL}(p-\zeta\Vert p)$ using a fact from \cite{prasadan_information_2024}[Proof of Theorem 3.6] which gives
    \begin{align*}
        \mathrm{KL}(p-\zeta\Vert p) &\geq  g(\alpha )+(\frac{1}{2}-\alpha )\log(1-\varrho).
    \end{align*}
    Substituting the above in \autoref{equation:median_estimator_binomial_concentration} we have
    \begin{align*}
        \mathbbm{P}\left( \tilde{\psi} =1\right) &\leq \exp\left[ -N'\big( g(\alpha )+(\frac{1}{2}-\alpha )\log(1-\varrho) \big)  \right]\\
        &\mathop{\leq}^{(i)} \exp\left[ -N'\cdot \frac{1}{2}(\frac{1}{2}-\alpha )\log (1/\varrho )  \right]\\
        &\mathop{=}^{(ii)} \exp\left[ -\frac{C_4N'}{2}\log (1+\frac{\delta ^2}{\sigma ^2} )  \right],
    \end{align*}
    where $ (i) $ is by the condition \hyperlink{constant_property_3}{(iv)} in \autoref{lemma:useful_constants_1}; $(ii)$ we substitute $\varrho$ and define $ C_4 := (\frac{1}{2}-\alpha )C_3 $.

    Finally, we complete the proof by arguing that $\mathbbm{P}_{\mu\in B(\nu_1,\delta)}( \psi=1 ) \leq \mathbbm{P}_{\mu\in B(\nu_1,\delta)}( \tilde{\psi}=1 )$: The portion of the corrupted samples is at most $\varepsilon$, for which we have $A_i \neq B_i$, while the non-corrupted ones remain $A_i=B_i$. As a result, if $ {\psi}=1 $, then at least $\frac{ N' }{ 2 }- N'\varepsilon  $ of $A_i$ can occur, for which we have the inequality
    \begin{align*}
         \dfrac{ N' }{ 2 }- N'\varepsilon\geq \dfrac{ N' }{ 2 }  -N'\alpha (1-\varrho) 
    \end{align*}
    by \hyperlink{constant_property_5}{(v)} of \autoref{lemma:useful_constants_1}. So consequently we must have $\tilde{\psi}=1$, and the desired bound follows.

\end{proof}
\renewcommand{\x}{\frac{ \delta ^2 }{ \sigma ^2 } }

With the above two lemmas, which are the two components of the robust test $ \psi $ defined in \autoref{definition:robust_estimator}, we can directly obtain the desired two-point testing error bound in \autoref{theorem:two_point_prob_bound}, with $ C_5 = C_2\wedge C_4 $ that only depends on $C_3$ and $\alpha$.

\subsection*{Proof of \autoref{theorem:covering_set_prob_bound}}\label{subsection:proof_covering_set_prob_bound}

Since $ \mathcal{M}=\{\nu _i\}_{i=1}^M $ is a $\delta$-covering of $ K'\ni \mu $, we can without loss of generality assume that $ \nu _1 $ is the point closest to $ \mu $ and thus $ \mu \in B(\nu _1, \delta ) $. Denoting $ T_i:= T(\delta ,\nu _i,\mathcal{M}) $, we have by the definition of $ T(\, \cdot \, ) $ and $i^*$ that
\begin{align*}
    \mathbbm{1}_{\left\Vert \nu _{i^*} -\nu _1 \right\Vert \geq C\delta } \leq \mathbbm{1}_{\max\{ T_{i^*},T_1 \} \geq C\delta } \leq \mathbbm{1}_{T_1 \geq C\delta } .
\end{align*} 
In the case where $ T_1\geq C\delta $, there exists some $ \nu _j $ such that $ \nu _j\succ \nu _1 $. By the condition $ \x\geq C_0N^{-1}\vee C_1^2\varepsilon $, we can apply \autoref{theorem:two_point_prob_bound}. Applying the union bound over $ M-1 $ possible $ \nu _j $, we obtain
\begin{align*}
    \mathbbm{P}\left( \left\Vert \nu _{i^*}-\nu _1 \right\Vert  \geq C\delta  \right) \leq \mathbbm{P}\left( T_1\geq C\delta  \right) \leq (M-1) \exp\left[ -C_5N\log(1+\x)  \right].
\end{align*}
Applying the triangle inequality yields
\begin{align*}
    \mathbbm{P}\left( \left\Vert \nu _{i^*}-\mu  \right\Vert  \geq (C+1)\delta  \right) \leq M \exp\left[ -C_5N\log(1+\x)  \right].
\end{align*}

\subsection*{Proof of \autoref{theorem:expected_risk_upper_bound}}\label{subsection:proof_expected_risk_upper_bound}

In this section, we prove \autoref{theorem:expected_risk_upper_bound}, which establishes the risk upper bound in bounded $K$ case. We first establish the probability tail bound on the error rate of the tree estimator $ \Y_J $ in \autoref{appendix_lemma:prob_bound_on_tree}. We then integrate the probability bound to obtain a bound on the expected risk in \autoref{appendix_lemma:tail_bound_for_interation_and_risk_bound}. Finally, the risk upper bound is obtained by carefully studying when the algorithm ends.

\renewcommand{\x}{\frac{\delta_J^2}{\sigma^2}}
\begin{lemma}\label{appendix_lemma:prob_bound_on_tree}
    Let the constants $ C,C_5 $ be as given in \autoref{theorem:two_point_prob_bound}, $C_0, C_1$ from \autoref{lemma:useful_constants_1}, and set $ c=2C+2 $. For any $J\in \mathbbm{Z}$, denote $\delta_J = \frac{d}{2^{J-1}(C+1)}$. Let $J^{\dagger}\in \mathbbm{N}$ be the maximal $J$ such that \autoref{eq:J_star_condition}:
    \begin{align*}
        C_5N\log(1+ \x) \geq 4\log M^\mathrm{loc}(c\delta_J,2c) \vee \log 2
    \end{align*}
    holds, \textbf{and} additionally $ \frac{\delta_{J}^2}{\sigma^2} \geq C_0N^{-1}\vee C_1^2\varepsilon$. We set $ J^\dagger = 1 $ if the above conditions never hold for any $ J\in \mathbbm{N} $. Assume $C_5N >2$, then for all $ 1\leq J\leq J^\dagger $, we have
    \begin{align*}
        \mathbbm{P}\left( \left\Vert \Y_J-\mu  \right\Vert > (C+1)\delta_J \right) &\leq 2\cdot\mathbbm{1}(J>1) \cdot \exp\left[ -\frac{1}{2}C_5N\log(1+ \x)  \right] ,
    \end{align*}
    where $\Y_J$ is $J^\mathrm{ th }$ element in the tournament winner chain from \autoref{algorithm:robust_algorithm}.

    \textbf{Note:} In the case where the above conditions to $\delta_J$ never hold, we adopt the convention that $ J^\dagger = 1 $ and the desired inequality holds trivially.
\end{lemma}

\begin{proof}

    Define the event of interest by $A_j:= \big\{\left\Vert \Y_j - \mu  \right\Vert > \frac{ d }{ 2^{j-1} }\big\}$. By \cite[Lemma B.4 (i)]{prasadan_information_2024}, we have for $ 1\leq J\leq J^\dagger $ that
    \begin{align*}
        \mathbbm{P}_{  }\left( A_J \right) \leq \underbrace{\mathbbm{P}_{  }\left( A_1 \right) }_{\mathrm{ I }} + \underbrace{\mathbbm{P}_{  }\left( A_2\cap A_1^\complement \right) }_{\mathrm{ II }} + \sum_{j=3}^J \underbrace{\mathbbm{P}_{  }\left( A_j\cap A_{j-1}^\complement \right) }_{\mathrm{ III }},
    \end{align*}
    that is, we decompose the probabilistic tail bound on the $J$-th layer into the tail bound on each layer $j\leq J$ of the tree. Since all $j$ satisfying $1\leq j\leq J\leq J^\dagger $ have $ \frac{\delta_j^2}{\sigma^2} \geq C_0N^{-1}\vee C_1^2\varepsilon $, we can apply the result of \autoref{theorem:two_point_prob_bound} to all such $j$ layers. In the following we first analyze the general case $3\leq j\leq J$ (term III), then address the edge cases $j=1,2$ (terms I and II).

    \begin{itemize}[topsep=2pt,itemsep=0pt]
        \item[\textbf{III}] By our tree construction, for $ 3\leq j \leq J$, we have
        \begin{align*}
            \mathbbm{P}&\left(A_j\cap A_{j-1}^\complement \right)  \\
            &\mathop{\leq}^{(i)} \sum_{u\in \mathcal{L}(j-1)\cap B(\mu , \frac{d}{2^{j-2}})} \mathbbm{P}\left( \left\Vert \Y_j-u \right\Vert  \geq \frac{d}{2^{j-1}} , \Y_{j-1}=u \right) \\
            &\mathop{=}^{(ii)} \sum_{u\in \mathcal{L}(j-1)\cap B(\mu , \frac{d}{2^{j-2}})} \mathbbm{P}\left( \left\Vert \mathop{ \arg\min }\limits_{\nu \in \mathcal{O}(u)} T( \delta_j ,\nu ,\mathcal{O}(u) ) - \mu  \right\Vert \geq (C+1)\delta_j , \Y_{j-1}=u \right) \\
            &\leq \sum_{u\in \mathcal{L}(j-1)\cap B(\mu , \frac{d}{2^{j-2}})} \mathbbm{P}\left( \left\Vert \mathop{ \arg\min }\limits_{\nu \in \mathcal{O}(u)} T( \delta_j ,\nu ,\mathcal{O}(u) ) - \mu  \right\Vert \geq (C+1)\delta_j \right)\\
            &\mathop{\leq}^{(iii)}\sum_{u\in \mathcal{L}(j-1)\cap B(\mu , \frac{d}{2^{j-2}})} M^{\mathrm{ loc } }(\frac{d}{2^{j-2}},2c) \cdot \exp \left[ -C_5N\log(1+\frac{\delta_j^2}{\sigma^2})  \right] \\
            &\mathop{\leq}^{(iv)} \big[M^{\mathrm{ loc } }(\frac{d}{2^{j-2}},2c)\big]^2 \cdot \exp \left[ -C_5N\log(1+\frac{\delta_j^2}{\sigma^2})  \right],
        \end{align*}
        where $(i)$ since we must have $\Y_{j-1}$ being some point $u$ in the discrete set $ \mathcal{L}(j-1)\cap B(\mu , \frac{d}{2^{j-2}})$; $(ii)$ since we choose $\Y_j$ as the winner of the tournament, according to \autoref{definition:tournament_T}; $(iii)$ since $ \mathcal{O}(u) $ is a $ \delta_j = d/2^{j-1}(C+1) $ covering set using \autoref{lemma:tree_properties} and we apply \autoref{theorem:covering_set_prob_bound}; finally $(iv)$ since $ \mathrm{ card }( \mathcal{L}(j-1)\cap B(\mu , \frac{d}{2^{j-2}})) \leq M^\mathrm{ loc }(\frac{d}{2^{j-2}},2c) $ by \autoref{lemma:tree_properties}.
        \item[\textbf{I}] We have $ \mathbbm{P}\left( A_1 \right) = 0 $.
        \item[\textbf{II}] Level 2 is a $ d/c $-covering set of $ K\cap B(\nu ,d) $. Applying \autoref{theorem:covering_set_prob_bound} and repeating the above process, we have
        \begin{align*}
            \mathbbm{P}_{  }\left( A_2\cap A_1^\complement \right) &= \mathbbm{P}_{  }\left( A_2 \right) \\
            &\leq  M^{\mathrm{ loc } }(d,c) \cdot \exp \left[ -C_5N\log(1+\frac{\delta_2^2}{\sigma^2})  \right]\\
             &\mathop{ \leq }\limits^{(i)} \big[M^{\mathrm{ loc } }(d,2c)\big]^2 \cdot \exp \left[ -C_5N\log(1+\frac{\delta_2^2}{\sigma^2})  \right] ,
        \end{align*}
        where $ (i) $ we use the fact that $ M^{\mathrm{ loc } }(d,2c)$ is an integer. The result turns out to be consistent with case $\text{III}$.
    \end{itemize}

    Combining the above results, we have
    \begin{align}\label{equation:applying_summation}
        \mathbbm{P}_{  }\left( \left\Vert \Y_J-\mu  \right\Vert > (C+1)\delta _J \right) &= \mathbbm{P}_{  }\left( A_J \right) = \mathbbm{P}(A_1)+\mathbbm{P}(A_2\cap A_1^\complement) + \sum_{j=3}^J \mathbbm{P}(A_j\cap A_{j-1}^\complement) \notag \\
        &\leq \sum_{j=2}^J \big[M^{\mathrm{ loc } }(\frac{d}{2^{j-2}},2c)\big]^2 \cdot \exp \left[ -C_5N\log(1+ \frac{\delta_j^2}{\sigma^2})  \right] \notag \\
        &\leq  \big[M^{\mathrm{ loc } }(\frac{d}{2^{J-2}},2c)\big]^2 \cdot \sum_{j=2}^J \exp \left[ -C_5N\log(1+ \frac{\delta_j^2}{\sigma^2})  \right] \notag\\
        &\mathop{ \leq }\limits^{(i)}  2\cdot\mathbbm{1}(J>1)[M^{\mathrm{ loc } }(\frac{d}{2^{J-2}},2c)]^2 \cdot \exp \left[ -C_5N\log(1+ \x)  \right] \notag\\
        &\mathop{ \leq }\limits^{(ii)} 2\cdot\mathbbm{1}(J>1) \cdot \exp\left[ -\frac{1}{2}C_5N\log(1+ \x)  \right].
    \end{align}
    Step (i) involves a non-trivial summation of which the details are provided in \autoref{appendix_lemma:sum_of_sequence}. Step $(ii)$ follows since the choice of $ J $ is such that \autoref{eq:J_star_condition} holds.

\end{proof}

\begin{definition}\label{appendix_definition:three_conditions}
    It turns out that the following three conditions are crucial for the probability tail bound to hold, and finally for the risk upper bound. We restate them here and define some related notations for later use in the proof of \autoref{theorem:expected_risk_upper_bound}.

    Let the constants $ C,C_0,C_1,C_5 $ be as in \autoref{appendix_lemma:prob_bound_on_tree}, and set $c=2C+2 $. For any $J\in \mathbbm{Z}$, denote $\delta_J = \frac{d}{2^{J-1}(C+1)}$. We define the following three conditions:
    \begin{itemize}[topsep=2pt,itemsep=0pt]
        \item[\textbf{A}] $ C_5N\log(1+\x) \geq 4\log M^{\mathrm{ loc } }(\delta c,2c) \vee \log 2 $;
        \item[\textbf{B}] $ \x \geq C_0N^{-1}  $;
        \item[\textbf{C}] $\x \geq  C_1^2\varepsilon $.
    \end{itemize}
    We define $J^A$ as the largest $J$ such that condition \textbf{A} holds, and similarly $J^B$ and $J^C$ for conditions \textbf{B} and \textbf{C}, respectively.
\end{definition}

\begin{remark}\label{appendix_remark:three_conditions_A_existence}
    We remark that condition \textbf{A} would always hold when $J$ is sufficiently small and thus some $J^A\in\mathbbm{Z}$ always exists: Recall that as we argued in \autoref{remark:volumn_argument1} that we have $\log M^{\mathrm{ loc } }(\delta ,2c) \lesssim n$ for any $\delta >0$, while the left hand side of condition \textbf{A} can be made arbitrarily large by choosing small $J$. As a result some such $J^A$ always exists. 
\end{remark}

\begin{lemma}\label{appendix_lemma:tail_bound_for_interation_and_risk_bound}
    Let the constants $ C,C_0,C_1,C_5 $ be as in \autoref{appendix_lemma:prob_bound_on_tree}, and set $c=2C+2 $. Let $ J^A, J^B, J^C $ be as defined in \autoref{appendix_definition:three_conditions}. 
    Following these definitions, we have $J^* = J^A\vee 1$ and $J^\dagger = \min\{J^A, J^B, J^C\}\vee 1$ where $J^*$ is from \autoref{theorem:expected_risk_upper_bound} and $J^\dagger$ is from \autoref{appendix_lemma:prob_bound_on_tree}.
    Denote $\nu^* = \Y_{J^*} $, and then $ \nu ^{**} $ be the output after at least $ J^* $ iterations (say, at step $ J^{**}\geq J^* $ so $\nu^{**}=\Y_{J^{**}+1}$). Assume $C_5N>2$, then there exists a constant $ \omega $ depending only on $C$ such that for any $ x\geq \delta _{J^\dagger} = \frac{ d }{ 2^{J^\dagger - 1}(C+1)} $, we have:
    \begin{align*}
        \mathbbm{P}\left( \left\Vert \mu -\nu ^{**}   \right\Vert> \omega x  \right) &\leq 2\cdot\mathbbm{1}(J^*>1) \cdot \exp\left[ -\frac{1}{2}C_5N\log(1+ \dfrac{ x^2 }{ 4\sigma ^2} )  \right].
    \end{align*}
    As a result, for any $ \delta \geq \delta _{J^\dagger} $, it holds that:
    \begin{align}\label{eq:integrated_tail_bound}
        \mathbbm{E}_W\mathbbm{E}_X\left\Vert \mu -\nu ^{**} \right\Vert ^2& \lesssim \delta^2 + \frac{\sigma^2}{N}.
    \end{align}
\end{lemma}

\begin{proof}

    For any given $ x \geq  \delta_{J^\dagger} $, we can find some $ \tilde{J}\leq J^\dagger $ such that $ x \in [\delta_{\tilde{J}},\delta_{\tilde{J}-1}) $. Let $ J = \tilde{J}\vee 1 $ and thus we have $ 1\leq J \leq J^\dagger \leq J^* \leq J^{**} $. Now to bound the tail probability of $ \left\Vert \mu -\nu ^{**} \right\Vert $, we consider the following decomposition: 
    \begin{align*}
        \mu -\nu ^{**} = \mu -\Y_J + \Y_J -\nu ^* + \nu ^* -\nu ^{**} ,
    \end{align*}
    where $ \nu ^* = \Y_{J^*} $ and $ \nu ^{**} = \Y_{J^{**}+1} $. In the decomposition of $ \mu -\nu ^{**} $, we have the following using \autoref{lemma:tree_properties}:
    \begin{equation}\label{eq:distance_relation_between_Js}
        \begin{aligned}
            \left\Vert \Y_J-\nu ^* \right\Vert &\leq  \dfrac{ d(2+4c) }{ c\cdot 2^J }= \dfrac{ 5+4C }{ 1+C }\dfrac{ d }{ 2^{J} } , \\
        \left\Vert \nu ^*-\nu ^{**} \right\Vert &= \left\Vert \Y_{J^*} - \Y_{J^{**}+1} \right\Vert \\
        \leq& \dfrac{ d(2+4c) }{ c\cdot 2^{J^*} }  = \dfrac{ 5+4C }{ 1+C } \dfrac{ d }{ 2^{J^*} } .
        \end{aligned}
    \end{equation}
    
    Now define $\omega:= 6+5C $, and we have
    \begin{align*}
        \mathbbm{P}\left( \left\Vert \mu -\nu ^{**}   \right\Vert> \omega  x  \right) &\leq  \mathbbm{P}_{  }\left( \left\Vert \mu -\Y_J \right\Vert + \left\Vert \Y_J-\nu ^* \right\Vert + \left\Vert \nu ^*-\nu ^{**} \right\Vert > \omega  x  \right) \\
        &\mathop{ \leq }\limits^{(i)}  \mathbbm{P}_{  }\left( \left\Vert \Y_J- \mu \right\Vert +\dfrac{ 5+4C }{ 1+C }\dfrac{ d }{ 2^{J} } + \dfrac{ 5+4C  }{ 1+C } \dfrac{ d }{ 2^{J^*} } > \omega x  \right)\\
        &\mathop{ \leq }\limits^{(ii)}  \mathbbm{P}_{  }\left( \left\Vert \Y_J- \mu \right\Vert > \dfrac{ d }{ 2^{J-1} } \right)\\
        &\mathop{ \leq }\limits^{(iii)}   2\cdot\mathbbm{1}(J>1) \cdot \exp\left[ -\frac{1}{2}C_5N\log(1+ \x)  \right]\\
        &\mathop{ = }\limits^{(iv)}   2\cdot\mathbbm{1}(J>1) \cdot \exp\left[ -\frac{1}{2}C_5N\log(1+ \frac{\delta _{\tilde{J}}^2}{\sigma ^2})  \right]\\
         &\mathop{ \leq }\limits^{(v)} 2\cdot\mathbbm{1}(J>1) \cdot \exp\left[ -\frac{1}{2}C_5N\log(1+ \dfrac{ x^2 }{ 4\sigma ^2 } )  \right],
    \end{align*}
    where $ (i) $ uses the relation between $ \Y_J,\nu ^*,\nu ^{**} $ discussed in \autoref{eq:distance_relation_between_Js}; step $ (iii) $ applies \autoref{appendix_lemma:prob_bound_on_tree}; $ (iv) $ notes that in the case $ J>1 $ we have $ \delta _J=\delta _{\tilde{J}} $ (and if $ J\leq 1 $ the inequality holds trivially since we must have $\mathbbm{P}\left( \left\Vert \mu -\nu ^{**}   \right\Vert> \omega  x  \right)  = 0$); $ (v) $ uses $ x < \delta_{\tilde{J}-1} = 2\delta_{\tilde{J}} $ and the monotonicity of $ t\mapsto \log(1+t) $; $ (ii) $ follows from the definition of $ \omega  $:
    \begin{align*}
        \omega x &= \frac{1}{2}\big(  2+\dfrac{ 5+4C }{ 1+C } + \dfrac{ 5+4C }{ 1+C }    \big)(C+1)x\\
        &\mathop{\geq}^{(i)} \big(  2+\dfrac{ 5+4C }{ 1+C } + \dfrac{ 5+4C }{ 1+C }    \big)\dfrac{ d }{ 2^{J} } \\
        &\mathop{\geq}^{(ii)} \dfrac{ d }{ 2^{J-1} } + \dfrac{ 5+4C }{ 1+C }\dfrac{ d }{ 2^{J} } + \dfrac{ 5+4C  }{ 1+C } \dfrac{ d }{ 2^{J^*} },
    \end{align*}
    where in $(i)$ by the fact that $ x \geq \delta _{\tilde{J}}\geq \delta _{J} $ and in $(ii)$ by $\delta_J\geq \delta_{J^*}$.  

    Note that $ \mathbbm{1}(J>1)\leq \mathbbm{1}(J^*>1) $, and thus we can conclude the desired result:
    \begin{align*}
        \mathbbm{P}\left( \left\Vert \mu -\nu ^{**}   \right\Vert> \omega x  \right) &\leq 2\cdot\mathbbm{1}(J^*>1) \cdot \exp\left[ -\frac{1}{2}C_5N\log(1+ \dfrac{ x^2 }{ 4\sigma ^2} )  \right] ,\qquad x\geq \delta_{J^\dagger}.
    \end{align*}

    Next, we integrate the tail probability over $ x $ to obtain the bound on the expectation: 
    \begin{align*}
        \mathbbm{E}_W\mathbbm{E}_X\left[ \left\Vert \mu -\nu ^{**} \right\Vert ^2 \right] &= \int_0^\infty \mathbbm{P}\left( \left\Vert \mu -\nu ^{**}   \right\Vert ^2 > u  \right) \,\mathrm{d}u \\
        &= 2\omega^2\int_0^\infty  u\cdot \mathbbm{P}\left( \left\Vert \mu -\nu ^{**}   \right\Vert  > \omega  u  \right)\,\mathrm{d}u\\
        &\leq 2\omega ^2\Big[ \int_0^{\delta }  u \,\mathrm{d}u + \int_{\delta }^\infty  u\cdot \mathbbm{P}\left( \left\Vert \mu -\nu ^{**}   \right\Vert > \omega  u  \right) \,\mathrm{d}u\Big] \\
        &\leq \omega ^2\delta ^2 + 4\omega ^2\sigma ^2 \mathbbm{1}(J^*>1) \int_{\delta / \sigma  }^\infty  \dfrac{ u }{ (1+u^2/4)^{C_5N/2-1} }\,\mathrm{d}u\\
        &=  \omega ^2\delta ^2 + \omega ^2\mathbbm{1}(J^*>1)\dfrac{ 8\sigma ^2 }{ C_5N/2-1 }  \cdot \dfrac{ 1 }{ (1+\delta ^2/4\sigma ^2)^{C_5N/2-1} },\quad \delta \geq \delta _{J^\dagger} ,
    \end{align*} 
    and the desired rate upper bound follows.

\end{proof}

\begin{lemma}\label{appendix_lemma:cancel_the_log_sign}
    Let the constants $ C,C_0,C_1,C_5 $ be as in \autoref{appendix_lemma:prob_bound_on_tree}, conditions \textbf{A}, \textbf{B}, \textbf{C} and corresponding $ J^A $, $ J^B $, $ J^C \in \mathbbm{Z} $ be as in \autoref{appendix_definition:three_conditions}. 
    We recall that $ \delta_J = \dfrac{ d }{ 2^J(C+1) } $.
    
    If $ N\gtrsim \mathop{ \sup }\limits_{\delta >0}\log M^\mathrm{ loc }(\delta ,2c) $, then we have:
    \begin{align*}
        N(\dfrac{ \delta _{J^A} }{ \sigma  } )^2 \asymp \log M^\mathrm{ loc }(\frac{\delta _{J^A}}{2}c,2c)\vee 1. 
    \end{align*}
    Similarly, we define $ J^B $ and $ J^C $ as the largest $ J $ such that conditions \textbf{B} and \textbf{C} hold, respectively. We then have $(\dfrac{ \delta _{J^B} }{ \sigma  })^2\asymp  N^{-1}$ and $(\dfrac{ \delta _{J^C} }{ \sigma  })^2 \asymp  \varepsilon$. 
    
    Furthermore, we have the relation $\delta _{J^A}\gtrsim \delta _{J^B} $.    
\end{lemma}

\begin{proof}
    Note by condition \textbf{A} we have
    \begin{align*}
            C_5\log(1+ (\dfrac{ \delta _{J^A} }{ \sigma  } )^2 ) \lesssim C_5\log(1+ (\dfrac{ \delta _{J^A} }{ 2 \sigma  } )^2 ) \leq \dfrac{ 4\log M^\mathrm{ loc }(\frac{\delta _{J^A}}{2}c,2c)\vee \log 2  }{ N }, 
    \end{align*}
    By the condition $ N\gtrsim \mathop{ \sup }\limits_{\delta >0}\log M^\mathrm{ loc }(\delta ,2c) $ and trivially $N\geq 1$, it follows that:
    \begin{align*}
        C_5\log(1+ (\dfrac{ \delta _{J^A} }{ \sigma  } )^2 ) &\lesssim  \dfrac{ 4\log M^\mathrm{ loc }(\frac{\delta _{J^A}}{2}c,2c)\vee \log 2 }{ \mathop{ \sup }\limits_{\delta >0}\log M^\mathrm{ loc }(\delta ,2c)  \vee 1 } \lesssim 1,
    \end{align*} 
    where, say, the the upper bound constant is $ B $. Then we have $\log(1+ (\dfrac{ \delta _{J^A} }{ \sigma  } )^2 ) \leq  \dfrac{ B }{ C_5 }$. Now we argue the following: notice that the function $ x\mapsto \log(1+x) $ lies above its secant line passing through $(0,0)$ and $(t,\log(1+t))$ on $[0,t]$, meaning that $\log(1+x)\ge \frac{x\log(1+t)}{t}$ for $x\in[0,t]$. Now set $t = \exp(B/C_5)-1$ and $x=\delta_{J^A}^2/\sigma^2$, which we just showed satisfies $x\in[0,t].$ This yields:
    \begin{align*}
        \log (1 + \frac{\delta_{J^A}^2}{\sigma^2} ) \geq \dfrac{ B/C_5 }{ \exp(B/C_5)-1 }\cdot \dfrac{ \delta _{J^A}^{ 2} }{ \sigma^2  }:= \varpi \dfrac{ \delta _{J^A}^{ 2} }{ \sigma^2  }.
    \end{align*}
    where we denote $\varpi:=\dfrac{ \frac{ B }{ C_5 } }{ \exp \frac{ B }{ C_5 } -1 }$. Then using this relation, since the critical $ \delta _{J^A} $ satisfies condition \textbf{A}, we have:
    \begin{align*}
        C_5\varpi N(\dfrac{ \delta _{J^A} }{ \sigma  } )^2 \leq 4\log M^\mathrm{ loc }(\frac{\delta _{J^A}}{2}c,2c)\vee \log 2.
    \end{align*}
    Combining this with the original condition \textbf{A}, which gives 
    \begin{align*}
        C_5N(\frac{ \delta _{J^A} }{ \sigma  } )^2 \geq C_5N\log (1+ (\frac{ \delta _{J^A} }{ \sigma  } )^2 ) \geq 4\log M^\mathrm{ loc }(\frac{\delta _{J^A}}{2}c,2c)\vee\log 2,
    \end{align*}
    we have the desired result:
    \begin{align*}
        N(\dfrac{ \delta _{J^A} }{ \sigma  } )^2 \asymp \log M^\mathrm{ loc }(\frac{\delta _{J^A}}{2}c,2c)\vee 1.
    \end{align*}
    The argument for $ J^B $ and $ J^C $ follows using a similar process since $\frac{\delta_{J^B}^2}{\sigma^2}\geq C_0N^{-1}\geq \frac{\delta_{J^B+1}^2}{\sigma^2}$ while $\delta_{J^B}=2\delta_{J^B+1}$ (and similarly for $J^C$).

    Finally, we prove the relation between $ \delta _{J^A} $ and $ \delta _{J^B} $. We have
    \begin{align*}
        (\dfrac{ \delta _{J^A} }{ \sigma  } )^2 \asymp \dfrac{  \log M^\mathrm{ loc }(\frac{\delta _{J^A}}{2}c,2c)\vee 1 }{ N } \geq N^{-1} \asymp (\dfrac{ \delta _{J^B} }{ \sigma  } )^2.
    \end{align*}

\end{proof}

With the above lemmas, we can now proceed to the proof of \autoref{theorem:expected_risk_upper_bound}. 

\begin{proof}[Proof of \autoref{theorem:expected_risk_upper_bound}.]

    The idea of the proof is as follows: By \autoref{appendix_lemma:tail_bound_for_interation_and_risk_bound}, the risk upper bound holds non-trivially  in the situation that all three conditions \textbf{A}, \textbf{B}, \textbf{C} from \autoref{appendix_definition:three_conditions} hold for some $J\geq 1$. However, using \autoref{algorithm:robust_algorithm}, when we traverse the tournament winner chain $\Y$ and as $ J $ grows, we would eventually have one of the three conditions break at some $ J^\dagger $ (where $J^\dagger = 1$ for the case that some of the conditions are violated at the very beginning). In the following we first focus on the edge case that any of the three conditions are violated at the beginning that $ J=1 $. Then we proceed to that all three conditions hold initially.

    \textbf{Part 1}: 
    For the edge case that any of the three conditions is violated at $ J=1 $, we will prove that the risk upper bound is just the trivial bound $ d^2 $. The discussion is separated into three sub-cases depending on which condition is violated at $ J=1 $. 

    \begin{itemize}[topsep=2pt,itemsep=0pt]
        \item If condition \textbf{A} is violated at $ J=1 $, meaning that $\delta_{J^*}^2 = \delta_{J^A}^2\wedge \delta_1^2 = \delta_1^2 \asymp d^2$ and thus
        \begin{align*}
             \mathbbm{E}_W\mathbbm{E}_X\left[ \left\Vert \mu -\nu ^{**} \right\Vert ^2 \right] \leq d^2 = \max\{ d^2, \varepsilon \sigma ^2 \} \wedge d^2 \asymp \max\{ \delta _{J^*}^2, \varepsilon\sigma ^2 \} \wedge d^2.
        \end{align*}
        \item If condition \textbf{C} is violated at $ J=1 $, meaning that $ d^2\lesssim \varepsilon \sigma ^2 $, so
        \begin{align*}
            \mathbbm{E}_W\mathbbm{E}_X\left[ \left\Vert \mu -\nu ^{**} \right\Vert ^2 \right] \leq d^2 \asymp \max\{ \delta _{J^*}^2, \varepsilon\sigma ^2 \} \wedge d^2.
        \end{align*}
        \item Since we have proved the case that either condition \textbf{A} or \textbf{C} is violated at $ J=1 $, it suffices to discuss the case that condition \textbf{B} is violated while conditions \textbf{A} and \textbf{C} hold at $ J=1 $. In this case we have $ \delta_{J^A}^2\lesssim d^2 \lesssim \delta_{J^B}^2 $ since $J^B<1\leq J^A$, however by \autoref{appendix_lemma:cancel_the_log_sign} we have $\delta_{J^A}^2 \gtrsim \delta_{J^B}^2$. Putting these together we have $ \delta_{J^A}^2 \asymp  d^2\asymp \delta_{1}^2$. Then we have $ \delta_{J^*}^2 = \delta_{J^A}^2 \wedge \delta_1^2 \asymp d^2 $,        
        which gives the desired risk upper bound:       
        \begin{align*}
             \mathbbm{E}_W\mathbbm{E}_X\left[ \left\Vert \mu -\nu ^{**} \right\Vert ^2 \right] \leq d^2 \asymp \max\{ \delta _{J^*}^2, \varepsilon\sigma ^2 \} \wedge d^2.
        \end{align*}        
    \end{itemize}

    \textbf{Part 2}: 
    With the above discussion, we conclude that the desired upper bound would hold if any of the conditions is never satisfied. So in the following part we turn to the case that all three conditions are satisfied at the beginning when $J=1$, so we have $ J^A,\, J^B,\, J^C \geq 1 $. Thus we have $ J^*=J^A $ and $ J^\dagger = \min \{ J^*, J^B, J^C\} $. Moreover, $d^2\gtrsim \max\{ \delta_{J^*}^2, \frac{\sigma^2}{N}, \varepsilon\sigma^2 \}$ using \autoref{appendix_lemma:cancel_the_log_sign} and that $\delta_1^2 \asymp d^2$.
    
    As $ J $ grows, one of the conditions would break at $ J^\dagger = \min\{J^*,\, J^B,\, J^C\} <\infty $ so the final bound would be determined by the relation between $ J^*, J_B, J_C $ (or equivalently, $ \delta _{J^*}^2 $, $ \frac{\sigma ^2}{N} $, and $ \varepsilon \sigma ^2 $, up to constants), according to \autoref{appendix_lemma:tail_bound_for_interation_and_risk_bound}, which gives
    \begin{align*}
        \mathbbm{E}_W\mathbbm{E}_X\left[ \left\Vert \mu -\nu ^{**} \right\Vert ^2 \right] &\lesssim \delta_{J^\dagger}^2 + \dfrac{ \sigma ^2 }{ N } .
    \end{align*}
    We first address the condition $ \frac{1}{2}C_5N > 1 $ so that we can apply \autoref{appendix_lemma:tail_bound_for_interation_and_risk_bound}: Note that in the proof of \autoref{theorem:two_point_prob_bound} we have $C_5$ depend only on $C_3$ and $\alpha$, and in \autoref{lemma:useful_constants_1} we remarked that the value of $C$ does not influence other constants. On the other hand, the local packing number can be made arbitrarily large if we choose a sufficiently large $C$ (a way to see this is to select a segment of length $ \delta $ in $ B(\nu ,\delta )\cap K $ and partition it into small segments). As a result, with large enough $C$, we can always have $N>2/C_5$ hold automatically when assuming $ N\gtrsim \mathop{ \sup }\limits_{\delta >0}\log M^\mathrm{ loc }(\delta ,2c) $, then the lemma can be applied. 

    Now we proceed to the discussion that one of the conditions breaks at some $ J^\dagger \geq 1 $. 
    \begin{itemize}[topsep=2pt,itemsep=0pt]
        \item If condition \textbf{A} is the first to break, we have $ \delta_{J^*}^2=\delta_{J^A}^2=\delta_{J^\dagger}^2\geq \delta_{J^C}^2\asymp \varepsilon \sigma ^2 $, then
        \begin{align*}
            \mathbbm{E}_W\mathbbm{E}_X\left[ \left\Vert \mu -\nu ^{**} \right\Vert ^2 \right] &\lesssim  \delta _{J^*}^2 + \dfrac{ \sigma ^2 }{ N } \mathop{ \lesssim  }\limits^{(i)}  \delta _{J^*}^2 \mathop{ \lesssim  }\limits^{(ii)} \max\{ \delta _{J^*}^2, \varepsilon\sigma ^2 \} ,
        \end{align*} 
        where $(i)$ follows since by \autoref{appendix_lemma:cancel_the_log_sign} we have $ \delta_{J^A}^2 \gtrsim \delta_{J^B}^2 \asymp \sigma^2/N $, and $ \delta_{J^*}^2 = \delta_{J^A}^2$ as argued in the beginning of this case; and $ (ii) $ since $ \varepsilon \sigma ^2 \lesssim \delta _{J^*}^2 $;
        \item If condition \textbf{B} is the first to break, we have $ \delta_{J^B}^2=\delta_{J^\dagger}^2\geq \delta_{J^*}^2\vee \delta_{J^C}^2 $. Thus we obtain $ \delta_{J^B}^2 \geq \delta_{J^C}^2 \asymp \varepsilon \sigma ^2 $; also \autoref{appendix_lemma:cancel_the_log_sign} gives $ \delta_{J^*}\gtrsim \delta_{J^B} $ , so together we have $ \delta_{J^*}^2 \asymp \delta_{J^B}^2 \asymp \sigma^2/N \gtrsim \varepsilon \sigma^2 $. Thus we obtain
        \begin{align*}
            \mathbbm{E}_W\mathbbm{E}_X\left[ \left\Vert \mu -\nu ^{**} \right\Vert ^2 \right] &\lesssim  \delta _{J^B}^2 + \dfrac{ \sigma ^2 }{ N } \mathop{ \asymp  } \max\{ \delta _{J^*}^2, \varepsilon\sigma ^2 \}.
        \end{align*} 
        \item If condition \textbf{C} is the first to break, we have $ \delta_{J^C}^2=\delta_{J^\dagger}^2\geq \delta_{J^*}^2\vee \delta_{J^B}^2 $. Thus $ \delta_{J^C}^2 \geq \delta_{J^B}^2\asymp \sigma^2/N $, and also $ \varepsilon \sigma ^2 \asymp \delta_{J^C}^2 \geq \delta_{J^*}^2$. Hence
        \begin{align*}
            \mathbbm{E}_W\mathbbm{E}_X\left[ \left\Vert \mu -\nu ^{**} \right\Vert ^2 \right]  &\lesssim \delta _{J^C}^2 + \dfrac{ \sigma ^2 }{ N } \mathop{ \lesssim  }  \delta _{J^C}^2\asymp \varepsilon \sigma ^2 \mathop{ \asymp  } \max\{ \delta _{J^*}^2, \varepsilon\sigma ^2 \}.
        \end{align*}
    \end{itemize}
    As we argued in the beginning of this part, we have $ d^2\gtrsim \max\{ \delta _{J^*}^2, \frac{\sigma ^2}{N}, \varepsilon \sigma ^2 \} $. Thus, we have the  upper bound $ \max\{ \delta _{J^*}^2, \varepsilon\sigma ^2 \}\wedge d^2 $. 
\end{proof}

The following lemma is useful for the proof of \autoref{appendix_lemma:prob_bound_on_tree}, which details the  computation justifying step $(i)$ of \autoref{equation:applying_summation}.

\begin{lemma}\label{appendix_lemma:sum_of_sequence}
    Let the setting be the same as in \autoref{appendix_lemma:prob_bound_on_tree}. We have
    \begin{align*}
        \sum_{j=2}^J \exp \left[ -C_5N\log(1+ (\frac{d}{2^{j-1}(C+1)\sigma })^2)  \right] 
        &\leq  2\cdot\mathbbm{1}(J>1)\exp \left[ -C_5N\log(1+ (\frac{d}{2^{J-1}(C+1)\sigma })^2)  \right] .
    \end{align*}   
\end{lemma}

\begin{proof}

    If $ J=1 $, the inequality holds trivially. Now we focus on the case that $ J\geq 2 $.
    Let $ m=(\frac{d}{2^{J-1}(C+1)\sigma })^2 $. We have: 
    \begin{align*}
        \sum_{j=2}^J \exp\left[ -C_5N\log(1+ (\frac{d}{2^{j-1}(C+1)\sigma })^2)  \right]  &= \sum_{l=0}^{J-2}\exp\left[ -C_5N\log (1+m\cdot 4^l \right]\\
        &\leq  \sum_{l=0}^\infty \exp\left[ -C_5N\log(1+ m \cdot 4^l)  \right] \\
        &= \sum_{l=0}^\infty \Big( \dfrac{ 1 }{ 1+m\cdot 4^l }  \Big) ^{C_5N} .
    \end{align*}
    We now verify that this sequence is bounded by a geometric series. Specifically, for any $ l \in \mathbbm{N} $:
    \begin{align*}
        \Big( \dfrac{ 1 }{ 1+m\cdot 4^{l+1} }  \Big) ^{C_5N}&\leq   \Big( \dfrac{ 1 }{ 1+m\cdot 4^{l} }  \Big) ^{C_5N} \cdot \Big( \dfrac{  1+m }{1+4m}  \Big) ^{C_5N} \\
        \Leftrightarrow 1+m\cdot 4^{l} &\leq  (1+m\cdot 4^{l+1}) \cdot \Big( \dfrac{  1+m }{1+4m}  \Big) \\
        \Leftrightarrow m\cdot 4^{l} &\leq  m\cdot 4^{l+1} - (1+m\cdot 4^{l+1}) \cdot \Big(\dfrac{ 3m }{ 4m+1 }   \Big) \\
        \Leftrightarrow (4m+1)\cdot 4^{l}& \geq 1+m\cdot 4^{l+1} \\
        \Leftrightarrow 4m+1& \geq 4m+ \dfrac{ 1 }{ 4^l } ,
    \end{align*} 
    which holds for all $ l\geq 0 $. Therefore, $ \big\{(1+m\cdot 4^{l})^{-C_5N}\}_{l=0}^\infty $ is upper bounded by a geometric series with ratio $ (\frac{ 1+m }{1+4m})^{C_5N}$. Applying the geometric series summation formula:
    \begin{align*}
        \sum_{j=2}^J \exp\left[ -C_5N\log(1+ (\frac{d}{2^{j-1}(C+1)\sigma })^2)  \right] &\leq \Big( \dfrac{ 1 }{ 1+m }  \Big) ^{C_5N} \cdot \sum_{l=0}^\infty \Big( \dfrac{ 1+m }{1+4m}  \Big) ^{C_5Nl} \\
        &\leq \Big( \dfrac{ 1 }{ 1+m }  \Big) ^{C_5N} \cdot \dfrac{ 1 }{ 1-\Big( \dfrac{ 1+m }{1+4m}  \Big) ^{C_5N} }. 
    \end{align*}
    Observing that $ \log(\frac{1+m}{1+4m})\leq \log(\frac{1}{1+m}) \vee -\log 3 $ we have 
    \begin{align*}
        C_5N\log (\frac{1+m}{1+4m})&\leq  C_5N\log(\frac{1}{1+m}) \vee -C_5N\log 3 \\
        &\mathop{ \leq }\limits^{(i)} -\log 2 \vee -2\log 3 = -\log 2 , 
    \end{align*}
    where $(i)$ follows by \autoref{eq:J_star_condition}, and $C_5N>2$ by assumption. Consequently, we obtain the desired inequality:
    \begin{align*}
        \sum_{j=2}^J \exp\left[ -C_5N\log(1+ (\frac{d}{2^{j-1}(C+1)\sigma })^2)  \right] &\leq   \Big( \dfrac{ 1 }{ 1+m }  \Big) ^{C_5N} \cdot \dfrac{ 1 }{ 1-\Big( \dfrac{ 1+m }{1+4m}  \Big) ^{C_5N} } \\
        &\leq  2\cdot \Big( \dfrac{ 1 }{ 1+m }  \Big) ^{C_5N} \\
        &= 2\cdot \exp\left[ -C_5N\log(1+ (\frac{d}{2^{J-1}(C+1)\sigma })^2)  \right].
    \end{align*}
\end{proof}

\subsection*{Proof of  \autoref{theorem:expected_risk_rate_bounded_case}}\label{subsection:proof_expected_risk_rate_bounded_case}

\renewcommand{\x}{\frac{\delta ^{*2}}{\sigma ^2}}
\begin{lemma}\label{appendix_lemma:cancel_the_log_sign_star}
    Let $ \delta ^* $ be as defined in \autoref{theorem:expected_risk_rate_bounded_case}, and let $ D $ be some constant depending only on $ c $. Under the condition $ N\gtrsim \mathop{ \sup }\limits_{\delta >0} \log M^\mathrm{ loc }(\delta ,c) $, then there exists a constant $ \phi_{c,D}\in (0,1) $ such that
    \begin{align*}
        \log (1+D^2 \x) \geq \phi_{c,D} \x,
    \end{align*}
    Furthermore, $ \phi_{c,D} $ is increasing in $ D $ and we have $\phi_{c,D}\xrightarrow[]{D\to\infty}\infty $.   
\end{lemma}

The derivation follows the same approach as in \autoref{appendix_lemma:cancel_the_log_sign} and is therefore omitted. Using the same calculation one can see that $\psi_{c,D} \asymp \log (D^2)$.

\newcommand{\boundaryA}{8\log 2}
\begin{proof}[Proof of \textbf{\autoref{theorem:expected_risk_rate_bounded_case}}.]

    We first consider the trivial case $ \delta ^*=0 $. In this case, $ M^\mathrm{ loc }(\delta ,c)=0 $ for any $ \delta >0 $, which implies that $ K $ is a single point and $ d=0 $. The algorithm output is then the trivial single point, yielding a minimax rate of $ 0 $. In the following, we assume $ \delta ^*>0 $ and divide the proof into two cases based on whether $ N\x \gtrless \boundaryA $.

    \begin{itemize}[topsep=2pt,itemsep=0pt]
        \item[\textbf{I}] If $ N\x\leq  \boundaryA$.
        
        We first argue that in this case $ d\lesssim \delta ^* $. Indeed we have 
            \begin{align*}
                \boundaryA &\geq  N\x = \dfrac{ N }{ 4 }\frac{(2\delta ^{*})^2}{\sigma ^2} \geq \dfrac{ 1 }{ 4 }\log M^\mathrm{ loc }(2\delta ^{*},c) .
            \end{align*}
            Then by the same argument as in \cite[proof of Theorem 3.10, case 3]{prasadan_information_2024}, we obtain $ d\leq 12 \delta ^* $.
        
        \paragraph{Lower bound} Define $ \tilde{\delta }:=  \frac{d}{12\sqrt{2}} $. 
            Choosing $ c $ sufficiently large so that $ \log M^\mathrm{ loc }(\tilde{\delta },c) \geq \boundaryA $ (which is always possible by the same argument as in the proof of \autoref{theorem:expected_risk_upper_bound}), and we have
            \begin{align*}
                \log M^\mathrm{ loc }(\tilde{\delta },c) & \geq \boundaryA \mathop{ \geq }\limits^{(i)}  \boundaryA \vee 2N\frac{ \tilde{\delta }^2 }{ \sigma ^2 } > 4\big( \dfrac{ N\tilde{\delta }^2 }{ 2\sigma ^2 } \vee \log 2  \big),
            \end{align*}
            where $ (i) $ since by definition we have $\tilde{\delta }\leq \delta ^*/\sqrt{2} $ and thus $ \boundaryA \geq  N \x  \geq  2N\frac{ \tilde{\delta }^2 }{ \sigma ^2 }$. Therefore, by \autoref{lemma:packing_lower_bound} we have
            \begin{align*}
                \mathfrak{M}\geq \dfrac{ \tilde{\delta }^2 }{ 8c^2 } \asymp d^2 \asymp \delta ^{*2}\wedge d^2.  
            \end{align*} 
            Combining with \autoref{lemma:diakonikolas_lower_bound} that $\mathfrak{M}\geq \varepsilon \sigma ^2 \wedge d^2 $ we have
            \begin{align*}
                \mathfrak{M}\gtrsim  (d^2\wedge \delta ^{*2}) \vee (\varepsilon \sigma ^2 \wedge d^2) = \max(\delta ^{*2},\varepsilon \sigma ^2) \wedge d^2.
            \end{align*} 

            \paragraph{Upper bound} As noted previously, $ d^2 $ is always a trivial upper bound. Thus using $ d^2\lesssim \delta ^{*2} $ we have
            \begin{align*}
                \mathfrak{M}\leq d^2 \asymp \max(\delta ^{*2},\varepsilon \sigma ^2) \wedge d^2.
            \end{align*}   

        \item[\textbf{II}] If $ N\x > \boundaryA $.
        
        \paragraph{Lower bound} We have
            \begin{align*}
                \log M^\mathrm{ loc }(\frac{\delta ^*}{2},c)&\mathop{ \geq }\limits^{(i)}   \log M^\mathrm{ loc }(\delta ^*,c)\mathop{ \geq }\limits^{(ii)}  N\x \geq 8\log 2 \vee N\x > 4\big( \frac{N(\delta ^*/2)^2}{2\sigma ^2} \vee \log 2 \big)  .
            \end{align*}
            where $ (i) $ by the monotonicity of local metric entropy from \autoref{lemma:monotonicity_of_local_entropy}; $ (ii) $ since by the definition of $\delta^*$ we have
            \begin{align*}
                \log M^{\operatorname{loc}}(\frac{\delta^{\ast}}{2},c)\ge \lim_{\eta\downarrow 0}\log M^{\operatorname{loc}}(\delta^{\ast}-\eta,c)\ge  \lim_{\eta\downarrow 0}N\frac{(\delta^*-\eta)^2}{\sigma^2} = N\frac{\delta^{\star 2}}{\sigma^2}.
            \end{align*}
            Then by \autoref{lemma:packing_lower_bound}, we have
            \begin{align*}
                \mathfrak{M}\geq \dfrac{ (\delta ^*/2)^2 }{ 8c^2 } \gtrsim \delta ^{*2} \geq \delta ^{*2} \wedge d^2.
            \end{align*}
            Combined with \autoref{lemma:diakonikolas_lower_bound}, we have the desired lower bound as before.

            \paragraph{Upper bound} To apply \autoref{theorem:expected_risk_upper_bound} to prove the upper bound, it suffices to show that $\delta^*\gtrsim \delta_{J^*}$ where $ J^* $ is from \autoref{theorem:expected_risk_upper_bound}. We do so using an intermediate quantity $\tilde{\delta} := D\delta^*$ and prove that $\tilde{\delta}\gtrsim \delta_{J^*}$ with $D$ chosen properly.

            We choose $ D $ to satisfy $D\geq \frac{4}{c}\vee \mathop{\inf}\{\tilde{D}:\phi_{c,\tilde{D}} \geq \frac{16}{C_5}\}$, where $ \phi_{c,D} $ is the constant from \autoref{appendix_lemma:cancel_the_log_sign_star}. Such choice of $ D $ is always possible by the property of $ \phi_{c,D} $ as established in the lemma. We then have
        \begin{align*}
            C_5N\log (1+\frac{\tilde{\delta }^2}{\sigma ^2}) &=  C_5N\log (1+D^2\frac{\delta ^{*2}}{\sigma ^2}) \\
            &\mathop{ \geq }\limits^{(i)}  C_5N\phi_{c,D}\frac{\delta ^{*2}}{\sigma ^2} \\
            &\mathop{ \geq  }\limits^{(ii)}  4N \frac{(2\delta ^{*})^2}{\sigma ^2} \\
            &\mathop{ \geq }\limits^{(iii)}  4\log M^\mathrm{ loc }(2\delta ^{*},c) \\
            &\mathop{ \geq }\limits^{(iv)}  4\log M^\mathrm{ loc }(\frac{c}{2}D\delta ^*,c)\\
            &= 4\log M^\mathrm{ loc }(\frac{c}{2}\tilde{\delta },c) ,
        \end{align*}
        where $(i)$ follows from \autoref{appendix_lemma:cancel_the_log_sign_star}, $(ii)$ from the choice of $ D $, $(iii)$ from the definition of $ \delta ^*$, and $(iv)$ from the monotonicity of $ M^\mathrm{ loc }(\delta ,c) $ in $ \delta $ established in \autoref{lemma:monotonicity_of_local_entropy}. Comparing the left-hand side and (ii), we also have
        \begin{align*}
            C_5N\log (1+\frac{\tilde{\delta }^2}{\sigma ^2}) &\geq  4N\frac{(2\delta ^{*})^2}{\sigma ^2}\geq 16\cdot 8\log 2 > \log 2,
        \end{align*} 
        Therefore, $ \tilde{\delta } $ satisfies \autoref{eq:J_star_condition_modified}, and by the monotonicity of $ \delta \mapsto C_5N\log(1+\frac{\delta ^2}{\sigma ^2}) - 4\log M^\mathrm{ loc }(\frac{c}{2}\delta ,c) \vee\log 2 $, we obtain that
        \begin{align*}
             \delta ^*\asymp \tilde{\delta }\mathop{>}^{(i)} \delta _{J^*+1} \asymp \delta _{J^*}，
        \end{align*} 
        where $(i)$ since $J^*$ is the greatest $J \geq 1$ such that \autoref{eq:J_star_condition_modified} holds (and in the edge case that $J^*=1$ the inequality still holds).
        Then by \autoref{theorem:expected_risk_upper_bound} we have
        \begin{align*}
            \mathfrak{M}\lesssim \max\{ \delta_{J^*}^2 , \varepsilon \sigma ^2\} \wedge d^2 \lesssim \max\{ \delta^{\ast 2} , \varepsilon \sigma ^2\} \wedge d^2.
        \end{align*}
    \end{itemize}
    
\end{proof}

\section{Proof for Unbounded Case}\label{section:proof_unbounded_case}

\subsection*{Proof of \autoref{theorem:random_set_S_prob_bound}}\label{subsection:proof_random_set_S_prob_bound}

We begin by establishing a fundamental tail bound that governs the concentration of the uncorrupted data around the true mean.

\renewcommand{\x}{\frac{R^2}{\sigma ^2}}   
\begin{lemma}\label{lemma:first_step_one_sample_prob_bound}
    For any $ R>0 $ we have 
    \begin{align*}
        \mathbbm{P}_{  }\left( \left\Vert \tilde{X}_i - \mu  \right\Vert \geq R \right) &\leq 5^n \exp\left[ -\frac{1}{4}\log(1+\x) \right].
    \end{align*}
\end{lemma}

\begin{proof}

    For convenience, let $ Y_i:= \tilde{X}_i - \mu $, which has mean zero and variance bounded by $ \sigma ^2 $.
    For any $ t\in (0,1) $, we construct a maximal $ t $-packing set $ \mathcal{M}(t,\mathbbm{S}^{n-1}) $ of the unit sphere $ \mathbbm{S}^{n-1} $. By \cite[Corollary 4.2.13]{vershynin_high-dimensional_2018}, we have $ \mathrm{card}\big(\mathcal{M}(t,\mathbbm{S}^{n-1}) \big)\leq (1+\frac{2}{t})^n $. For any $ v \in \mathcal{M}(t,\mathbbm{S}^{n-1}) $ we have by Cantelli's inequality that
    \begin{align*}
        \mathbbm{P}_{  }\left( v^\top Y_i > R \right) \leq \dfrac{ \sigma ^2 }{ \sigma ^2 + R^2 } = \exp\left[ -\log (1+\x) \right] ,
    \end{align*}
    Applying a union bound over all $ v\in \mathcal{M}(t,\mathbbm{S}^{n-1}) $, we obtain
    \begin{align*}
        \mathbbm{P}_{  }\left( \mathop{ \sup }\limits_{v\in M(t,\mathbbm{S}^{n-1}) } v^\top Y_i > R  \right) \leq (1+\frac{2}{t})^n \exp\left[ -\log (1+\x) \right]  .
    \end{align*}

    Since $ \mathcal{M}(t,\mathbbm{S}^{n-1}) $ is a maximal $ t $-packing set, for any unit vector $ \frac{Y_i}{\left\Vert Y_i \right\Vert } \in \mathbbm{S}^{n-1} $, there exists $ v\in \mathcal{M}(t,\mathbbm{S}^{n-1}) $ such that $ \left\Vert v - \frac{Y_i}{\left\Vert Y_i \right\Vert } \right\Vert < t $. We then have
    \begin{align*}
        \left\vert v^\top Y_i - \left\Vert Y_i \right\Vert  \right\vert  &= \left\Vert Y_i \right\Vert \left\vert v^\top\big( \dfrac{ Y_i }{ \left\Vert Y_i \right\Vert  }  - v \big) \right\vert\leq \left\Vert Y_i \right\Vert \left\Vert \frac{Y_i}{\left\Vert Y_i \right\Vert } - v \right\Vert \leq t\left\Vert Y_i \right\Vert   .
    \end{align*} 
    Rearranging the formula we have $(1-t)\left\Vert Y_i \right\Vert \leq v^\top Y_i $. Consequently,
    \begin{align*}
        \mathbbm{P}_{  }\left( \left\Vert Y_i \right\Vert >R \right) &= \mathbbm{P}_{  }\left( (1-t)\left\Vert Y_i \right\Vert >(1-t)R \right) \\
        &\leq \mathbbm{P}_{  }\left( \exists\, v \in \mathcal{M}(t,\mathbbm{S}^{n-1}): v^\top Y_i > (1-t)R \right) \\
        &=  \mathbbm{P}_{  }\left( \mathop{ \max }\limits_{v\in \mathcal{M}(t,\mathbbm{S}^{n-1}) } v^\top Y_i > (1-t)R  \right) \\
        &\leq  (1+\frac{2}{t})^n \exp\left[ -\log (1+(1-t)^2\x) \right].
    \end{align*} 
    Setting $ t = \frac{1}{2} $ yields
    \begin{align*}
        \mathbbm{P}_{  }\left( \left\Vert \tilde{X}_i - \mu  \right\Vert > R \right) &\leq  5^n\exp\left[ -\log (1+ \frac{R^2}{4\sigma ^2}) \right] \leq 5^n \exp\left[ -\frac{1}{4}\log(1+\x) \right].
    \end{align*} 
\end{proof}

Now we state our choice of some useful constants. We choose a constant $ \gamma \in (0,1) $ satisfying
\begin{align}\label{eq:condition_for_gamma}
    \gamma \geq \max\big\{ 1- \dfrac{ C_5 }{ 6\log 2 }, 1- \dfrac{ 4\log 5 }{ (C+1)^2C_0 }, 1-\dfrac{ 8\log 5 }{ (C+1)^2C_1^2 } \}  ,
\end{align}
where $ C>2,C_0,C_1,C_5 $ are the constants from \autoref{lemma:useful_constants_1} and we recall that $ C>2 $ is chosen sufficiently large. We then select $ R_0 $ such that
\begin{align}\label{eq:condition_for_R}
    \log(1+\frac{R_0^2}{\sigma ^2}) &= \max \big\{ \dfrac{(4\log 5 \vee (\frac{2C+1}{4})^2) n }{ 1-\gamma }, \dfrac{ -16g(\varepsilon ) }{ \gamma (1/2-\varepsilon ) } ,\dfrac{  16\log 2 }{ \gamma(1/2-\varepsilon ) }, \log(1+\frac{(2C+1)^2}{\sigma^2}) \big\} ,
\end{align}
where $ g(\, \cdot \, ) $ is defined in \autoref{lemma:useful_constants_1}, and we recall that we have $\frac{-2g(\varepsilon)}{\gamma(1/2-\varepsilon)} > \log 4 $ as argued in \autoref{lemma:useful_constants_1} so the second component is also non-trivial. Note that determining $ R_0 $ requires knowledge of the parameters $ n,\sigma,\varepsilon $. By selecting $ R\geq  R_0 $, the random set $ S=S(R) $ from \autoref{definition:random_set_S} captures the true mean $ \mu $ with high probability, as formalized \autoref{theorem:random_set_S_prob_bound}. We will explain the details about these conditions on the constants later in the proof of the theorem.

\begin{proof}[Proof of \textbf{\autoref{theorem:random_set_S_prob_bound}}.]

    We use the notation $ E_{\mu ,i} $ from \autoref{definition:random_set_S} (as well as its counterpart $ \tilde{E}_{\mu ,i} $ for uncorrupted data), so that $ S= \{\nu \in K:\, \sum_{i=1}^N \mathbbm{1}(E_{\nu ,i}) \leq N/2-1 \} $. Using this notation, we have
    \begin{align*}
        \mathbbm{P}_{  }\left( \tilde{E}_{\mu ,i} \right)&  =\mathbbm{P}_{  }\left( \left\Vert \tilde{X}-\mu  \right\Vert > R \right) \\
        &\mathop{ \leq }\limits^{(i)}  5^n \exp\left[ -\frac{1}{4}\log(1+\x) \right] \\
        &\mathop{ \leq }\limits^{(ii)}  \exp\left[ -\frac{\gamma }{4}\log(1+\x) \right] \\
        &\mathop{ \leq }\limits^{(iii)}  \frac{1}{2}\exp\left[ -\frac{\gamma }{8}\log(1+\x) \right] := \varrho /2,
    \end{align*} 
    where $(i)$ follows from \autoref{lemma:first_step_one_sample_prob_bound}; $(ii)$ from the first component of \autoref{eq:condition_for_R}; in $(iii)$ we have
    \begin{align*}
        \log(1+\x) \geq \frac{-16g(\varepsilon)}{\gamma (1/2-\varepsilon)} = \frac{8}{\gamma}\cdot \frac{-2 g(\varepsilon)}{(1/2-\varepsilon)} \geq \frac{8}{\gamma}\cdot \log 4 > \frac{8\log 2}{\gamma},
    \end{align*}
    using the second component of \autoref{eq:condition_for_R} and the fact that $\frac{-2g(\varepsilon)}{1/2-\varepsilon}\geq \log 4$.
    For convenience, we define $\varrho:=\exp\left[ -\frac{\gamma }{8}\log(1+\x) \right] $. 

    We now turn to the event $ \{\mu \not\in S\} $ by collecting the events $ E_{\mu ,i} $ for all $ i\in [N] $ (and as before, we use $\tilde{E}_{\nu,i}$ for the uncorrupted data version). Since the data contains at most $ \varepsilon $ fraction of corrupted samples, we have 
    \begin{align*}
        \mathbbm{P}_{  }\left( \mu \not\in S \right) &\leq  \mathbbm{P}_{  }\left( \text{more than }N/2\text{ of }E_{\mu ,i}\text{ occur} \right)  \\
        &\leq  \mathbbm{P}_{  }\left( \text{more than }N(\frac{1}{2}-\varepsilon )\text{ of }\tilde{E}_{\mu ,i}\text{ occur} \right)  \\
        &\mathop{ = }\limits^{(i)}  \mathbbm{P}_{  }\left( \mathrm{ Bin }\big(N, \mathbbm{P}_{  }\left( \tilde{E}_{\mu ,i} \right) \big) \geq N(\frac{1}{2}-\varepsilon )  \right) \\
        &= \mathbbm{P}_{  }\left( \mathrm{ Bin }\big(N, 1-\mathbbm{P}_{  }\left( \tilde{E}_{\mu ,i} \right) \big) \leq N(\frac{1}{2}+\varepsilon ) \right) \\
        &\mathop{ \leq }\limits^{(ii)}  \mathbbm{P}_{  }\left( \mathrm{ Bin }\big(N, p \big) \leq N(p-\zeta ) \right) \\
        &\leq \exp\left[ -N\cdot \mathrm{KL}(p-\zeta \Vert p) \right],
    \end{align*}
    where $(i)$ follows since the events $ \tilde{E}_{\mu ,i} $ are independently identically distributed; $(ii)$ by setting $ p:= 1-\varrho / 2 $ and $ \zeta := \frac{1}{2}-\varepsilon -\frac{\varrho }{2}=p-\frac{1}{2}-\varepsilon $, and again we use the stochastic dominance argument $\mathrm{Bin}(N,p)\leq_{\mathrm{st}}\mathrm{Bin}(N,1-\mathbbm{P}(\tilde{E}_{\mu,i}))$ as in $(iii)$ of \autoref{equation:median_estimator_binomial_concentration}. We also recall that $\mathrm{KL}(\,\cdot\,\Vert\,\cdot\,)$ represents the Kullback-Leibler divergence between two Bernoulli distribution, as previously defined in \autoref{appendix_lemma:median_part_proof}. Then following the same argument as in \autoref{appendix_lemma:median_part_proof}, we obtain
    \begin{align*}
        \mathrm{KL}(p-\zeta \Vert p)& \mathop{ \geq }\limits^{(i)}   g(\varepsilon )+(\frac{1}{2} - \varepsilon )\log (1/\varrho)\\
        & \mathop{ \geq }\limits^{(ii)} \frac{1}{2}(\frac{1}{2}-\varepsilon )\log (1/\varrho) ,
    \end{align*}
    where $(i)$ uses the function $ g(\, \cdot \, ) $ as defined in \autoref{lemma:useful_constants_1} and using the same logic as in \cite{prasadan_information_2024}[proof of Lemma 5.3]; $(ii)$ applies the second component of \autoref{eq:condition_for_R}, which gives
    \begin{align*}
        \log(1/\varrho) &= \frac{\gamma }{8}\log(1+\x) > \dfrac{ -2g(\varepsilon ) }{ 1/2-\varepsilon  }. 
    \end{align*}

    Combining these results, we obtain the desired bound:
    \begin{align*}
        \mathbbm{P}_{  }\left( \mu \not\in S \right) &\leq  \exp\left[ -N\cdot \mathrm{KL}(p-\zeta \Vert p) \right] \\
        &\leq  \exp\left[ -N\cdot \frac{1}{2}(\frac{1}{2}-\varepsilon )\log (1/\varrho) \right] \\
        &= \exp\left[ -\dfrac{ N(1/2-\varepsilon )\gamma  }{ 16 } \log(1+\x) \right].
    \end{align*}

    Further, using the third component of \autoref{eq:condition_for_R}, the above probability is less than $1/2$, finishing the proof.
    
\end{proof}

\subsection*{Proof of \autoref{theorem:expected_risk_upper_bound_unbounded}}\label{subsection:proof_expected_risk_rate_bounded_case_unbounded}

We begin by establishing the probability tail bound for the error rate of the tree estimator in \autoref{appendix_lemma:prob_bound_on_tree_unbounded}. We then integrate this result to obtain the expected risk rate for non-empty $ S(R) $ case in \autoref{appendix_lemma:tail_bound_for_interation_and_risk_bound_nonemptyS_unbounded}. Subsequently, we address the case of empty $S(R)$ in \autoref{appendix_lemma:tail_bound_for_interation_and_risk_bound_emptyS_unbounded} and combine these results to derive the final expected risk rate in \autoref{appendix_lemma:total_expectation_for_risk_bound_unbounded}. Finally, we present the proof of \autoref{theorem:expected_risk_upper_bound_unbounded}.

\renewcommand{\x}{\frac{\delta _J^2}{\sigma ^2}}
\begin{lemma}\label{appendix_lemma:prob_bound_on_tree_unbounded}
    Let constants $ C,C_0,C_1,C_5 $ be as given in \autoref{lemma:useful_constants_1} and set $ c=2C+2 $. Take $ R\geq R_0 $ where $ R_0 $ is from \autoref{theorem:random_set_S_prob_bound}. Let $ S=S(R) $ be as defined in \autoref{definition:random_set_S}. For any $ J\in \mathbbm{Z} $, denote $ \delta _J = \frac{d_m}{2^{J-1}(C+1)} $, where $ d_m = 2m+2R = 2R\cdot\frac{c}{c-1} $.

    Let $ J^\dagger\in\mathbbm{N} $ be the largest integer such that $ \delta _{J} $ satisfies \autoref{eq:J_star_condition_unbounded}:
    \begin{align*}
         C_5N\log(1+\frac{\delta _{J}^2}{\sigma ^2}) \geq 6\log M^\mathrm{ loc }(c\delta _{J},c) \vee \log 2 ,
    \end{align*}
    and also $ \frac{\delta _{J^\dagger}^2}{\sigma ^2} \geq C_0N^{-1}\vee C_1^2\varepsilon $. We set $ J^\dagger = 1 $ if the above condition is never satisfied for any $ J\in \mathbbm{N} $. Assume $C_5N>2$, then for all $ 1\leq J \leq J^\dagger $, we have  
    \begin{align*}
        \mathbbm{P}_{  }\left( \left\Vert \Y_J - \mu  \right\Vert >(C+1)\delta_J | S\neq \varnothing \right) &\leq \exp\left[ -\dfrac{ N(1/2-\varepsilon )\gamma  }{ 16 } \log(1+\frac{C^2\delta _J^2}{16\sigma ^2}) \right] \\ &\,\, + 4 \cdot \mathbbm{1}(J>1) \exp\left[ -\frac{C_5N}{2}\log(1+\frac{\delta _J^2}{\sigma ^2}) \right].
    \end{align*}    
\end{lemma}

\begin{proof}
    Let $ A_j:= \{ \left\Vert \Y_j - \mu  \right\Vert > \frac{d_m}{2^{j-1}} \} $, so the event of interest is $ \{A_J|S\neq \varnothing \} $. Using \cite[Lemma B.4 (i)]{prasadan_information_2024}, we have
    \begin{align*}
        \mathbbm{P}_{  }\left( A_J| S\neq \varnothing \right) &\leq  \underbrace{ \mathbbm{P}_{  }\left( A_1 | S\neq \varnothing  \right)  }_\text{I} + \underbrace{ \mathbbm{P}_{  }\left( A_2\cap A_1^\complement | S\neq \varnothing \right)  }_{\text{II}} + \sum_{j=3}^J\underbrace{\mathbbm{P}_{  }\left( A_j\cap A_{j-1}^\complement \cap A_1^\complement | S\neq \varnothing \right) }_\text{III}.
    \end{align*}
    This follows a similar approach to \autoref{appendix_lemma:prob_bound_on_tree} by decomposing the probability bound on $ J $ into tail bounds for each layer $j:\, 1\leq j\leq J $. The key difference is that we now have a non-empty $ A_1 $, and we need to deal with the conditional probability with respect to $ S\neq \varnothing $.
    We first handle the general case $ 3\leq j\leq J $ (part $ \text{III} $) and then turn to the edge cases $ j=1,2 $ (parts $ \text{I} $ and $ \text{II} $). 

    \begin{itemize}[topsep=2pt,itemsep=0pt]
        \item[\textbf{III}] For $ 3\leq j\leq J $, we have that
        \begin{align*}
            \mathbbm{P}&\left( A_j\cap A_{j-1}^\complement \cap A_1^\complement | S\neq \varnothing \right) \\
            &= \frac{1}{\mathbbm{P}_{  }\left( S\neq \varnothing \right)} \mathbbm{P}_{  }\left( \left\Vert \Y_j - \mu  \right\Vert \geq \frac{d_m}{2^{j-1}}, \left\Vert \Y_{j-1}-\mu  \right\Vert \leq \dfrac{ d_m }{ 2^{j-2} } , \left\Vert Y_1-\mu  \right\Vert \leq d_m , S\neq \varnothing \right) \\
            &\mathop{ \leq }\limits^{(i)}  2\mathbbm{P}_{  }\left( \left\Vert \Y_j - \mu  \right\Vert \geq \frac{d_m}{2^{j-1}}, \left\Vert \Y_{j-1}-\mu  \right\Vert \leq \dfrac{ d_m }{ 2^{j-2} } , \left\Vert Y_1-\mu  \right\Vert \leq d_m , S\neq \varnothing \right) \\
            &\mathop{ \leq }\limits^{(ii)} 2 \sum_{\substack{s\in S_m \cap \\ B(\mu ,d_m) }}\sum_{\substack{u\in \mathcal{L}^s(j-1) \cap \\ B(\mu , \frac{d_m}{2^{j-2}})}}     \mathbbm{P}_{  }\left( \left\Vert \Y_j - \mu  \right\Vert > \dfrac{ d_m }{ 2^{j-1} } , \Y_{j-1}=u, \Y_1 = s, S\neq \varnothing \right)\\
            &\mathop{ = }\limits^{(iii)}   2 \sum_{\substack{s\in S_m \cap \\ B(\mu ,d_m) }}\sum_{\substack{u\in \mathcal{L}^s(j-1) \cap \\ B(\mu , \frac{d_m}{2^{j-2}})}}  \mathbbm{P}_{  }\Big( \left\Vert \mathop{ \arg\min }\limits_{\nu \in \mathcal{O}(u)} T(\delta_j ,\nu, \mathcal{O}(u)) -\mu  \right\Vert >(C+1)\delta_j ,\\
            & \hspace{20em} \Y_{j-1}=u, Y_1 = s , S\neq \varnothing \Big) \\
            &\leq  2 \sum_{\substack{s\in S_m \cap \\ B(\mu ,d_m) }}\sum_{\substack{u\in \mathcal{L}^s(j-1) \cap \\ B(\mu , \frac{d_m}{2^{j-2}})}}  \mathbbm{P}_{  }\left( \left\Vert \mathop{ \arg\min }\limits_{\nu \in \mathcal{O}(u)} T(\delta_j ,\nu, \mathcal{O}(u)) -\mu  \right\Vert >(C+1)\delta_j \right)
        \end{align*}
        where $(i)$ follows from \autoref{theorem:random_set_S_prob_bound}; $(ii)$ follows from our construction of the tree $ \mathcal{L}^s(j) $ with root at $ s $, which enables a union bound; and $(iii)$ follows from the construction of $ T(\delta_j ,\nu,S) $ as described in \autoref{theorem:covering_set_prob_bound} and the fact that $ \mathcal{O}(u) $ is a $ \frac{d_m}{2^{j-1}} $ covering set of $ \mathcal{L}^s(j-1) $ using \autoref{lemma:tree_properties_unbounded}.

        For the two summations, we observe the following:
        \begin{align*}
            \mathrm{ card }\big( S_m\cap B(\mu ,d_m) \big) &\mathop{ \leq }\limits^{(i)}  M^\mathrm{ loc }_K (d_m ,\frac{d_m}{m}) =    M^\mathrm{ loc }_K (d_m ,2c)\mathop{ \leq }\limits^{(ii)}  M^\mathrm{ loc }_K (\frac{d_m}{2^{j-2}} ,2c),\\
            \mathrm{ card }\big( \mathcal{L}^s(j-1)\cap B(\mu ,\frac{d_m}{2^{j-2}}) \big) &\mathop{ \leq }\limits^{(iii)}  M^\mathrm{ loc }_K (\frac{d_m}{2^{j-2}} ,2c),
        \end{align*}
        where $(i)$ follows from the maximality of $ M^\mathrm{ loc } $, $(ii)$ from the monotonicity of $ M^\mathrm{ loc }_K (\delta ,c) $ in $ \delta $ as established in \autoref{lemma:monotonicity_of_local_entropy}, and $(iii)$ as stated in \autoref{lemma:tree_properties_unbounded}. Combining these results, we obtain
        \begin{align*}
             \mathbbm{P}&\left( A_j\cap A_{j-1}^\complement \cap A_1^\complement | S\neq \varnothing \right)\\
             &\leq  2 \big[ M^\mathrm{ loc }_K (\frac{d_m}{2^{j-2}} ,2c) \big]^2\cdot \mathbbm{P}_{  }\left( \left\Vert \mathop{ \arg\min }\limits_{\nu \in \mathcal{O}(u)} T(\delta_j ,\nu, \mathcal{O}(u)) -\mu  \right\Vert >(C+1)\delta_j \right)\\
             &\mathop{ \leq }\limits^{(i)}  2 \big[ M^\mathrm{ loc }_K (\frac{d_m}{2^{j-2}} ,2c) \big]^3 \exp\left[ -C_5N\log (1+\frac{\delta _j^2}{\sigma ^2}) \right],
        \end{align*}
        where $(i)$ follows from the union bound as in \autoref{theorem:covering_set_prob_bound}.

        \item[\textbf{I}] Since $ s=\Y_1 $ is chosen as the closest point to $ S(R) $ in the $ 2m $-covering set of $ K\supseteq S(R) $, there always exists some $ s' \in S(R) $ such that $ \left\Vert s' - \Y_1 \right\Vert \leq 2m $.  Moreover, by $\mathrm{Diam}(S(R))\leq 2R$ from \autoref{definition:random_set_S}[property (ii)], for any point $ s''\in S(R) $, we have $ \left\Vert s'-s'' \right\Vert \leq 2R $. Together, these imply $ \left\Vert s'' - \Y_1 \right\Vert \leq 2m + 2R = d_m $, and thus $ S(R)\subseteq B(\Y_1, d_m) $. Now we have
        \begin{align*}
            \mathbbm{P}\left( A_1 | S \neq \varnothing \right) &=  \mathbbm{P}_{  }\left( \mu \not\in B(\Y_1, d_m) | S \neq \varnothing \right) \\
            &\leq  \mathbbm{P}_{  }\left( \mu \not\in S(R) | S \neq \varnothing \right) \\
            &=  1 - \mathbbm{P}_{  }\left( \mu \in S(R) | S \neq \varnothing \right)\\
            &\mathop{\leq}^{(i)} 1 - \mathbbm{P}( \mu\in S(R) )  \\
            &\mathop{=}^{(ii)}\exp\left[ -\dfrac{ N(1/2-\varepsilon )\gamma  }{ 16 } \log(1+\frac{R^2}{\sigma ^2}) \right].
        \end{align*} 
        where in $(i)$ we observe that $\{\mu\in S\}\subset \{S\neq \varnothing\}$; $(ii)$ we apply \autoref{theorem:random_set_S_prob_bound}.
        
        \item[\textbf{II}] Level $ 2 $ is a $ \frac{d_m}{c} $ covering set of $ K\cap B(s,d_m) $. We therefore follow the same argument as in \textbf{III} except for some constant differences:
        \begin{align*}
            \mathbbm{P}&\left( A_2\cap A_1^\complement | S\neq \varnothing \right)  \\
            &= \frac{1}{\mathbbm{P}_{  }\left( S\neq \varnothing \right)} \mathbbm{P}_{  }\left( \left\Vert \Y_2 - \mu  \right\Vert > \frac{d_m}{2}, \left\Vert Y_1-\mu  \right\Vert \leq d_m , S\neq \varnothing  \right)\\
            &\leq  2\sum_{\substack{s\in S_m \cap \\ B(\mu ,d_m) }} \mathbbm{P}_{  }\left( \left\Vert \Y_2 - \mu  \right\Vert \geq d_m/2, \Y_1 = s, S\neq \varnothing \right) \\
            &= 2\sum_{\substack{s\in S_m \cap \\ B(\mu ,d_m) }} \mathbbm{P}_{  }\left( \left\Vert \mathop{ \arg\min }\limits_{\nu \in \mathcal{O}(u)} T( \frac{d_m}{c} ,\nu, \mathcal{O}(u)) -\mu  \right\Vert >(C+1)\frac{d_m}{c} , \Y_1 = s, S\neq \varnothing \right) \\
            &\leq 2\sum_{\substack{s\in S_m \cap \\ B(\mu ,d_m) }} \mathbbm{P}_{  }\left( \left\Vert \mathop{ \arg\min }\limits_{\nu \in \mathcal{O}(u)} T( \frac{d_m}{c} ,\nu, \mathcal{O}(u)) -\mu  \right\Vert >(C+1)\frac{d_m}{c} \right).
        \end{align*}
        Similarly, we have 
        \begin{align*}
            \mathrm{ card }\big( S_m \cap B(\mu ,d_m) \big)  &\leq  M^\mathrm{ loc }_K (d_m ,\frac{d_m}{m}) \leq M^\mathrm{ loc }_K (d_m,2c).
        \end{align*}
        Therefore, by \autoref{theorem:covering_set_prob_bound}, 
        \begin{align*}
             \mathbbm{P}\left( A_2\cap A_1^\complement | S\neq \varnothing \right) &\leq  2M^\mathrm{ loc }_K (d_m,2c)  \mathbbm{P}_{  }\left( \left\Vert \mathop{ \arg\min }\limits_{\nu \in \mathcal{O}(u)} T( \frac{d_m}{c} ,\nu, \mathcal{O}(u)) -\mu  \right\Vert >(C+1)\frac{d_m}{c} \right)\\
             &\leq  2\big[ M^\mathrm{ loc }_K (d_m,2c)  \big]^2 \exp\left[ -C_5N\log (1+\frac{d_m^2}{c^2\sigma ^2}) \right] \\
             &\leq  2\big[ M^\mathrm{ loc }_K (d_m,2c)  \big]^3 \exp\left[ -C_5N\log (1+\frac{d_m^2}{c^2\sigma ^2}) \right] \\
             &=  2\big[ M^\mathrm{ loc }_K (d_m,2c)  \big]^3 \exp\left[ -C_5N\log (1+\frac{\delta _2^2}{\sigma ^2}) \right].
        \end{align*}
        We note that this bound is consistent with that obtained in \textbf{III}.
    \end{itemize}

    Putting the above together we have
    \begin{align*}
        \mathbbm{P}_{  }\left( A_J | S\neq \varnothing \right) &\leq  \exp\left[ -\dfrac{ N(1/2-\varepsilon )\gamma  }{ 16 } \log(1+\frac{R^2}{\sigma ^2}) \right] \\
        & + 2 \sum_{j=2}^J \big[ M^\mathrm{ loc }_K (\frac{d_m}{2^{j-2}} ,2c)  \big]^3 \exp\left[ -C_5N\log (1+\frac{\delta _j^2}{\sigma ^2}) \right] \\
        &\leq  \exp\left[ -\dfrac{ N(1/2-\varepsilon )\gamma  }{ 16 } \log(1+\frac{R^2}{\sigma ^2}) \right] \\
        & + 2 \big[ M^\mathrm{ loc }_K (\frac{d_m}{2^{J-2}} ,2c)  \big]^3\sum_{j=2}^J  \exp\left[ -C_5N\log (1+\frac{\delta _j^2}{\sigma ^2}) \right] \\
        &\mathop{ \leq }\limits^{(i)}   \exp\left[ -\dfrac{ N(1/2-\varepsilon )\gamma  }{ 16 } \log(1+\frac{R^2}{\sigma ^2}) \right]\\
        & + 4 \cdot \mathbbm{1}(J>1) \big[ M^\mathrm{ loc }_K (\frac{d_m}{2^{J-2}} ,2c)  \big]^3 \cdot \exp\left[ -C_5N\log(1+\frac{\delta _J^2}{\sigma ^2}) \right]\\
        &\mathop{ \leq }\limits^{(ii)}  \exp\left[ -\dfrac{ N(1/2-\varepsilon )\gamma  }{ 16 } \log(1+\frac{R^2}{\sigma ^2}) \right] + 4 \cdot \mathbbm{1}(J>1) \cdot \exp\left[ -\frac{1}{2}C_5N\log(1+\frac{\delta _J^2}{\sigma ^2}) \right],
    \end{align*}
    where $ (i) $ follows the same steps are in the proof of \autoref{appendix_lemma:prob_bound_on_tree}, which utilizes the summation from \autoref{appendix_lemma:sum_of_sequence}; $ (ii) $ uses \autoref{eq:J_star_condition_unbounded}.

    We finalize the proof by noting the relation between $ \delta _J $ and $ R $: 
    \begin{align*} 
        \delta _J&= \frac{d_m}{2^{J-1}(C+1)} = \dfrac{ R\cdot \frac{2c}{c-1} }{ 2^{J-1}(C+1) }= \dfrac{ 4R }{ 2^{J-1}(2C+1) }\leq \dfrac{ 4R }{ C }   .
    \end{align*}
    This implies
    \begin{align*}
         \exp\left[ -\dfrac{ N(1/2-\varepsilon )\gamma  }{ 16 } \log(1+\frac{R^2}{\sigma ^2}) \right] &\leq  \exp\left[ -\dfrac{ N(1/2-\varepsilon )\gamma  }{ 16 } \log(1+\frac{C^2\delta _J^2}{16\sigma ^2}) \right].
    \end{align*}
    Combining these results, we obtain the desired bound on $\mathbbm{P}_{  }\left( A_J | S\neq \varnothing \right)$.

\end{proof}

\begin{definition}\label{appendix_definition:three_conditions_unbounded}
    Similar to \autoref{appendix_definition:three_conditions} in the bounded case, we define three conditions and their corresponding quantities.

    Let $ C,C_0,C_1,C_5 $ be the constants from \autoref{appendix_lemma:prob_bound_on_tree_unbounded} and set $ c=2C+2 $. For any $ J\in \mathbbm{Z} $, define $ \delta _J = \frac{d_m}{2^{J-1}(C+1)} $. We define the following conditions:
    \begin{itemize}[topsep=2pt,itemsep=0pt]
        \item [\textbf{A}] $ C_5N \log(1+\frac{\delta _J^2}{\sigma ^2}) \geq 6\log M^\mathrm{ loc }(c\delta _J,c) \vee \log 2 $, i.e., \autoref{eq:J_star_condition_unbounded};
        \item [\textbf{B}] $ \frac{\delta _J^2}{\sigma ^2} \geq C_0N^{-1} $;
        \item [\textbf{C}] $ \frac{\delta _J^2}{\sigma ^2} \geq C_1^2\varepsilon $.
    \end{itemize}
    We set $ J^A $ be the largest $ J $ such that condition \textbf{A} holds, and similarly define $ J^B $ and $ J^C $ for conditions \textbf{B} and \textbf{C} respectively. The existence of such $ J^A $ is guaranteed using the same argument as in \autoref{appendix_remark:three_conditions_A_existence}.

    Using the definition, immediately we have $ J^* = J^A\vee 1 $ and $ J^\dagger =\min \{ J^A, J^B, J^C \} \vee 1 $ where $ J^* $ is from \autoref{theorem:expected_risk_upper_bound_unbounded} and $ J^\dagger $ is from \autoref{appendix_lemma:prob_bound_on_tree_unbounded}.
\end{definition}

\begin{lemma}\label{appendix_lemma:tail_bound_for_interation_and_risk_bound_nonemptyS_unbounded}
    Let constants $ C,C_0,C_1,C_5 $ be as in \autoref{appendix_lemma:prob_bound_on_tree_unbounded} and set $ c=2C+2 $. Let $ J^A, J^B, J^C $ be as defined in \autoref{appendix_definition:three_conditions_unbounded}. Denote $ \nu ^*= \Y_{J^*} $, and then $ \nu ^{**} $ to be the output of our tree algorithm after at least $ J^* $ iterations (say at step $ J^{**}\geq J^* $ so $ \nu ^{**} = \Y_{J^{**}+1} $)
    
    Then there exist constants $ \omega, C_8(\varepsilon), C_9(\varepsilon) $ depending only on $ \varepsilon$ and absolute constants such that: if $C_8(\varepsilon)N > 1$, for all $ x \geq \delta _{J^\dagger} = \frac{d_m}{2^{J^\dagger-1}(C+1)} $ it holds that
    \begin{align*}
        \mathbbm{P}\left( \left\Vert \mu -\nu ^{**} \right\Vert > \omega x | S\neq \varnothing  \right)  \leq 7 \exp\left[ -C_8(\varepsilon) N\log(1 + C_9(\varepsilon)\frac{x^2}{\sigma ^2}) \right].
    \end{align*}
    As a result, for any $ \delta \geq \delta _{J^\dagger} $, we have
    \begin{align*}
        \mathbbm{E}_W\mathbbm{E}_X\left[ \left\Vert \mu -\nu ^{**} \right\Vert ^2  \big| S\neq \varnothing  \right] \lesssim \delta ^2 + \frac{\sigma ^2}{N}.
    \end{align*}
    We also note that $ C_8(\varepsilon) $ is chosen such that $ C_8(\varepsilon) \leq \frac{C_5}{2}\wedge \frac{(1/2-\varepsilon )\gamma }{16} $.
\end{lemma}

\begin{proof}

    Following a similar approach to \autoref{appendix_lemma:tail_bound_for_interation_and_risk_bound}, we first bound $ \mathbbm{P}_{  }\left( \left\Vert \mu -\nu ^{**} \right\Vert \gtrsim x | S(R) \neq \varnothing \right) $, then integrate the tail bound to obtain the risk bound. The tail bound is decomposed into two parts depending on whether $ x \gtrsim \delta _1 $ or $ \delta _{J^\dagger} \lesssim  x \lesssim \delta _1 $. We remark that the treatment differs from \autoref{appendix_lemma:tail_bound_for_interation_and_risk_bound} due to the unboundedness of $ K $. In the bounded case, the edge case of $ J < 1 $ (corresponding to large $x$) degenerates since the estimator cannot fall outside the bounded set $ K $. Here in the unbounded case, however, we use the localization set $ S(R) $ to handle the case of large $ x $.
    We first present the two inequalities valid in two regimes respectively, then show how to harmonize them by analyzing the boundary of the two regimes. 

    \begin{itemize}[topsep=2pt,itemsep=0pt]
        \item For $ x\geq hR $ where $h$ is some constant to be determined soon to harmonize the two parts, we have for $ C_7 := \dfrac{ 8c }{ c-1 } = \dfrac{ 16(C+1) }{ 2C+1 } $ that
        \begin{align}\label{equation:the_first_half_for_x}
            \mathbbm{P}_{  }\left( \left\Vert \mu -\nu ^{**} \right\Vert > \frac{C_7x}{h} | S\neq \varnothing \right) \leq 2\exp\left[ -\dfrac{ N(1/2-\varepsilon )\gamma  }{ 16 }\log(1+\frac{x^2}{h^2\sigma ^2})  \right] ,
        \end{align}
        following the same logic as \cite[Lemma D.2 Part I]{prasadan_information_2024} and using \autoref{theorem:random_set_S_prob_bound}.
        \item For $ x\in [\delta _{J^\dagger},\delta_1] $,  we follow the same steps as in \autoref{appendix_lemma:tail_bound_for_interation_and_risk_bound}. Let $ \omega = 6+5C $, and we have
        \begin{align*}
            \mathbbm{P}_{  }\left( \left\Vert \mu -\nu ^{**} \right\Vert > \omega x  | S\neq \varnothing \right) &\leq \exp\left[ -\dfrac{ N(1/2-\varepsilon )\gamma  }{ 16 } \log(1+\frac{C^2x^2}{64\sigma ^2}) \right] \\
            &+ 4 \cdot \mathbbm{1}(J^*>1) \cdot \exp\left[ -\frac{1}{2}C_5N\log(1+\frac{x^2}{4\sigma ^2}) \right]  .
        \end{align*}
    \end{itemize}

    To combine these two bounds, we choose a proper harmonizing constant $h$ so that the boundaries of the two inequalities match. To be specific, we notice the fact that
    \begin{align*}
        \omega \delta _1 &= \frac{12+10C}{2}\cdot \frac{d_m}{C+1} =  \frac{2(12+10C)}{2C+1}R\geq  \dfrac{ 16(C+1) }{ 2C+1 } R = C_7R ,
    \end{align*}
    which means $C_7R/\omega \leq \delta_1$. Therefore, we set $h = C_7/\omega$, which reformulates \autoref{equation:the_first_half_for_x} as
    \begin{align*}
         \mathbbm{P}_{  }\left( \left\Vert \mu - \nu ^{**} \right\Vert > \omega x \big| S\neq \varnothing \right) &\leq  2\exp\left[ -\dfrac{ N(1/2-\varepsilon )\gamma }{ 16 }\log(1+\frac{\omega ^{2}x^2}{C_7^2\sigma ^2})  \right] ,\quad x>\delta_1\geq C_7R/\omega.
    \end{align*} 

    Now we can put the two parts together. For any $x\geq \delta_{J^\dagger}$ it holds that
    \begin{align*}
         \mathbbm{P}_{  }\left( \left\Vert \mu - \nu ^{**} \right\Vert > \omega x \big| S\neq \varnothing \right) &\leq  2\exp\left[ -\dfrac{ N(1/2-\varepsilon )\gamma }{ 16 }\log(1+\frac{\omega ^{2}x^2}{C_7^2\sigma ^2})  \right] \\
         & + \exp\left[ -\dfrac{ N(1/2-\varepsilon )\gamma  }{ 16 } \log(1+\frac{C^2x^2}{64\sigma ^2}) \right] \\ & + 4 \cdot \mathbbm{1}(J^*>1) \cdot \exp\left[ -\frac{1}{2}C_5N\log(1+\frac{x^2}{4\sigma ^2}) \right] \\
         &\leq  7 \exp\left[ -C_8(\varepsilon) N\log(1+C_9(\varepsilon)\frac{x^2}{\sigma ^2}) \right],
    \end{align*} 
    where $ C_8(\varepsilon) ,C_9 (\varepsilon)$ are constants depending only on $ \varepsilon  $ and absolute constants. Without loss of generality, we may assume $ C_8(\varepsilon) \leq \frac{C_5}{2}\wedge \frac{(1/2-\varepsilon )\gamma }{16}$, which ensures the claimed tail bound holds.

    Then we integrate the tail bound to obtain the risk upper bound:
    \begin{align*}
        \mathbbm{ E }\left[ \left\Vert \mu -\nu ^{**} \right\Vert ^2  | S\neq \varnothing\right] &= \int_0^\infty \mathbbm{P}_{  }\left( \left\Vert \mu -\nu ^{**} \right\Vert ^2 > x | S\neq \varnothing \right) \,\mathrm{d}x\\
        &= 2\omega ^2 \int_0^{\infty} u\cdot \mathbbm{P}_{  }\left( \left\Vert \mu -\nu ^{**} \right\Vert > \omega u | S\neq \varnothing   \right) \,\mathrm{d}u\\
        &\leq  2\omega ^2 \int_0^{\delta } u \,\mathrm{d}u + 14\omega ^2 \int_{\delta }^{\infty} u\cdot \exp\left[ -C_8(\varepsilon) N\log(1+C_9(\varepsilon)\frac{u^2}{\sigma ^2}) \right] \,\mathrm{d}u,
    \end{align*}
    where $ \delta \geq \delta _{J^\dagger} $. Then under the condition $ C_8(\varepsilon) N > 1 $, we have 
    \begin{align*}
         \mathbbm{ E }_{ R }\mathbbm{E}_{ X }&\left[ \left\Vert \mu -\nu ^{**} \right\Vert ^2  | S\neq \varnothing\right] \\
         &\lesssim \omega ^2 \cdot \delta^2 + \omega ^2 \cdot \frac{\sigma ^2}{C_9(\varepsilon)(C_8(\varepsilon) N-1)}\cdot \frac{1}{(1+C_9(\varepsilon)\delta ^2/\sigma ^2)^{C_8(\varepsilon) N-1}} \\
         &\lesssim \delta ^2 + \frac{\sigma ^2}{N} ,\quad \delta \geq \delta _{J^\dagger}.
    \end{align*}
        
\end{proof}

In the above, the error bound under $\{S(R)\neq \varnothing\}$ is established. We now turn to the case that $ \{S(R) = \varnothing\} $, where the estimator is the lexicographically smallest point in $S(\hat{R})$ where $\hat{R}=\min\{t>0:S(t)\neq \varnothing\}$. We denote the point $\hat{p}$. Notice that $\{S\neq\varnothing\} = \{\hat{R}\leq R\}$, it suffice to investigate different values of $\hat{R}>R$. The following lemma does so by peeling the space into $\{\hat{R}\geq R\} = \biguplus_{k=1}^\infty \{2^{k-1}R<\hat{R}\leq 2^{k}R\}$ and derive an error bound for each peel. Then, as we will see in \autoref{appendix_lemma:total_expectation_for_risk_bound_unbounded}, the peels are combined with the previous parts using the law of total expectation to produce the final error bound.

\begin{lemma}\label{appendix_lemma:tail_bound_for_interation_and_risk_bound_emptyS_unbounded}
    Let the setting be the same as in \autoref{appendix_lemma:tail_bound_for_interation_and_risk_bound_nonemptyS_unbounded}, and consider the case that $S(R) = \varnothing$ where the estimator is denoted $\hat{p}$ as argued above. For any $ k\in \mathbbm{N} $ such that $ p_k:= \mathbbm{P}_{  }\left( 2^{k-1}R < \hat{R} \leq 2^kR \right) > 0 $ we have
    \begin{align*}
        \mathbbm{E}_{  }\left[ \left\Vert \hat{p}-\mu  \right\Vert ^2 \big| 2^{k-1}R < \hat{R} \leq 2^k R \right] &\lesssim  \sigma^2 \exp\left[ \frac{16}{N(1/2-\varepsilon )\gamma }\log (1/p_k) \right].
    \end{align*} 
\end{lemma}

\begin{proof}

    Let $ p_0 := \mathbbm{P}_{  }\left( S\neq \varnothing \right) = \mathbbm{P}_{  }\left( \hat{R}\leq R \right) $ and $ p_k:= \mathbbm{P}_{  }\left( 2^{k-1}R < \hat{R} \leq 2^{k}R \right) $ for all $ k\in \mathbbm{N} $. We have
    \begin{align}\label{eq:peeling_prob_bound_for_pk}
        p_k &= \mathbbm{P}_{  }\left( \hat{R}\leq 2^k R \right) - \mathbbm{P}_{  }\left( \hat{R} \leq 2^{k-1} R \right) \notag\\ 
        &\mathop{ \leq }\limits^{(i)}   1 - \mathbbm{P}_{  }\left( \mu \in S(2^{k-1}R) \right) \notag\\
        &\mathop{ \leq }\limits^{(ii)}  \exp\left[ - \frac{N(1/2-\varepsilon )\gamma }{16}\log (1+ \frac{(2^{k-1}R)^2}{\sigma ^2}) \right] ,
    \end{align}
    where $(i)$ follows from the fact that $ \{\mu \in S(2^{k-1}R) \} \subseteq \{ S(2^{k-1}R) \neq \varnothing \} \subseteq \{\hat{R}\leq 2^{k-1}R\} $ and $(ii)$ uses \autoref{theorem:random_set_S_prob_bound}. Then for any $ t\geq 2^kR $, we have
    \begin{align*}
        \mathbbm{P}_{  }\left( \left\Vert \hat{p}-\mu  \right\Vert > 2t \big| 2^{k-1}R<\hat{R}\leq 2^{k}R \right) &= \dfrac{ \mathbbm{P}_{  }\left( \left\Vert \hat{p}-\mu  \right\Vert > 2t , 2^{k-1}R<\hat{R}\leq 2^{k}R \right) }{ p_k }  \\
        &\leq  \dfrac{ \mathbbm{P}_{  }\left( \left\Vert \hat{p}-\mu  \right\Vert > 2t , \hat{p}\in S(2^{k}R) \right) }{ p_k }\\
        &\mathop{ \leq }\limits^{(i)}   \dfrac{ \mathbbm{P}_{  }\left( \mu \not\in S(t) \right) }{ p_k }\\
        &\mathop{ \leq }\limits^{(ii)}   p_k^{-1} \exp\left[ -\frac{N(1/2-\varepsilon )\gamma }{16}\log(1+\frac{t^2}{\sigma ^2}) \right],
    \end{align*}
    where $ (i) $ by the nested property and diameter of set $ S(R) $ as argued in \autoref{definition:random_set_S}; $(ii)$ again uses \autoref{theorem:random_set_S_prob_bound}. 

    We observe that a non-trivial bound can be guaranteed by $ t\geq 2\sigma \exp\left[ \frac{8}{N(1/2-\varepsilon )\gamma }\log (1/p_k) \right] =: t_k \geq 2\cdot 2^{k-1}R = 2^{k}R $ where we plug in \autoref{eq:peeling_prob_bound_for_pk}, so that the above quantity is smaller than $ 1 $.

    We can therefore integrate the tail bound as follows:
    \begin{align*}
        \mathbbm{E}_{  }\big[ \left\Vert \hat{p}-\mu  \right\Vert ^2 \big|& 2^{k-1}R < \hat{R} \leq 2^kR \big] =  \int_0^\infty \mathbbm{P}_{  }\left( \left\Vert \hat{p} - \mu  \right\Vert ^2 > u \big| 2^{k-1}R < \hat{R} \leq 2^kR  \right)\,\mathrm{d}u\\
        &= 4\int_0^\infty 2t\cdot \mathbbm{P}_{  }\left( \left\Vert \hat{p}-\mu  \right\Vert > 2t \big| 2^{k-1}R<\hat{R}\leq 2^{k}R \right) \,\mathrm{d}t\\
        &\leq  8\int_0^{2\sigma \exp\left[ \frac{8}{N(1/2-\varepsilon )\gamma }\log (1/p_k) \right]} t\,\mathrm{d}t \\ &\quad +  8\int_{2\sigma \exp\left[ \frac{8}{N(1/2-\varepsilon )\gamma }\log (1/p_k) \right]}^\infty t\cdot p_k^{-1} \exp\left[ -\frac{N(1/2-\varepsilon )\gamma }{16}\log(1+\frac{t^2}{\sigma ^2}) \right] \,\mathrm{d}t\\
        &\lesssim  \sigma^2 \exp\left[ \frac{16}{N(1/2-\varepsilon )\gamma }\log (1/p_k) \right] \\
        &\quad + \frac{\sigma ^2}{N}p_k^{-1}\exp\left[ \big( 1-\frac{N(1/2-\varepsilon )\gamma }{16}\big)\log(1+\frac{t_k^2}{\sigma ^2})  \right]\\
        &\mathop{ \leq }\limits^{(i)}    \sigma^2 \exp\left[ \frac{16}{N(1/2-\varepsilon )\gamma }\log (1/p_k) \right] +  \frac{\sigma ^2}{N}\exp\left[ \frac{16}{N(1/2-\varepsilon )\gamma }\log (1/p_k) \right]\\
        &\lesssim  \sigma^2 \exp\left[ \frac{16}{N(1/2-\varepsilon )\gamma }\log (1/p_k) \right].
    \end{align*}
    where $(i)$ since $  \frac{N(1/2-\varepsilon )\gamma }{16} \geq C_8N >1 $ using \autoref{appendix_lemma:tail_bound_for_interation_and_risk_bound_nonemptyS_unbounded} and the fact that the integration boundary $ t_k = 2\sigma \exp\left[ \frac{8}{N(1/2-\varepsilon )\gamma }\log (1/p_k) \right] $ satisfies $ \log(1+\frac{t_k^2}{\sigma ^2} ) \geq \frac{16}{N(1/2-\varepsilon )\gamma }\log 1/p_k $, as a consequence of the choice of $ t_k $ to make the probability bound in the previous inequality less than $ 1 $.

\end{proof}

\begin{lemma}\label{appendix_lemma:total_expectation_for_risk_bound_unbounded}
    Let the setting be the same as in \autoref{theorem:expected_risk_upper_bound_unbounded}, with constants $ C,C_8 $ from \autoref{appendix_lemma:tail_bound_for_interation_and_risk_bound_nonemptyS_unbounded}. If $ C_8 N > 2 $, then we have
    \begin{align*}
        \mathbbm{E}_{  }\left[ \left\Vert \nu ^{**} - \mu  \right\Vert ^2 \right] &\lesssim  \delta _{J^\dagger}^2 + \frac{\sigma ^2}{N} .
    \end{align*}
\end{lemma}

\begin{proof}

    With \autoref{appendix_lemma:tail_bound_for_interation_and_risk_bound_nonemptyS_unbounded} and \autoref{appendix_lemma:tail_bound_for_interation_and_risk_bound_emptyS_unbounded}, we can apply the law of total expectation to obtain the risk bound. We have
    \begin{align*}
        \mathbbm{E}_{  }\left[ \left\Vert \nu ^{**} - \mu  \right\Vert ^2 \right] &= p_0 \cdot   \mathbbm{E}_{  }\left[ \left\Vert \nu ^{**} - \mu  \right\Vert ^2 \big| S\neq \varnothing \right] + \sum_{k\in \mathbbm{N}, p_k >0} p_k \cdot \mathbbm{E}_{  }\left[ \left\Vert \hat{p}-\mu  \right\Vert ^2 \big| 2^{k-1}R < \hat{R} \leq 2^kR \right] \\
        &\lesssim \underbrace{ \delta _{J^\dagger}^2 + \frac{\sigma ^2}{N}}_{\text{\autoref{appendix_lemma:tail_bound_for_interation_and_risk_bound_nonemptyS_unbounded}}} + \sum_{k\in \mathbbm{N}, p_k >0} p_k \cdot \underbrace{ \sigma^2 \exp\left[ \frac{16}{N(1/2-\varepsilon )\gamma }\log (1/p_k) \right]  }_{\text{\autoref{appendix_lemma:tail_bound_for_interation_and_risk_bound_emptyS_unbounded}}} \\
        &=  \delta _{J^\dagger}^2 + \frac{\sigma ^2}{N} + \sigma ^2 \underbrace{\sum_{k\in \mathbbm{N}, p_k >0}\exp\left[ \log p_k \cdot \big(1-\frac{16}{N(1/2-\varepsilon )\gamma }\big) \right]}_{\star}.
    \end{align*}
    For term $ {\star} $, we have 
    \begin{align*}
        \sum_{k\in \mathbbm{N}, p_k >0}\exp\left[ \log p_k \cdot \big(1-\frac{16}{N(1/2-\varepsilon )\gamma }\big) \right] &\mathop{ \leq }\limits^{(i)}   \sum_{k\in \mathbbm{N}}\exp\left[ -\big(\frac{N(1/2-\varepsilon )\gamma }{16} - 1\big) \cdot \log(1+\frac{(2^{k-1}R)^2}{\sigma ^2})  \right] \\
        &\mathop{ \leq }\limits^{(ii)}  2\cdot \exp\left[ -\big(\frac{N(1/2-\varepsilon )\gamma }{16} - 1\big) \cdot \log(1+\frac{R^2}{\sigma ^2})  \right],
    \end{align*}
    where $(i)$ follows by applying \autoref{eq:peeling_prob_bound_for_pk} to bound $ p_k $ and using $\frac{N(1/2-\varepsilon)\gamma}{16}>C_8N>2$; and $(ii)$ follows from the same summation argument as in \autoref{appendix_lemma:sum_of_sequence}.

    Now together we have
    \begin{align*}
         \mathbbm{E}_{  }\left[ \left\Vert \nu ^{**} - \mu  \right\Vert ^2 \right] & \lesssim \delta_{J^\dagger}^2 + \frac{\sigma ^2}{N} + \sigma ^2\exp\left[ -\big(\frac{N(1/2-\varepsilon )\gamma }{16} - 1\big) \cdot \log(1+\frac{R^2}{\sigma ^2})  \right]\\
         &\mathop{ \lesssim }\limits^{(i)}   \delta_{J^\dagger}^2 + \frac{\sigma ^2}{N} + \sigma ^2 \cdot \exp\left[ -\frac{N(1/2-\varepsilon )\gamma }{32}  \cdot \log(1+\frac{R^2}{\sigma ^2})  \right] \\
         &\mathop{ \leq }\limits^{(ii)}  \delta_{J^\dagger}^2 + \frac{\sigma ^2}{N} + \sigma ^2 \cdot \exp\left[ -N\frac{\log 2}{2} \right] \\
         &\lesssim  \delta_{J^\dagger}^2 + \frac{\sigma ^2}{N},
    \end{align*}
    where $ (i) $ since by \autoref{appendix_lemma:tail_bound_for_interation_and_risk_bound_nonemptyS_unbounded} we have $ \frac{N(1/2-\varepsilon )\gamma }{16} > C_8N > 2 $; $(ii)$ follows from the third component of \autoref{eq:condition_for_R}.
    
\end{proof}

\begin{proof}[Proof of \textbf{\autoref{theorem:expected_risk_upper_bound_unbounded}}.]

    The approach of this proof is similar to that of \autoref{theorem:expected_risk_upper_bound} by first discussing the edge case and then the non-trivial bound. We note that the second part would follow exactly the same discussion as in the proof of \autoref{theorem:expected_risk_upper_bound}[Part II] since we can just remove the trivial $ d^2 $ bound. Thus it suffices to discuss the edge case part at $ J=1 $. Here instead, we argue that conditions \textbf{A}, \textbf{B}, and \textbf{C} from \autoref{appendix_definition:three_conditions_unbounded} must hold at $ J=1 $, that is, $ \delta_1 $ satisfies \autoref{eq:J_star_condition_unbounded} and also $ \frac{\delta_1^2}{\sigma ^2}\geq C_0N^{-1}\vee C_1^2\varepsilon  $.
    
    \begin{itemize}[topsep=2pt,itemsep=0pt]
        \item[\textbf{A}] Observing that $ \delta _{1} = d_m/(C+1) = \frac{4R}{2C+1} $, we have
        \begin{align*}
            C_5N \log(1+\frac{\delta _1^2}{\sigma ^2}) & = C_5N \log(1+(\frac{4}{2C+1})^2\frac{R^2}{\sigma ^2})\\
            &\mathop{ \geq }\limits^{(i)}  C_5N(\frac{4}{2C+1})^2 \log(1+\frac{R^2}{\sigma ^2})\\
            &\mathop{ \geq }\limits^{(ii)}  C_5N \frac{n}{1-\gamma }\\
            &\geq  \frac{C_5n}{1-\gamma }\\
            &\mathop{ \geq }\limits^{(iii)} 6n\log 2 \vee \log 2\\
            &\mathop{ > }\limits^{(iv)}    6n\log(1+\frac{2}{\delta _1})\vee \log 2\\
            &\mathop{ \geq }\limits^{(v)}   6\log M^\mathrm{ loc }_K (c\delta _1 ,2c)\vee \log 2,
        \end{align*}
        where $ (i) $ since $ \log (1+ax) \geq a\log (1+x) $ for any $ a\in (0,1) $ and $ x>0 $; 
        where $(ii)$ follows from the first component of \autoref{eq:condition_for_R}; $(iii)$ from the first component of \autoref{eq:condition_for_gamma}; $(iv)$ noting that by the fourth component of \autoref{eq:condition_for_R} we have $ R\geq R_0 \geq 2C+1 = c $, thus we have $ \delta _1 = \frac{4R}{2C+1} > 2 $;
        and $(v)$ from the volume argument in \cite{wainwright_high-dimensional_2019}[Example 5.8].
        \item[\textbf{B}] We have that
        \begin{align*}
            \frac{\delta _1^2}{\sigma ^2} &= \frac{d_m^2}{(C+1)^2\sigma ^2} \geq \frac{R^2}{(C+1)^2\sigma ^2} \geq \frac{1}{(C+1)^2} \log (1+\frac{R^2}{\sigma ^2}) \\
            &\mathop{ \geq }\limits^{(i)}   \frac{4n\log 5}{(C+1)^2(1-\gamma )}\geq  \frac{4\log 5}{(C+1)^2(1-\gamma )}\\
            &\mathop{ \geq }\limits^{(ii)}  C_0 \geq C_0N^{-1},
        \end{align*}
        where $(i)$ follows from the second component of \autoref{eq:condition_for_R} and $(ii)$ from the second component of \autoref{eq:condition_for_gamma}.
        \item[\textbf{C}] Adapting the argument from condition \textbf{B}, we have
        \begin{align*}
            \frac{\delta _1^2}{\sigma ^2} &\geq \frac{4\log 5}{(C+1)^2(1-\gamma )} \\
            &\mathop{ \geq }\limits^{(i)}  \frac{C_1^2}{2}\geq  C_1^2\varepsilon,
        \end{align*}
        where $(i)$ follows from the third component of \autoref{eq:condition_for_R} and $(ii)$ from the third component of \autoref{eq:condition_for_gamma}.  
    \end{itemize}
    
    With the above conditions, we can apply \autoref{appendix_lemma:total_expectation_for_risk_bound_unbounded} to obtain the desired risk upper bound that
    \begin{align*}
        \mathbbm{E}_{  }\left[ \left\Vert \nu ^{**}-\mu  \right\Vert ^2  \right] &\lesssim \max\{\delta _{J^*}^2 , \varepsilon \sigma ^2\} .
    \end{align*}

\end{proof}

\textbf{Remark:} We again remark that condition $C_8N>2$ can be satisfied by choosing $C$ to be sufficiently large, so that we can apply \autoref{appendix_lemma:total_expectation_for_risk_bound_unbounded}. Recall that in \autoref{appendix_lemma:tail_bound_for_interation_and_risk_bound_nonemptyS_unbounded} we have 
\begin{align*}
    \frac{1}{C_8}\geq \frac{2}{C_5}\vee \frac{16}{1/2-\varepsilon}\frac{1}{\gamma}
\end{align*}
where the $C_5$ part does not involve $C$ as we argued in the proof of \autoref{theorem:expected_risk_upper_bound}; for the $\frac{1}{\gamma}$ part we notice that when $C$ is chosen sufficiently large, $\gamma$ is made arbitrarily close to $1$ according to its definition \autoref{eq:condition_for_gamma}. Thus, the above reduces to $\frac{1}{C_8}\gtrsim \frac{1}{C_5}$, which holds with large enough $C$ chosen.

\subsection*{Proof of \autoref{theorem:expected_risk_rate_unbounded_case}}

The proof follows a similar manner to that of its bounded constraint counterpart. However, as mentioned previously, we need to handle the case where $ K $ is replaced by its closure $ \bar{K} $. In \autoref{lemma:closed_K}, we argue that the same rate holds with this replacement.

\begin{proof}[Proof of \textbf{\autoref{theorem:expected_risk_rate_unbounded_case}}.]

    We first address the edge case where $ \delta ^* = 0 $. In this scenario, $ K $ must be a singleton set, the algorithm output is simply the trivial point, and the minimax rate is $ 0 $, yielding a trivial result. Therefore, in what follows, we assume $ \delta ^* > 0 $. 
    
    We also observe that $ M^\mathrm{loc}(\delta ,c)$ can be made arbitrarily large for any $ \delta > 0 $ by choosing $ c $ sufficiently large in the case of unbounded star-shaped sets $ K $. The argument is as follows: for any $ \delta > 0 $, we can always find a line segment of length $ \delta $ in $ K $, and we can divide this line segment into arbitrarily many parts, each of length $ \delta /c $, provided $ c $ is large enough, thus forming a local packing. Given the definition of $ \delta ^* $ as a supremum in \autoref{equation:definition_of_delta_star}, we can always ensure $ N\frac{\delta ^{*2}}{\sigma ^2} > 8\log 2 $ if $ c $ is large enough. Therefore, we can focus on $ N\frac{\delta ^{*2}}{\sigma ^2}> 8\log 2 $, which follows a similar argument to the proof of \autoref{theorem:expected_risk_rate_bounded_case}[case II].

    \paragraph{Lower bound} By the same argument as in case \textbf{II}, we have
    \begin{align*}
        \log M^\mathrm{loc}(\frac{\delta ^*}{2},c) \geq 8\log 2 \vee N\frac{\delta ^{*2}}{\sigma ^2} \geq 4 \big( \frac{N(\delta ^*/2)^2}{\sigma ^2} \vee \log 2 \big).
    \end{align*}
    By \autoref{lemma:packing_lower_bound}, we have $ \mathfrak{M}\gtrsim \delta ^{*2} $. As argued in \autoref{theorem:expected_risk_upper_bound_unbounded}, we also have $ \mathfrak{M} \gtrsim \varepsilon \sigma ^2 $. Combining these results, we obtain
    \begin{align*}
        \mathfrak{M} \gtrsim \max\{\delta ^{*2},\varepsilon \sigma ^2\} .
    \end{align*}

    \paragraph{Upper bound} Following the same argument as in \autoref{theorem:expected_risk_upper_bound_unbounded}, we find some $ \tilde{\delta  } $ that satisfies $ \tilde{\delta }\asymp \delta ^* $ while $ \tilde{\delta }\gtrsim \delta _{J^*} $ where $ J^* $ is from the \autoref{eq:J_star_condition_unbounded} in \autoref{theorem:expected_risk_upper_bound_unbounded}. The process is exactly the same except a trivial constant difference in the definition of $ J^* $. The proof would give the desired upper bound of $ \mathfrak{M}\lesssim \max\{\delta ^{*2},\varepsilon \sigma ^2\}  $.
       
\end{proof}

\autoref{theorem:expected_risk_rate_unbounded_case} establishes the minimax rate when $ K $ is replaced by its closure $ \bar{K} $. The following lemma explains that the rate is preserved even without such replacement. For clarity of the statement, we recall that $\delta^*$ in the above proof of \autoref{theorem:expected_risk_rate_unbounded_case} should now be re-written as
\begin{align*}
    \delta ^*&:= \sup\big\{ \delta \geq 0 : N\frac{\delta ^2}{\sigma ^2}\leq \log M^\mathrm{ loc }_{\bar{K}}(\delta ,c)  \big\},
\end{align*}
while its correspondence in terms of a possibly non-closed $K$ is defined as
\begin{align*}
    \tilde{\delta }&:=  \sup\big\{ \delta \geq 0 : N\frac{\delta ^2}{\sigma ^2}\leq \log M^\mathrm{ loc }_K(\delta ,c)  \big\}.
\end{align*}

\begin{lemma}\label{lemma:closed_K}
    Let $ \tilde{\delta } $ as defined above. Then we have 
    \begin{align*}
        \mathfrak{M} \asymp \max\big( \tilde{\delta }^{2}, \varepsilon \sigma ^2 \big) .
    \end{align*}
\end{lemma}

\begin{proof}
    The lower bound established in \autoref{lemma:packing_lower_bound} and \autoref{lemma:diakonikolas_lower_bound_unbounded} still holds. Therefore, we only need to address the upper bound. By \cite[Remark 5.13]{prasadan_information_2024}, we have the relation
    \begin{align*}
        \log M^\mathrm{ loc }_K(\delta ,c) \leq \log M^\mathrm{ loc }_{\bar{K}}(\delta ,c) \leq \log M^\mathrm{ loc }_{K}(\delta \frac{2c+2}{2c+1} ,2c).
    \end{align*}
    We further define $ \hat{\delta } $ by $ \hat{\delta }:= \sup\big\{ \delta \geq 0 : N\frac{\delta ^2}{\sigma ^2}\leq \log M^\mathrm{ loc }_K(\delta \frac{2c+2}{2c+1} ,2c)  \big\} $. By monotonicity, we have $ \delta ^*\leq \hat{\delta } $. Moreover, by \cite[Lemma 2]{prasadan_facts_2025}, we have $ \hat{\delta }\asymp \tilde{\delta } $ (the $ 2c $ term can be canceled using the same argument as \autoref{remark:J_star_condition_unbounded_modified}). Therefore, we have $\delta ^*\lesssim \tilde{\delta }$, and consequently
    \begin{align*}
        \mathfrak{M} = \mathfrak{M}(K) \leq \mathfrak{M}(\bar{K}) \lesssim \max\big( \delta ^{*2}, \varepsilon \sigma ^2 \big)\lesssim \max\big( \tilde{\delta }^{2}, \varepsilon \sigma ^2 \big).
    \end{align*}   
    Combining with the lower bound, we conclude the proof.
\end{proof}

\section{Proof for Symmetric or Known Distribution Case}\label{appendix:proof_for_symmetric_or_known_distribution_case}

\begin{proof}[Proof of \textbf{\autoref{theorem:novikov_prob_bound}}.]

    By the same argument as in the proof of \autoref{theorem:two_point_prob_bound}, it suffices to prove the statement with $ \sigma  $ replaced by $ \sigma _V $, since we have a direct relation between them. 

    To apply \cite[Theorem 1.5]{novikov_robust_2023}, we establish the following notation: the confidence level $ \delta _0 := \exp\left[ - C_{12}N\frac{\delta ^2}{\sigma _V^2} \right] $; the constant term in the big-O symbol $ Q_1 $, that is, with probability at least $ 1-\delta _0 $, we have
    \begin{align*}
        \left\Vert \hat{\mu }_{\mathrm{ Novikov }(\{V_i\}_1^N) - m } \right\Vert  \leq Q_1\big( \rho\cdot \big[ \sqrt{\dfrac{ n+\log(1/\delta _0) }{ N } } + \varepsilon  \big] \big)
    \end{align*}
    where dimension $ n=1 $ since the discriminant quantities $ V_i $s are one-dimensional as defined in \autoref{equation:definition_V_i}. We also remind that the $ Q_1 $ depends only on $ Q $\footnote{
        We refer readers to \cite[Appendix F, Theorem E.3, then Theorem C.1]{novikov_robust_2023} for the source of the constants.
    }; We choose $ \rho = \frac{10}{\sqrt{99}}\times \sigma _V $ so that by Chebyshev's inequality, we have $ \mathbbm{P}_{  }\left( \left\vert \tilde V_i-\mathbb{E}[\tilde V_i]  \right\vert \leq \rho \right) \geq 1/100   $ as desired. The absolute constants will later be absorbed into $ Q_1 $. By assumption, there are constants $C_{10}, C_{11}$ such that $\frac{\delta ^2}{\sigma_V^2} \geq C_{10}N^{-1}\vee C_{11}^2\varepsilon ^2$.
    
    With the above definitions, we apply the theorem and obtain that with probability at least $ 1-\delta _0 = 1-\exp\left[ - C_{12}N\frac{\delta ^2}{\sigma _V^2} \right] $,
    \begin{align*}
        \left\vert \hat{\mu }_\mathrm{ Novikov }\big( \{V_i\}_1^N \big) - m \right\vert &\mathop{ \leq }\limits^{(i)}  Q_1\sigma _V\Big( \sqrt{\dfrac{ 1 + C_{12}N\frac{\delta ^2}{\sigma _V^2} }{ N } } + \varepsilon  \Big)\\
        &\mathop{ \leq }\limits^{(ii)} Q_1\delta \Big( \sqrt{\dfrac{ 1 }{ C_{10} } + C_{12} } + \frac{1}{C_{11}} \Big)  \\
        &\mathop{ \leq }\limits^{(iii)}  (C-2)\delta ,
    \end{align*}
    where $ (i) $ follows from \cite[Theorem 1.5]{novikov_robust_2023}; $ (ii) $ follows from the condition $ C_{10}N^{-1} \vee C_{11}^2\varepsilon ^2 \leq \frac{\delta ^2}{\sigma _V^2} $; $ (iii) $ holds as long as we choose $ C $ large enough such that $ C-2 \geq Q_1 \Big( \sqrt{\frac{ 1 }{ C_{10} } + C_{12} } + \frac{1}{C_{11}} \Big) $. Then by the same argument as in \autoref{equation:reduce_from_mean_to_zero} we have the desired probabilistic bound. 
    
\end{proof}

\section{Proof for Examples}\label{appendix:proof_for_examples}

\begin{proof}[Proof of \textbf{\autoref{lemma:ell_1_ball_minimax_rate}}.]

    We first solve the condition for $ \delta ^{*2} $ given by $ N\delta ^{*2} \asymp \log M^\mathrm{ loc }_{B_1(1)}(\delta^* ,c)  $ (recall that we set $\sigma = 1$), which yields
    \begin{align*}
        \begin{cases}
            \frac{ \log(\delta^{*2}n) }{ \delta^{*4}n^2 }&\mathop{ \asymp }\limits^{(i)} \frac{N}{n^2},  \delta ^* \gnsim 1/\sqrt{n},\\
            \delta ^{*2} &\mathop{ \asymp }\limits^{(ii)}  \frac{ n }{ N },  \delta ^* \lesssim 1/\sqrt{n}.
        \end{cases}  
    \end{align*}

    We discuss the solution depending on the value of $ N $:

    \begin{itemize}[topsep=2pt,itemsep=0pt]
        \item If $ N\gtrsim n^2 $: For a solution falling into case $ (i) $ we should have $ 1 \gnsim \frac{\log (\delta^{* 2}n)}{\delta^{* 4}n^2} \asymp \frac{N}{n^2} \gtrsim 1 $, yielding a contradiction. Thus the solution of $ \delta ^* $ falls into case $ (ii) $, which gives $ \delta ^{*2} \asymp \frac{ n }{ N } $. We note that the solution is compatible with the condition of $ \delta^*\lesssim 1/\sqrt{n} $ since $ N\gtrsim n^2 $. 

        Thus, the minimax rate is $ \max\big( \frac{n}{N}, \varepsilon  \big)\wedge 1 $. This recovers the LSE rate for the low-dimensional setting. 
        \item If $ n\lesssim N\lnsim  n^2 $: For a solution falling into case $ (ii) $ we should have $ \delta ^* \asymp \sqrt{\frac{n}{N}} \gnsim \frac{1}{\sqrt{n}} $, which contradicts with the condition of case $ (ii) $. Thus the solution has to fall into case $ (i) $, which gives 

        \begin{align*}
            \delta ^{*2} \asymp \sqrt{\dfrac{ \log(n/\sqrt{N}) }{ N } }  ,
        \end{align*}
        matching the result in \cite[Corollary 4.16]{rigollet_high-dimensional_2023}
        \footnote{
            There is a typo in their textbook, with a square root missing.
        }.
        We can also check that the condition of $ \delta^* \gnsim 1/\sqrt{n} $ is compatible. Thus the minimax rate is $ \max\big(  \sqrt{\frac{\log (n/\sqrt{N})}{N}}, \varepsilon  \big) \wedge 1 $.
    \end{itemize}

\end{proof}

\end{document}